\documentclass[12pt]{amsart}
\usepackage{amssymb, a4wide, mathdots, url,hyperref,graphicx}
\usepackage[usenames]{color}
\usepackage{enumerate}
\usepackage{enumitem}

\newcommand{\abs}[1]{\left|{#1}\right|}

\newcommand{\Cu}{\mathcal{C}}

\newcommand{\GL}{\operatorname{GL}}
\newcommand{\PGL}{\operatorname{PGL}}
\newcommand{\SL}{\operatorname{SL}}

\newcommand{\Sp}{\operatorname{Sp}}
\newcommand{\SO}{\operatorname{SO}}
\newcommand{\Mp}{\operatorname{Mp}}
\newcommand{\Ad}{\operatorname{Ad}}
\newcommand{\swrz}{\mathcal{S}}

\newcommand{\cs}{\mathfrak{c}}

\newcommand{\Sh}{\operatorname{Sh}}

\newcommand{\A}{\mathbb{A}}
\newcommand{\Z}{\mathbb{Z}}
\newcommand{\C}{\mathbb{C}}
\newcommand{\R}{\mathbb{R}}

\newcommand{\bs}{\backslash}

\newcommand{\diag}{\operatorname{diag}}

\newcommand{\Whit}{\mathcal{W}}

\newcommand{\triv}{{\bf{1}}}

\newcommand{\I}{\mathcal{I}}
\newcommand{\J}{\mathcal{J}}

\renewcommand{\O}{\mathcal{O}}
\newcommand{\sm}[4]{\left(\begin{smallmatrix}{#1}&{#2}\\{#3}&{#4}\end{smallmatrix}\right)}

\renewcommand{\a}{\breve{a}}
\renewcommand{\b}{\breve{b}}

\renewcommand{\sec}{{\mathbf{s}}}
\newcommand{\inj}{\iota}

\renewcommand{\H}{\mathcal{H}}

\newcommand{\p}{\mathfrak{p}}

\newcommand{\rel}{{\operatorname{rel}}}

\newcommand{\val}{{\operatorname{val}}}
\newcommand{\g}{\mathfrak{g}}

\renewcommand{\i}{\varsigma}
\newcommand{\K}{{K_*}}
\newcommand{\RI}{\mathcal{R}}
\newcommand{\D}{\mathcal{D}}
\newcommand{\E}{\mathcal{E}}
\newcommand{\Kl}{\mathcal{K}}

\newtheorem{theorem}{Theorem}[section]
\newtheorem{lemma}[theorem]{Lemma}

\newtheorem{proposition}[theorem]{Proposition}%[subsection]
\newtheorem{corollary}[theorem]{Corollary}%[subsection]
\theoremstyle{remark}
\newtheorem{remark}[theorem]{Remark}%[subsection]
\begin{document}

\title{On the cubic Shimura lift to PGL(3): Hecke correspondences}

\author{Solomon Friedberg}
\address{Department of Mathematics, Boston College, Chestnut Hill, MA 02467}
\email{solomon.friedberg@bc.edu}

\author{Omer Offen}
\address{Department of Mathematics, Brandeis University, Waltham MA 02453}
\email{offen@brandeis.edu}

\thanks{This work was supported by the NSF, grant numbers DMS-2401309 and DMS-2100206 (Friedberg) and DMS-2401308 (Offen).}

\date{}


\begin{abstract}  In this paper we establish a new Fundamental Lemma for Hecke correspondences. Let $F$ be a local field containing the cube roots of unity.  
We exhibit an algebra isomorphism of the spherical Hecke algebra of $\PGL_3(F)$ and the spherical Hecke algebra of anti-genuine functions on
the cubic cover $G'$ of $\SL_3(F)$.  Then we show that there is a 
matching (up to a specific transfer factor)  of distributions on the two groups for all functions that correspond under this isomorphism.
On $\PGL_3(F)$ the distributions are relative distributions attached to a period involving the minimal representation on $\SO_8$,
while on $G'$ they are metaplectic Kuznetsov distributions.  This Fundamental Lemma  is a key step towards 
establishing a relative trace formula
that would give a new global Shimura lift from genuine automorphic representations on the triple cover of $\SL_3$ to automorphic representations on $\PGL_3$, and also
 characterize the image of the lift by means of a period.  It extends the matching for the unit elements of the Hecke algebras established by the authors in prior work.
 \end{abstract}

\maketitle

\tableofcontents

\section{Introduction}
In this paper we establish a new Fundamental Lemma matching distributions on two different groups. 
 Let $F$ be a local field containing the cube roots of unity.  Then
we exhibit an algebra isomorphism between the spherical Hecke algebra of $\PGL_3(F)$ and the spherical Hecke algebra of anti-genuine functions on
the cubic cover $G'$ of $\SL_3(F)$.  We show that there is a matching (up to a specific transfer factor)  of distributions on the two groups for all functions that correspond under this isomorphism.
On $\PGL_3(F)$ the distributions are relative distributions attached to a period involving the minimal representation on $\SO_8$ and to a Whittaker functional,
while on $G'$ they are metaplectic Kuznetsov distributions.  

Our result is purely local,
but it is motivated by a global problem: generalizing the Shimura correspondence.  For any fixed number field $L$ 
with ring of adeles $\A$, the Shimura map lifts genuine automorphic 
representations on the double cover of $\SL_2(\A)$ to automorphic representations of $\PGL_2(\A)$.  A broad generalization is expected: for any split connected
reductive algebraic group $H_1$ defined over $L$, where $L$ contains a full set of $n$-th roots of unity, there should be a lifting of genuine automorphic
representations of the $n$-fold cover of $H_1(\A)$ to automorphic representations on the adelic points of a suitable reductive group $H_2$. 
(If $H_1$ is not connected, there is a similar conjecture but more than one $n$-fold cover.)  See for example Kazhdan and Patterson \cite{MR0840303} and Flicker and Kazhdan \cite{MR876160}.
This expectation is comparable to
Langlands functoriality in both its formulation and, it appears, its difficulty. Indeed, the unramified local map has been known for some time \cite{MR876160} 
(here it is taking $n$-th powers on the Satake parameters); for a general treatment, see Savin \cite{MR929534}. 
However there are only a few cases where a corresponding global lifting is known.
For an $n$-fold cover of $\GL_2(\A)$, Flicker used the trace formula to establish such a map \cite{MR567194},
while for the $3$-fold cover of $\SL_2(\A)$, Ginzburg, Rallis and Soudry \cite{MR1439552} 
used a theta correspondence for the exceptional group $G_2(\A)$ to prove its existence.
They also showed that the image of the cubic rank one lift could be detected by a non-vanishing period.    Some additional cases may be handled by
the theta correspondence, but notably, there is no case
of higher rank and cover of degree larger than 2 for which the Shimura correspondence is known.  Our result here is motivated by this problem.

The use of the Arthur-Selberg trace formula to establish the generalized Shimura correspondence has been pursued since the 1980s (\cite{MR0840303}, \cite{MR876160}) but seems to be difficult.
Here we approach the problem by a related method, a comparison of relative trace formulas.  This approach is available when one has a period that is
expected to be non-zero exactly on the image of the lift.  Here by a {\it period} of an irreducible automorphic representation $\pi$ on a group $G(\A)$ we mean an integral
of the form
$$\int_{H(L)\backslash H(\A)} \varphi(h) \theta(i(h))\,dh,\qquad \varphi\in\pi$$
where $H$ is an algebraic subgroup of $G$, $i:H\to H'$ is a morphism to another reductive group $H'$, and $\theta$ is an automorphic form on $H'(\A)$.
The key step in the comparison is a matching of geometric distributions, that is, an equality of orbital integrals on the two groups, one arising from the period, up to a specific transfer factor. 
Once established for enough test functions, the geometric matching may be combined with the spectral expansions on the two groups to prove a lifting.

Jacquet developed the relative trace formula and used 
it \cite{MR0983610} to give an alternative proof of the original Shimura lift.  For the 3-fold cover of $\SL_2$,
Mao and Rallis  \cite{MR1729444} established such a relative trace formula, using the period identified by Ginzburg, Rallis and Soudry.   Jacquet \cite{MR1113084}
conjectured a period for the double cover of $\GL_n$, and Mao for $\GL_3$ \cite{MR1605813} and Do for arbitrary $n$ \cite{MR3323346}, \cite{MR4176834}
have established parts of the geometric matching for this case.
However, for higher rank groups {\it and} higher degree covers, 
there is only one other case where a period is even conjectured: the conjectural Shimura  lift from the 3-fold
cover of $\SL_3$ to $\PGL_3$.  Here Bump, Friedberg and Ginzburg \cite{MR1882041} conjectured that the expected lift is detected by a period that
involves an automorphic minimal representation on the split special orthogonal group $\SO_8(\A)$ constructed in \cite{MR1469105};  the group $H$ above is $\PGL_3$ and the map $i$ above is the Adjoint map. 
This conjecture opens up the possibility of proving the
conjectured Shimura correspondence via the relative
trace formula and also showing that the image of the lift is characterized by the non-vanishing of this period. We emphasize that this lifting is not expected to be realized by means of a theta correspondence.

In this work we establish the Fundamental Lemma for such a relative trace formula for the full spherical Hecke algebras.  The matching for the unit elements of these algebras was
established in prior work of the authors \cite{FO}.   Two key ingredients in that work play a role here. First,  the minimal representation of $\SO_8$ is realized as a theta lift of the trivial
representation of $\SL_2$, following Deitmar and Krieg \cite{MR1289491} and Kudla and Rallis \cite{MR1289491}.  The action of the Weil representation enters (in an additional way) in 
computing the relative orbital integrals for the full Hecke algebra.
Second, the comparison of orbital integrals requires a number theoretic fact:
Kloosterman sums involving a cubic character are equal to certain cubic exponential sums.  This was established for sums ``modulo p" (i.e., the first power of the local uniformizer)
by Duke and Iwaniec \cite{MR1210520} using the Hasse-Davenport relation, and extended to arbitrary level using the method of stationary phase by the authors (\cite{FO},Corollary 7.5). 
We make use of this equality at numerous points in our computations (once again, it enters in new ways here).   

To go from the unit elements to the full spherical Hecke algebras we establish the equality on a natural basis for the spherical Hecke algebra.
To do so, we introduce several new ingredients.  First, decomposing the double cosets into left cosets, we give
expansions of the two classes of orbital integrals as linear combinations of orbital integrals for translates of the unit  element
in the Hecke algebra by a diagonal matrix (on the relative side, via the action of the Weil  representation).  See Corollaries~\ref{cor exp j-G} and \ref{cor exp j}.
These expansions are based on the Iwasawa decompositions.  We arrive at linear combinations whose coefficients are the Whittaker coefficients
of the two bases. The coefficients are then computed.  This requires working with the Cartan decompositions for the groups we are comparing, and on the metaplectic group the cocycle 
and a splitting over a maximal compact subgroup (\cite{FO}, Lemma 8.7) play a role.   

We then give each orbital integral for translates of the unit  element
in the Hecke algebra by a diagonal matrix explicitly and compute them by dividing the complicated domains of integration
into pieces on which the integrands are well behaved.  For the relative side, cubic exponential integrals arise and we make extensive use
of their properties.  We obtain expressions that are difficult for combinatorial reasons. We are able to evaluate them by introducing recursive methods; however, the recursions are complicated as they begin from different values for different
base cases and involve sums of cubic exponential integrals. 
Using number theoretic properties, we are able to give a complete evaluation. 
We show that the orbital integral for the largest relevant orbit (given in Proposition~\ref{Final-form-I-OI}), is non-zero in 24 different cases.  In these cases it is given by a monomial in $q$ 
(the size of the residual field), a monomial in $q$ times a sum of Gauss sums,
or a monomial in $q^{\pm 1}$ times a sum of cubic exponential sums.  

Next we turn to the 3-fold cover of $\SL_3$,
where the relevant orbits arise from relevant Bruhat cells.
Once again the initial domain of integration is divided into subdomains over which one may carry out the integration.
Because of the Hilbert symbols that come from the covering group, we ultimately obtain expressions in terms of cubic Kloosterman sums.
We use the properties of these sums to evaluate all terms, (remarkably!) once again obtaining 24 different non-zero cases.  Using the number theoretic ingredients noted above, the 
evaluations match for all basis elements in correspondence by the Hecke algebra isomorphism and for all relevant orbits. The Fundamental Lemma follows.

We remark that a general theory of periods for automorphic forms and relative Langlands duality has been formulated by Ben-Zvi, Sakellaridis and Venkatesh \cite{BZSV}, following
earlier work on periods and spherical varieties by Sakellaridis and Venkatesh \cite{MR3764130}.  Periods attached to strongly tempered hyperspherical
Hamiltonian spaces have been studied  by Wan and Zhang \cite{WZ} and Mao, Wan and Zhang \cite{MWZ}.
Though it is natural to hope that this framework can be extended to metaplectic groups,  we do not at present know of a theory of automorphic periods that includes
covering groups or that would predict the period that underlies this work. 
For periods that are known, one could hope to establish a comparison via the relative trace formula.  This work exhibits both the complications and their resolution when
the dimension of the generic relevant orbit is greater than one.  In fact, we believe that this is the first direct proof based on explicit computations
of a Fundamental Lemma for the full spherical Hecke algebras for a relative rank two case. (For a direct proof of another general rank situation, i.e., not going through function fields 
and then using model theory to transfer to characteristic zero see \cite{MR2172953}).

This paper is organized as follows.  In Section~\ref{sec: Notation and so on} we introduce the groups and the orbital integrals of concern, describe the relevant orbits for each,
and exhibit the matching established in \cite{FO}, Sections 4 and 5.  Then in Section~\ref{sec: Statement of the FL} we introduce the two spherical Hecke algebras, give natural bases for them and state the Fundamental Lemma,
Theorem~\ref{thm main local}.
Section~\ref{ss prem} provides number theoretic preliminaries concerning cubic Hilbert symbols, Gauss sums, Kloosterman sums with cubic characters, cubic exponential sums (written
as $p$-adic integrals), and also the relation between Kloosterman sums with cubic characters and cubic exponential sums.

The remainder of this paper is concerned with the evaluation of the orbital integrals. In Section~\ref{sec: Expansions of the OIs} we give
an expansion of each orbital integral as a linear combination of orbital integrals for translates of the unit  element
in the Hecke algebra by a diagonal matrix.  We also introduce an involution of one of the Hecke algebras that gives a useful symmetry for the orbital integrals.  
Then in Section~\ref{sec I comp} we compute the relative orbital integral on $\PGL_3$
for the two-parameter family.   In Section~\ref{sec I comp - small} we compute the relative orbital integrals attached to the other relevant orbits.  Once again, the use
of recursive methods is critical in the evaluation.  Next we turn to the metaplectic orbital integrals.
In Section~\ref{sec J comp} we compute the big cell orbital integral for the 3-fold cover of $\SL_3$.  Section~\ref{sec J comp - small} carries out the evaluation for the smaller relevant Bruhat cells.
Finally in Section~\ref{sec compare} we compare all the expressions and prove Theorem~\ref{thm main local}.

\section{Notation, set up and the orbital integrals}\label{sec: Notation and so on}

\subsection{Notation}
Let $F$ be a nonarchimedean local field of residual characteristic greater than 3. In this work we assume that $F$ contains a cube root of unity $\rho$ and write $\mu_3=\mu_3(F)=\langle \rho \rangle$. 
Fix once and for all an identification of $\mu_3(F)$ with $\mu_3(\C)$. 

We introduce the following standard notation: 
\begin{itemize}
\item $\O$ is the ring of integers of $F$; 
\item $\O^*$ its multiplicative group;
\item $\p$ the maximal ideal of $\O$;  
\item $q$ the size of the residual field $\O/\p$;
\item $\psi$ and additive character on $F$ with conductor $\O$;
\item $\varpi$ a choice of uniformizer in $\O$ (so that $\p=\varpi\O$);
\item $\val$ the standard discrete valuation and $\abs{\cdot}$ the standard absolute value on $F$ so that 
$$
\val(u\varpi^n)=n\ \ \  \text{and}\ \ \ \abs{u\varpi^n}=q^{-n}, \ \ \ u\in \O^*, \ n\in \Z;
$$
\item $(\cdot,\cdot):F^\times \times F^\times\rightarrow\mu_3$ the cubic Hilbert symbol.
\end{itemize}

For a group $X$, we denote by $L$ (resp. $R$) the action of $X$ by left (resp. right) translations on functions $f:X\rightarrow \C$. That is,
\[
L(x)f(y)=f(x^{-1}y) \ \ \ \text{and}\ \ \ R(x)f(y)=f(yx),\ \ \ x,y\in X.
\]
We denote by $\delta_Q$ the modulus character of a locally compact topological group $Q$.
\subsection{The group $\PGL_3(F)$}
Let $G=\PGL_3(F)=\GL_3(F)/Z$ where $Z$ is the center of $\GL_3(F)$.  Throughout the paper,
for $g\in\GL_3(F)$, we sometimes write $g$ instead of $gZ$ for the projection of $g$ to $G$. 
Let $\swrz(G)$ be the space of Schwartz-Bruhat functions on $G$.
We identify $\swrz(G)$ with the space of smooth and $Z$-invariant functions on $\GL_3(F)$ of compact support mod $Z$.

Let $N$ be the group of upper triangular unipotent matrices in $\GL_3(F)$. 
The group $N$ embeds in $G$ via $u\mapsto uZ$ ($u\in N$) and we consider it as a subgroup of $G$.
Henceforth, for an element $u\in N$ we assign the coordinates
\begin{equation}\label{coords-u}
u=\begin{pmatrix} 1 & x & z \\ & 1 & y  \\ & & 1 \end{pmatrix}.
\end{equation}
Let $\theta$ be the generic character on $N$ given by
$
\theta(u)=\psi(x+y).
$
Let $K=\GL_3(\O)$ so that $\K:=KZ/Z$ is a maximal compact subgroup of $G$. We denote by $B$ the standard Borel subgroup of $G$ with unipotent radical $N$.

\subsection{The cubic metaplectic cover of $\SL_3(F)$}
Let $G'$ be the cubic metaplectic cover of $\SL_3(F)$. It is a central extension of $\SL_3(F)$ by $\mu_3$ and there is an exact sequence
\[
1\longrightarrow \mu_3\longrightarrow G' \longrightarrow \SL_3(F)\longrightarrow 1.
\]
We may describe $G'$ using a 2-cocycle on $\SL_3(F)$ with values in $\mu_3$. In fact, it will be more convenient to consider a 2-cocycle on $\GL_3(F)$. Let $\sigma$ be the block compatible 2-cocycle on $\GL_3(F)$ defined in \cite{MR1670203}. Form the cubic metaplectic cover $G''$ of $\GL_3(F)$ whose elements are the pairs $(g,z)\in\GL_3(F)\times \mu_3$ with group operation
\[
(g_1,z_1)(g_2,z_2)=(g_1g_2,z_1z_2\sigma(g_1,g_2)).
\]
Then $G'$ is the subgroup of $G''$ with underlying set $\SL_3(F)\times \mu_3$. We view $\mu_3$ as a central subgroup of $G'$ via $z\mapsto (e,z)$. (Throughout, we denote by $e$ the identity element of a group.)

Let $\swrz(G')$ be the space of anti-genuine Schwartz-Bruhat functions on $G'$, i.e., the space of smooth functions $f':G' \rightarrow \C$ of compact support such that
\[
f(zg)=z^{-1}f(g),\ \ \ z\in \mu_3,\ g\in G'.
\]

Let $\sec(g)=(g,1)$, $g\in \GL_3(F)$, be the trivial section. This section splits over $N$, and so 
we consider $N$ as a subgroup of $G'$. We often write $u$ instead of $\sec(u)$ for $u\in N$. We regard $\theta$ as a character of the subgroup $N$ of $G'$ via this identification.

Let $K=\GL_3(\O)$ and $K_1=\SL_3(\O)$.
Since the residual characteristic of $F$ is not $3$, the group $K$ splits over $G''$ (\cite{MR244258});
that is, there is a map $\kappa: K\to \mu_3$ satisfying 
\[
\kappa(g_1g_2)=\kappa(g_1)\kappa(g_2)\sigma(g_1,g_2), \ \ \ g_1,g_2\in K.
\] 
The splitting $\kappa$ is not canonical; we choose the one described in \cite[Section 8.3]{FO}.  The map $g\mapsto (g,\kappa(g))$ from $K\rightarrow G''$ is a group embedding. 
Let $K''$ denote the image of $K$ in $G''$ and let $K'$ denote the image of $K_1$ in $G'$. The group $K'\mu_3$ is a maximal compact subgroup of $G'$.

Let $T_1$ be the diagonal torus in $\SL_3(F)$, $T'=T_1\times \mu_3$ its pre-image in $G'$ and $B'=T' N$. 
Note that $\sec(t) \sec(u)\sec(t)^{-1}=\sec(tut^{-1})$, $t\in T_1$, $u\in N$.
Henceforth we use the fact that $N$ is normal in $B'$ without further comment.

\subsection{The relative orbital integrals for $G$}
Let $H=\SL_2(F)\times N$ and consider the right action of $H$ on $F^8$ defined by
\begin{multline}\label{eq action}
\xi\cdot (g,u)=(\xi_1 g,(\xi_2+x\xi_1)g, (\xi_3-y\xi_1)g,(\xi_4+x\xi_3+y\rho^2 \xi_2-(yx+\rho z)\xi_1)g), \\ \xi=(\xi_1,\xi_2,\xi_3,\xi_4)\in F^8 \ \ \ \text{with}\ \ \  \xi_i\in F^2, i=1,2,3,4,\ g\in \SL_2(F),\ u\in N.
\end{multline}
Denote by $H_\xi$ the stabilizer of $\xi\in F^8$ in $H$.
Let
\begin{multline*}
\cs(\xi,h)=\psi \left( 2\sum_{1\le i<j\le 4}a_{ij}(u)\xi_i\sm{0}{1}{-1}{0}{}^t\xi_j \right),\ \ \  \xi\in F^8, \ h=(g,u),\
g\in \SL_2(F),\ u\in N \\ \text{where}\ \ \ \ (a_{i,j}(u))=\begin{pmatrix} 0 & (z-xy)y\rho^2 & \rho xz & -(xy+z\rho^2) \\ (z-xy)y\rho^2 & 0 & -(xy\rho^2 +z)  & \rho y \\ \rho x z& -(xy\rho^2 +z) & 0 & x \\  -(xy+z\rho^2) & \rho y & x& 0\end{pmatrix}.
\end{multline*}
We have the cocycle relation
\begin{equation}\label{eq cocyc}
\cs(\xi,h_1h_2)=\cs(\xi,h_1)\cs(\xi\cdot h_1,h_2),\ \ \ \xi\in F^8,\ h_1,h_2\in H.
\end{equation}
By abuse of notation we also denote by $\theta$ the character of $H$ defined by $\theta(g,u)=\theta(u)$, $g\in \SL_2(F)$, $u\in N$.
An 8-tuple $\xi\in F^8$ is called {\it relevant} if 
\[
\cs(\xi,h)\theta(h)=1\quad \text{for~all~}h\in H_\xi.
\]
Denote by $\Xi_\rel$ the set of relevant elements in $F^8$. We define the relative orbital integrals
\[
\O(\xi,\phi)=\int_{H_\xi\bs H}\phi(\xi \cdot h)\cs(\xi,h)\theta(h)\ dh,\ \ \ \phi\in \swrz(F^8),\ \xi\in \Xi_\rel.
\]

This description of the relative orbital integrals is established in \cite{FO}.  To recount briefly, it has its roots in an embedding of $G$ in the split special orthogonal group
$\SO_8(F)$, and the Weil representation of $\Mp_{16}$, the metaplectic double cover of $\Sp_{16}(F)$.   
The adjoint representation, realized as in \cite[\S3.1]{FO}, gives an embedding $\Ad:G\hookrightarrow \SO_8(F)$.
Let $\i: \SL_2(F)\times \SO_8(F) \rightarrow \Mp_{16}$ be the embedding defined in \cite[\S4.1]{FO} and let $\omega_\psi$ be the Weil representation of $\Mp_{16}$ on the space $\swrz(F^8)$ of Schwartz-Bruhat functions on $F^8$.
We denote by $\omega$ the restriction of the Weil representation to $\SL_2(F)\times G$ via the above embeddings, i.e.,
\[
\omega(g,x)\phi=\omega_\psi(\i(g,\Ad(x))\phi, \ \ \ g\in \SL_2(F),\ x\in G,\ \phi\in \swrz(F^8).
\]
Then
\[
\omega(h)\phi(\xi)=\cs(\xi,h)\phi(\xi \cdot h),\ \ \ \xi\in F^8,\ h\in H\text{ and }\phi\in\swrz(F^8)
\]
so that
\[
\O(\xi,\phi)=\int_{H_\xi\bs H}\omega(h)\phi(\xi)\theta(h)\ dh,\ \ \ \phi\in \swrz(F^8),\ \xi\in \Xi_\rel.
\]
If $g$ is in $\SL_2(F)$ (resp. $G$), by abuse of notation  we write $\omega(g)$ for $\omega(g,e)$ (resp. $\omega(e,g)$).

It easily follows from \eqref{eq cocyc} that an element $\xi\in F^8$ is relevant if and only if its entire $H$-orbit $\xi\cdot H$ consists of relevant elements
and
\begin{equation}\label{eq orb trans}
\cs(\xi,h)\theta(h)\O(\xi\cdot h,\phi)=\O(\xi,\phi),\ \ \ \xi\in F^8,\ h\in H,\ \phi\in\swrz(F^8).
\end{equation}
Finally, $\swrz(F^8)$ has a natural $\swrz(G)$-module structure
\[
\omega(f)\phi(\xi)=\int_G f(g)\omega(g)\phi(\xi)\ dg.
\]
The local factors appearing in the geometric expansion of the relative trace formula for $G$ are $\O(\xi,\omega(f)\phi)$ for $f\in \swrz(G)$ and $\phi\in \swrz(F^8)$.
Our matching will be established when $\phi=\phi_\circ$, the characteristic function of $\O^8$.

\subsection{The orbital integrals for $G'$}

For $x\in G'$, let $(N\times N)_x$ be the stabilizer of $x$ in $N\times N$ with respect to the action $x\cdot (u_1,u_2)=u_1^{-1}xu_2$. We say that $x$ is {\it relevant} if $\theta(u_1^{-1}u_2)=1$ whenever $(u_1,u_2)\in (N\times N)_x$, and denote by $\Xi_\rel'$ the set of relevant elements in $G'$. 
It is easy to see that $x$ is relevant if and only if any element of $\mu_3 NxN$ is relevant.

Consider the orbital integrals
\[
\O'(x,f')=\int_{(N\times N)_x\bs N\times N}f'(u_1^{-1}x u_2)\theta(u_1^{-1}u_2)\ d(u_1,u_2),\ \ \ x\in \Xi_\rel',\ f'\in \swrz(G')
\]
for a right $N\times N$-invariant measure on the quotient space. The exact normalization of the measure is discussed after explicating representatives of orbits of relevant elements.
The variable change $u_1\mapsto u_1u^{-1}$ gives
\begin{equation}\label{eq Nequiv}
\O'(x,L(u)f')=\theta(u)\O'(x,f'),\quad u\in N,\, x\in \Xi_\rel',\ f'\in \swrz(G').
\end{equation}

\subsection{Matching of relevant orbits}\label{ss orbreps}
The set $\Xi_\rel$ of relevant elements in $F^8$ is analyzed in \cite[\S4.2]{FO}. A complete set of representatives for the $H$-orbits in $\Xi_\rel$ is given by:
\begin{enumerate}
\item the generic family $\Xi_\circ=\{\xi(a,b):a,b\in F^*\}$ where $\xi(a,b)=(0,1,a,0,b,0,0,0)$,
\item the two one parameter families $\Xi_i=\{\xi_i(a): a\in F^*\}$, $i=1,2$ where 
\[
\xi_1(a)=(0,1,0,a,-\frac12 \rho^2a^{-2},0,0,0)\ \ \ \text{and}\ \ \ \xi_2(a)=(0,1,\frac12 \rho a^{-2}, 0,0,a,0,0)
\]
\item and the three isolated orbits
\[
\xi[1]=(0,0,0,1,0,\rho^2,0,0),\ \xi[2]=(0,0,0,1,0,\rho,\frac12\rho^2,0)\ \ \ \text{and}\ \ \  \xi[3]=(0,1,0,0,0,0,0,0,0).
\]
\end{enumerate}
The set $\Xi_\rel'$ of relevant elements in $G'$ is well known. A complete set of representatives for the relevant $\mu_3 (N\times N)$-orbits  in $\Xi_\rel'$ is given by:
\begin{enumerate}
\item the generic family $\Xi_\circ'=\{\gamma(a,b):a,b\in F^*\}$ where 
\[
\gamma(a,b)=\sec(g_{a,b})\quad\text{~where~}\quad
g_{a,b}=\begin{pmatrix}  &  & b^{-1} \\  & -a^{-1}b &  \\ a &  & \end{pmatrix},
\]
arising from the big Bruhat cell,
\item the two one parameter families $\Xi_i'=\{\gamma_i(a): a\in F^*\}$, $i=1,2$ where 
\[
\gamma_1(a)=\sec\begin{pmatrix} &   a^{-2}  \\ aI_2 &    \end{pmatrix}\ \ \ \text{and}\ \ \ \gamma_2(a)=\sec\begin{pmatrix}  & aI_2  \\ a^{-2}   & \end{pmatrix}
\]
\item and the three isolated orbits
\[
\gamma[i]=\sec(\rho^iI_3),\ \ \ i=1,2,3.
\]
\end{enumerate}
The orbit spaces $\Xi_\rel/H$ and $\Xi_\rel'/\mu_3(N\times N)$ are visibly in bijection (though the bijection is not canonical). 

\subsection{Normalization of measures}\label{ss measures}
Haar measures $dx$ on $F$ and $d^*x$ on $F^*$ are normalized so that $\int_\O\,dx=1=\int_{\O^*} \,d^*x$.

We normalize the $N\times N$ invariant measure on the quotient $(N\times N)_x\bs (N\times N)$ for $x\in \Xi_\rel'$ as in \cite{MR1340178}. 
Explicitly, if $x$ is one of the orbit representatives listed in \S\ref{ss orbreps}, for a continuous function $\varphi:(N\times N)_x\bs (N\times N)\rightarrow \C$ of compact support we write 
\[
\int_{(N\times N)_x\bs (N\times N)}\varphi(u)\ du=\int_{V_x\times N}\varphi(v,u)\ dv \ du
\]
where for any subgroup $V$ of $N$ the Haar measure $dv$ on $V$ is normalized so that $V\cap K'$ has volume one and
\[
V_x=\begin{cases}
\{e\}& x=\gamma[i], i=1,2,3\\
\{u\in N: u_{2,3}=0\} & x=\gamma_1(a),\ a\in F^*\\
\{u\in N: u_{1,2}=0\} & x=\gamma_2(a),\ a\in F^*\\
N& x=\gamma(a,b),\ a,b\in F^*.
\end{cases}
\]

A general $x\in \Xi_\rel'$ is of the form $\zeta x_0\cdot v$ for unique $\zeta\in \mu_3$ and representative $x_0$ listed in \S\ref{ss orbreps} and some $v\in N\times N$.
The measure on $(N\times N)_x\bs (N\times N)$ is transferred from $(N\times N)_{x_0}\bs (N\times N)$ via the isomorphism $(N\times N)_{x_0}\bs (N\times N)\rightarrow (N\times N)_x\bs (N\times N)$
given by
\[
(N\times N)_{x_0} u\mapsto  (N\times N)_x v^{-1} u.
\]
Note that this isomorphism is independent of the choice of $v$.

\section{Statement of the Fundamental Lemma for the spherical Hecke algebras}\label{sec: Statement of the FL}
Denote by $\H$ (resp. $\H'$) the spherical Hecke algebra of $G$ (resp. $G'$). That is, $\H$ consists of bi-$\K$-invariant functions in $\swrz(G)$ and $\H'$ of bi-$K'$-invariant functions in $\swrz(G')$. 
We explicate the isomorphism between the two algebras that underlies the Shimura lift for unramified representations. For that purpose we give a linear basis for each 
Hecke algebra and a bijection between bases that extends linearly to the algebra isomorphism. The proof that we construct an algebra isomorphism is exactly as in \cite[Theorem 3.4]{MR0840303}.

Let 
\[
d_{m,n}=\diag(\varpi^{-m-n},\varpi^{-n},1) \ \ \ \text{and}\ \ \ d'_{m,n}=\diag(\varpi^{-n-2m},\varpi^{m-n},\varpi^{2n+m}),\ \ \  m,\,n \in \Z.
\] 
Denote by $f_{m,n}\in \H$ the characteristic function of $Kd_{m,n}KZ/Z$ and by $f_{m,n}'$ the unique function in $\H'$
supported on $K'\sec(d_{m,n}')K'\mu_3$ and such that $f_{m,n}'(\sec(d_{m,n}'))=1$.  (Note that if $t=\diag(t_1,t_2,t_3)\in T_1$
then there is a function $f'\in\H'$ such that $f'(t)\neq0$ if and only if
$\val(t_1t_2^{-1})\equiv \val(t_2t_3^{-1})\equiv0\bmod 3.$)

The sets $\{f_{m,n}:m,\,n\in \Z_{\ge 0}\}$ and $\{f_{m,n}':m,\,n\in \Z_{\ge 0}\}$ are respectively bases of $\H$ and $\H'$.
We let $\Sh: \H\rightarrow \H'$ be the unique linear map extending the correspondence
\begin{equation}\label{Satake relation}
\Sh(f_{m,n})=q^{-2(m+n)}f_{m,n}',\ \ \ m,\,n \in \Z_{\ge 0}.
\end{equation}
Then $\Sh$  is an isomorphism of algebras. 

Our main theorem is the following.

\begin{theorem}\label{thm main local}
Let $f \in \H$.  Then we have the following identities of orbital integrals:
\begin{enumerate}
\item\label{FL-part1} $\O(\xi(a,b),\omega(f)\phi_\circ)=(d,c)\O'(\gamma(c,d),\Sh(f))$ for $a,b\in F^*$ where $c=-54a$ and $d=54b$.
\item\label{FL-part2} $\O(\xi_i(a),\omega(f)\phi_\circ)=t_i(c_i)\O'(\gamma_i(c_i),\Sh(f))$ for $i=1,2$ and $a\in F^*$ 
with
\begin{alignat*}{3}
c_1&=3(\rho^2-\rho)a^{-1},\quad &t_1(c)&=|c|^2\psi(3\rho c^{-1})\\
c_2&=[3(1-\rho)]^{-1}a,  &t_2(c)&=|c|^{-2}\psi(-3\rho c).
\end{alignat*}
\item\label{FL-part3}  $\O(\xi[i],\omega(f)\phi_\circ)= \O'(\gamma[i],\Sh(f))$ for $i=1,2,3$.
\end{enumerate}
\end{theorem}
We observe that  the product over all places of the local transfer factors in Theorem~\ref{thm main local} is $1$, as is needed for global applications.
Also, due to invariances for the orbital integrals,
the maps on parameters above could be given in different ways; see Remark~\ref{closing remark} below.

By linearity, it suffices to prove these identities for $f=f_{m,n}$, $m,n\in \Z_{\ge 0}$. This is carried out by direct computation of the relative orbital integrals $\O(\xi,\omega(f_{m,n})\phi_\circ)$ in \S \ref{sec I comp} 
(the generic orbit) and
\S \ref{sec I comp - small} (the smaller orbits) and of the metaplectic orbital integrals
$\O'(x,f_{m,n}')$ in \S\ref{sec J comp} (the big cell) and \S\ref{sec J comp - small} (the relevant small cells).  These results are then compared in \S\ref{sec compare} to establish the theorem.
It is expected that parts~\ref{FL-part2} and \ref{FL-part3} of the theorem could alternatively be deduced from part~\ref{FL-part1} by the study of the asymptotics of the orbital integrals in part~\ref{FL-part1}. 
Doing so would require developing a theory of relative Shalika germs in this setting.  We plan to return to this topic in a later paper.

\section{Number theoretic preliminaries}\label{ss prem} 
In this Section we collect number theoretic preliminaries that are required for the comparison of orbital integrals.

\subsection{The cubic Hilbert symbol}\label{sss Hilb}

Recall that $(\cdot,\cdot): F^*\times F^*\rightarrow \mu_3$ is the cubic Hilbert symbol.
The following basic properties of this symbol will be used throughout our computation. 
For $x,\,y,\,z\in F^*$ we have
\begin{itemize}
\item $(y,x)=\overline{(x,y)}$ and in particular $(x,x)=1$;
\item $(xy,z)=(x,z)(y,z)$;
\item $(x,1-x)=1$, $x\ne 1$;
\item $(x,y)=1$ for all $y\in F^*$ if and only if $x\in F^{*3}$. 
\end{itemize}
The main local identities of this paper are obtained in the non-archimedean case under the assumption that the residual characteristic of $F$ is greater than $3$. In this case, $1+\p\subseteq F^{*3}$ by  Hensel's Lemma and we further have
\begin{itemize}
\item $(x,u)=1$ for all $u\in \O^*$ if and only if $3| \val(x)$;
\item $(x+y,z)=(x,z)\ \ \ \text{whenever} \ \ \  \abs{y}<\abs{x}$.
\end{itemize}

\subsection{Gauss sums}

Consider the cubic Gauss sum
\[
\g=\int_{\val(x)=-1} (\varpi,x)\psi(x)\ dx.
\]
It is well known that 
\begin{equation}\label{eq abs gauss}
\g\bar\g=q.
\end{equation}
More generally, for $a\in F^*$ it will be convenient to set
\[
\g_a=\int_{\val(x)=-1}(a,x)\psi(x)\ dx.
\]
Clearly, $\g_{ab^3}=\g_a$, $b\in F^*$ and the variable change $x\mapsto -x$ shows that 
\begin{equation}\label{conjugate gauss}
\bar\g_a=\g_{a^{-1}}.
\end{equation}
If the residual characteristic of $F$ is greater than $3$, then the map $a\mapsto (\varpi,a)\g_a$ is $\O^*$-invariant and consequently,
\begin{equation}\label{the-gauss-sum}
\g_a=\begin{cases} (\varpi,a)\g & \val(a)\equiv1\bmod 3\\(\varpi,a)\bar\g & \val(a)\equiv2 \bmod 3 \\  -(\varpi,a) & \val(a)\equiv0 \bmod3.  \end{cases}
\end{equation}
Consequently, $\g_a\bar\g_a=q$ if $3\nmid \val(a)$.

Behind the following lemma is a well-known relation between Jacobi and Gauss sums.
\begin{lemma}\label{lem jac}
For $a\in F^*$ such that $3\nmid\val(a)$ we have
\[
\int_{u,\,1-u\in\O^*} (a,u(u-1))\ du=q^{-1}\frac{\g_a^2}{\bar \g_a}=q^{-2}\g_a^3.
\]
\end{lemma}
\begin{proof}
Let $\chi$ be the multiplicative character of the residual field defined by $\chi(u+\p)=(a,u)$, $u\in\O^*$ and extend $\chi$ to $\O/\p$ by setting $\chi(0)=0$. Let $\xi$ be the additive character $\xi(x+\p)=\psi(\varpi^{-1}x)$, $x\in \O$. We have 
\[
\int_{u,\,1-u\in\O^*} (a,u(u-1))\ du=q^{-1}\sum_{u\in \O/\p}\chi(u)\chi(1-u)\ \ \ \text{while}\  \ \ \g_a=(\varpi,a)\sum_{v\in \O/\p}\chi(v)\xi(v).
\]
Therefore 
\[
(\varpi,a)\g_a^2=\sum_{u,v\in \O/\p}\chi(uv)\xi(u+v)=\sum_{t\in \O/\p} c(t)\xi(t)\ \ \ \text{where}\ \ \ c(t)=\sum_{u+v=t}\chi(uv).
\]
Note that for $t=0$ we have 
\[
c(0)=\chi(-1)\sum_{u\in \O/\p}\chi^2(u)=0
\]
since $\chi$ is a character of order three while for $t\ne 0$ the variable change $(u,v)\mapsto t(u,v)$ shows that $c(t)=\chi(t^{2})c(1)=\chi^{-1}(t)c(1)$
so that we conclude that 
\[
(\varpi,a)\g_a^2=c(1)\sum_{t\in \O/\p} \chi^{-1}(t)\xi(t) =c(1)(\varpi,a)\bar\g_a
\]
and the lemma follows.
\end{proof}

We will also need the following lemma.
\begin{lemma}\label{lem sum of cubes}
For $a\in F^*$ such that $\abs{a}=q$ we have
\[
\int_{\O}\psi(ax^3)\ dx=q^{-1}[\g_a+\bar\g_a].
\]
\end{lemma}
\begin{proof}
Let $\xi$ be the character of the residual field $\O/\p$ defined by $\xi(x+\p)=\psi(ax)$, $x\in \O$, 
and let $\chi$ be the character of $(\O/\p)^*$ defined by $\chi(u+\p)=(a,u)$, $u\in \O^*$. Note that $\chi$ is a cubic multiplicative character so that 
\[
\sum_{k=-1}^1 \chi^k(u)=\begin{cases}3 & u \text{ is a cube in }(\O/\p)^* \\ 0 & \text{otherwise.}\end{cases}
\]
It follows that 
\[
\int_{\O}\psi(ax^3)\ dx=q^{-1}[1+\sum_{u\in (\O/\p)^*}\xi(u^3)]=q^{-1}[1+\sum_{u\in (\O/\p)^*}\sum_{k=-1}^1 \chi^k(u)\xi(u)].
\]
Clearly, for $k=0$ we have $\sum_{u\in (\O/\p)^*}\xi(u)=-1$ while for $k=\pm1$,  the sum is given by $\sum_{u\in (\O/p)^*} \chi^k(u)\xi(u)=g_{a^k}$. Together with \eqref{conjugate gauss} the lemma follows.
\end{proof}

\subsection{Kloosterman sums with cubic characters}\label{Kloosterman sums with cubic characters}
For $t\in F^*$ and $a,b\in F$ let 
\[
\Kl(t;a,b)=\int_{\O^*}(t,u)\psi(au+bu^{-1})\ du.
\]
For $|a|,|b|>1$, this integral may be expressed as a classical Kloosterman sum, twisted by a cubic character when $\val(t)\not\equiv0\bmod 3$.
The variable change $u\mapsto vu$ shows that
\begin{equation}\label{eq klinv}
\Kl(t;av^{-1},bv)=(t,v)\Kl(t;a,b),\ \ \ v\in \O^*.
\end{equation}
It follows from \cite[Lemma 7.2]{FO} that
\begin{equation}\label{eq klvan}
\Kl(t;a,b)=0,\ \ \ \max(a,b)\ge q^{2}\ \ \ \text{and}\ \ \ \abs{a}\ne \abs{b}. 
\end{equation}
Furthermore, it is easy to verify that
\begin{equation}\label{eq kl as gauss}
q\Kl(t;a,b)=(a,t)\g_t \
\ \ \text{and} \ \ \ q\Kl(t;b,a)=(t,a)\bar \g_t , \ \ \ b\in \O,\ \abs{a}=q
\end{equation}
and
\begin{equation}\label{eq klvol}
\Kl(t;a,b)=\begin{cases}
1-q^{-1} & 3\mid \val(t)\\
0 &\text{otherwise,}
\end{cases}\ \ \ a,b\in \O.
\end{equation}
In fact, for $k\in \Z$, the variable change $x\mapsto \varpi^{-k} x$ together with \eqref{eq klvan} shows that
\[
\int_{\abs{x}=q^k}(t,x)\psi(ax+bx^{-1})\ dx=0, \ \ \ \max(q^k\abs{a},q^{-k}\abs{b})\ge q^2\ \ \ \text{and}\ \ \ q^k\abs{a}\ne q^{-k}\abs{b}.
\]

\subsection{Cubic exponential sums}\label{Cubic exponential notation}
For $a,b\in F$ and $k\in \Z$ let
\[
\Cu(a,b;k)=\int_{\abs{x}\le q^k}\psi(ax+bx^3)\ dx \ \ \ \text{and} \ \ \ \Cu^*(a,b;k)=\int_{\abs{x}=q^k}\psi(ax+bx^3)\ dx.
\]
By definition we have
\begin{equation}\label{C and Cstar}
\Cu(a,b;k)=\Cu^*(a,b;k)+\Cu(a,b;k-1).
\end{equation}
For $a\in F$ and $n\in \Z$ we often write $a_n=\varpi^n a$. The variable change $x\mapsto x_n$ shows that
\begin{equation}\label{eq cubicvarch}
\Cu(a,b;k)=q^{-n}\Cu(a_n,b_{3n};k+n) \ \ \ \text{and} \ \ \ \Cu^*(a,b;k)=q^{-n}\Cu^*(a_n,b_{3n};k+n).
\end{equation}
Also, the variable change $x\mapsto ux$ shows that 
\begin{equation}\label{Cu and units}
\Cu(ua,u^3b;k)=\Cu(a,b;k) \quad \text{for all $u\in \O^*$}
\end{equation} 
and in particular $\Cu(a,b;k)\in\R$ (use $u=-1$).   Similarly $\Cu^*(a,b;k)\in\R$.
Combined with  \cite[Lemma 7.3]{FO} we have the following vanishing properties of cubic exponential sums, which we use frequently in the sequel:
\begin{equation}\label{vanishing conditions for Cu star} \text{$\Cu^*(a,b;k)=0$ if $\max(\abs{a}q^k,\abs{b}q^{3k})\ge q^2$ and $\abs{a}\ne \abs{b}q^{2k}$}
\end{equation}
and
\begin{equation}\label{vanishing condition for Cu}
\text{$\Cu(a,b;k)=0$ if $\max(\abs{b}q^{3k},1)<\abs{a}q^k$.}
\end{equation}
Combining \eqref{C and Cstar}, \eqref{vanishing conditions for Cu star} and \eqref{vanishing condition for Cu} we also have
\begin{equation}\label{eq another Cu vanishing}
\text{$\Cu(a,b;k)=0$ if $\abs{a}q^k\ge q^2$ and $\abs{b}q^{2k}=q\abs{a}$.}
\end{equation}
Applying \eqref{eq cubicvarch} with $n=-k$ and Lemma \ref{lem sum of cubes} we also deduce that
\begin{equation}\label{special Cu sum}
\Cu(a,b;k)=q^{k-1}[\g_b+\bar\g_b]\quad \text{when $|a|\leq q^{-k}$, $|b|=q^{1-3k}$}.
\end{equation}

Later, in Section~\ref{additional-sums}, we introduce additional notation for the cubic exponential sums that appear in our context.

\subsection{Relations between Kloosterman sums and cubic exponential sums}\label{Kl-Cu-relations}
We recall the following identity relating $\Kl(t;c,d)$ to a cubic exponential integral:
\begin{equation}\label{eq di}
\Kl(t;c,d)=(t,cd^{-1})\,\Cu(-3a,c^{-1}d^{-1}a^3;0)
\end{equation}
whenever
\begin{enumerate}
\item\label{cubic-one} either $\abs{a}=\abs{c}=\abs{d}=q$ and $3\nmid \val(t)$
\item\label{cubic-two} or $\abs{a}=\abs{c}=\abs{d}>q$.
\end{enumerate}
This identity in case~\eqref{cubic-one} is due to Duke and Iwaniec, while case~\eqref{cubic-two} is established in \cite[Corollary 7.5]{FO}.

\section{Expansions of the orbital integrals for Hecke functions}\label{sec: Expansions of the OIs}
The direct computation of the orbital integrals $\O(\xi,\omega(f)\phi_\circ)$, $f\in \H$ and $\O'(\xi',f')$, $f'\in \H'$ is quite involved. 
In this section we make use of the Iwasawa decomposition in order to expand these orbital integrals as linear combinations of orbital integrals for translates of the unit  element
in the Hecke algebra by a diagonal matrix. Orbital integrals for these test functions as well as the coefficients of the expansion are more accessible for computation.

\subsection{An expansion of the relative orbital integrals for Hecke functions on $G$}

For any $\mu=(\mu_1,\mu_2,\mu_3)\in \Z^3$ let 
\[
t_\mu=\diag(\varpi^{\mu_1},\varpi^{\mu_2},\varpi^{\mu_3}).
\] 
Embed $\Z\hookrightarrow \Z^3$ diagonally, and let $\Omega=\Z^3/\Z$.
The coset $t_\mu Z$ depends only on the image of $\mu$ in $\Omega$; we
sometimes write simply $t_\mu$ instead of $t_\mu Z$.
Let $f_\circ$ be the unit element in $\H$, that is, the characteristic function of $\K$. 

\subsubsection{An expansion for right $\K$-invariant functions}
Let $\swrz(G)^{\K,r}$ be the subspace of right $\K$-invariant functions in $\swrz(G)$.
\begin{lemma}\label{lem rexp}
Let $f\in \swrz(G)^{\K,r}$.  Then 
\begin{equation}\label{eq testexp-G}
f(g)=\sum_{\mu\in \Omega} \delta_{B}(t_\mu )\int_N f(ut_\mu^{-1})(L(ut_\mu^{-1})f_\circ)(g) \ du
\end{equation}
where on the right hand side, at most one summand is non-zero for any $g\in G$.
\end{lemma}
\begin{proof}
Let $S$ be a set of coset representatives for $G/\K$ so that $G=\sqcup_{s\in S} s \K$.  Then
\[
f(g)=\sum_{s\in S} f(s) (L(s)f_\circ)(g)
\]
with at most one summand nonzero for any fixed $g$.
By the Iwasawa decomposition we can find a set of representatives $S$ as above consisting of elements of $G$ of the form $ut_\mu^{-1}Z$ for some $u\in N$ and $\mu\in \Omega$.
Note further that for $u,u'\in N$ and $\mu,\,\mu'\in \Omega$ we have
\[
ut_\mu^{-1}Z \K=u't_{\mu'}^{-1}Z  \K
\]
if and only if $u^{-1}u'\in t_\mu^{-1} (N\cap K)t_{\mu'}Z$.
In particular, no such $u,u'$ exist unless $\mu=\mu'$.

It follows that
\[
f(g)=\sum_{\mu\in \Omega} \sum_{u\in N/t_\mu^{-1}(N\cap K)t_\mu} f(ut_\mu^{-1})(L(ut_\mu^{-1})f_\circ)(g).
\]
Clearly every summand is independent of the choice of representative $u$ and since the volume of $t_\mu^{-1}(N\cap K)t_\mu$ is $\delta_B(t_\mu)^{-1}$ we conclude that
\[
\sum_{u\in N/t_\mu^{-1}(N\cap K)t_\mu} f(ut_\mu^{-1})(L(ut_\mu^{-1})f_\circ)(g)=
\delta_B(t_\mu)\int_N f(ut_\mu^{-1}) (L(ut_\mu^{-1})f_\circ)(g) \ du.
\]
The lemma follows.
\end{proof}

\subsubsection{An expansion of the test functions on $F^8$}
\begin{lemma}\label{lem phi exp}
For $f\in \swrz(G)$, we have
\begin{equation*}
\omega(f)\phi_\circ=\sum_{\mu\in\Omega} \delta_B(t_\mu) \int_N f(ut_\mu^{-1})\,\omega(ut_\mu^{-1})\phi_\circ \ du.
\end{equation*}
\end{lemma}
\begin{proof}
We apply Lemma~\ref{lem rexp}. Note that the function $L(g)f_\circ$ is supported on the coset $gK_*$.
Since $\phi_\circ$ is a spherical vector of the Weil representation and $\inj(\Ad(K))\subseteq \Sp_{16}(\O)$ it follows that 
\[
\omega(L(g)f_\circ)\phi_\circ=\omega(g)\phi_\circ\ \ \ g\in G.
\]
The lemma is then a straightforward consequence of the expansion \eqref{eq testexp-G}.
\end{proof}

\subsubsection{An expansion of the relative orbital integrals}
We define the following auxiliary orbital integrals.
For $g\in G$, $\xi\in\Xi_{\rel}$, let 
\[
I(\xi;g)=\O(\xi,\omega(g)\phi_\circ).
\]
\begin{lemma}\label{lem whit I for G}
For every $x\in \Xi_{\rel}$, $g\in G$, $u\in N$ and $k\in K_*$ we have
\begin{itemize} 
\item $\psi(u)\O(\xi,\omega(ugk)\phi_\circ)=\O(\xi,\omega(g)\phi_\circ)$;
\item if $t=\diag(t_1,t_2,t_3)$, then $\O(\xi,\omega(t)\phi_\circ)=0$ unless $\abs{t_1}\le \abs{t_2}\le \abs{t_3}$.
\end{itemize}
\end{lemma}
\begin{proof}
The left $N$-equivariance is immediate from the definition of $\O(\xi,\omega(ug)\phi_\circ)$ as an integral over $H_\xi\backslash H$ and the variable change $h\mapsto hu^{-1}$, $u\in N$. As in the proof of Lemma~\ref{lem phi exp} we have $\omega(k)\phi_\circ=\phi_\circ$. The first part follows. The second part is a standard consequence of the first.
\end{proof}

Define the linear form
\[
\Whit(f)=\int_N f(u)\overline{\theta}(u)\ du,\ \ \ f\in \swrz(G).
\]
Let $\Lambda\subset\Omega$ be the subset
\[
\Lambda=\{\lambda=(\lambda_1,\lambda_2,\lambda_3)\in \Z^3:\lambda_1\le \lambda_2\le \lambda_3\}/\Z.
\]
Then Lemmas~\ref{lem phi exp} and \ref{lem whit I for G} together with a simple variable change after substituting into the
definition of the orbital integral 
imply the following corollary.

\begin{corollary}\label{cor genexp-G}
Let $f\in \swrz(G)^{K,r}$, $x\in \Xi_\rel$.  Then we have the following expansion of the relative orbital integrals:
\begin{equation}\label{eq l-expansion2-G}
\O(x,\omega(f)\phi_\circ)=\sum_{\mu\in \Lambda} \delta_{B}(t_\mu)\Whit(R(t_\mu^{-1})f)\O(x,\omega(t_\mu^{-1})\phi_\circ).
\end{equation}
\end{corollary}

\subsubsection{Computation of the Whittaker coefficients for basis elements of $\H$}\label{Computation of the psi-nonmetaplectic}
We compute the coefficients in \eqref{eq l-expansion2-G} for $f=f_{m,n}$. For $\mu\in\Lambda$, set 
\[
\Psi_{m,n}(\mu)=\Whit(R(t_\mu^{-1})f_{m,n}).
\]

To evaluate the coefficients $\Psi_{m,n}(\mu)$, we require the following lemma concerning the Cartan decomposition.  We work in $\GL_3(F)$.
Recall that for $\lambda=(\lambda_1,\lambda_2,\lambda_3)\in \Z^3$, 
$
t_\lambda=\diag(\varpi^{\lambda_1},\varpi^{\lambda_2},\varpi^{\lambda_3}).
$
Let $\abs{\lambda}=\lambda_1+\lambda_2+\lambda_3$ so that $\det t_\lambda=\varpi^{\abs{\lambda}}$.

For $\mu,\lambda\in \Z^3$ let
\begin{equation}\label{G sub lambda mu}
G_{\mu,\lambda}=\{(x,y,z)\in F^3: u=\begin{pmatrix} 1 & x & z \\ & 1 & y \\ & & 1 \end{pmatrix} \text{ and }ut_\mu^{-1}\in Kt_\lambda K\}
\end{equation}
and
\[
\Gamma_{\mu,\lambda}=\{(x,y,z)\in G_{\mu,\lambda}: \val(x),\,\val(y)\ge -1\}.
\]
The description of the sets $G_{\mu,\lambda}$ and $\Gamma_{\mu,\lambda}$ may be obtained from  \cite[\S 3]{MR1703755}, as follows.
\begin{lemma}\label{lem gamma}
Let $\mu\in \Z^3$ and $\lambda=(\lambda_1,\lambda_2,\lambda_3)\in \Z^3$ with $\lambda_1\le \lambda_2\le \lambda_3$.
\begin{itemize}
\item The set $G_{\mu,\lambda}$ is empty unless $\abs{\mu}+\abs{\lambda}=0$ and $-\lambda_3\le \mu_1,\, \mu_2,\ \mu_3\le -\lambda_1$.
\item If in addition $\mu\in \Lambda$, $\abs{\mu}+\abs{\lambda}=0$ and $-\lambda_3\le \mu_1,\, \mu_2,\ \mu_3\le -\lambda_1$ then $\Gamma_{\mu,\lambda}$ is empty except in the following eight cases:
\begin{enumerate}
\item $\mu=(-\lambda_3,-\lambda_2,-\lambda_1)$ in which case
\[
\Gamma_{\mu,\lambda}=\O^3.
\]
\item $\mu=(-\lambda_3,-\lambda_2+1,-\lambda_1-1)$ in which case
\[
\Gamma_{\mu,\lambda}=\{(x,y,z)\in F: x\in \O,\,\val(y)=-1,\,z\in xy+\O\}.
\]
\item $\mu=(1-\lambda_3,-\lambda_2-1,-\lambda_1)$ in which case
\[
\Gamma_{\mu,\lambda}=\{(x,y,z)\in F: \val(x)=-1,\,y,z\in \O\}.
\]
\item $\mu=(1-\lambda_3,-\lambda_2,-\lambda_1-1)$ in which case
$\Gamma_{\mu,\lambda}$ is the disjoint union of the three domains defined by the conditions
\begin{enumerate}
\item $x,y\in\O, \val(z)=-1$,
\item $\val(x)=-1$, $y\in \O$ and if $\lambda_2=\lambda_1+1$ then $\val(z)\ge -1$ (resp. if  $\lambda_2>\lambda_1+1$ then $\val(z)=-1$)
and
\item $\val(y)=-1$, $x\in \O$ and if $\lambda_3=\lambda_2+1$ then $\val(z-xy)\ge -1$ (resp. if  $\lambda_3>\lambda_2+1$ then $\val(z-xy)=-1$).
\end{enumerate}
\item $\mu=(1-\lambda_3,-\lambda_2+1,-\lambda_1-2)$ in which case
\[
\Gamma_{\mu,\lambda}=\{(x,y,z)\in F: \val(x)=\val(y)=-1,\ \val(z-xy)\ge -1\}.
\]
\item $\mu=(2-\lambda_3,-\lambda_2-1,-\lambda_1-1)$ in which case
\[
\Gamma_{\mu,\lambda}=\{(x,y,z)\in F: \val(x)=\val(y)=-1, \val(z)\geq -1\}.
\]
\item $\mu=(2-\lambda_3,-\lambda_2,-\lambda_1-2)$ in which case
$\Gamma_{\mu,\lambda}$ is the disjoint union of the two domains defined by the conditions
\begin{enumerate}
\item $\val(x)=\val(y)=-1$, $\val(z)=\val(xy-z)=-2$ and
\item $\val(x),\val(y),\val(xy)\ge -1,\ \val(z)=-2$.
\end{enumerate}
\item $\mu=(\mu_1,-\lambda_2,-\lambda_1-\lambda_3-\mu_1)$ with $\mu_1\ge 3-\lambda_3$ in which case
\[
\Gamma_{\mu,\lambda}=\{(x,y,z)\in F: x,y\in \p^{-1},\val(z)=-(\mu_1+\lambda_3)\}.
\]
\end{enumerate}
\end{itemize}
\end{lemma}

This leads to the following evaluation.

\begin{proposition}\label{prop psi-G} 
Let $m,\,n\in \Z_{\ge 0}$ and $\mu\in \Lambda$. Then $\Psi_{m,n}(\mu)=0$ except in the following seven cases; for each we
give a representative for $\mu$ modulo diagonal translation by $\Z$.
\begin{enumerate}
\item $\mu=(0,n,m+n)$ in which case
\[
\Psi_{m,n}(\mu)=1.
\]
\item $\mu=(0,n+1,m+n-1)$ in which case
\[
\Psi_{m,n}(\mu)=-1.
\]
\item $\mu=(1,n-1,m+n)$ in which case
\[
\Psi_{m,n}(\mu)=-1.
\]
\item $\mu=(1,n,m+n-1)$ in which case
maybe more neat to write
\[
\Psi_{m,n}(\mu)=1-q-\delta_{n,1}-\delta_{m,1}
\]
where $\delta_{i,j}$ is the usual Kronecker delta function. 
\item $\mu=(1,n+1,m+n-2)$ in which case
\[
\Psi_{m,n}(\mu)=q.
\]
\item $\mu=(2,n-1,m+n-1)$ in which case
\[
\Psi_{m,n}(\mu)=q.
\]
\item $\mu=(2,n,m+n-2)$ in which case
\[
\Psi_{m,n}(\mu)=-q.
\]
\end{enumerate}
\end{proposition}

\begin{proof} Fix $m,n\in\Z_{\geq0}$. To prove Proposition~\ref{prop psi-G}, we observe that
$$\Psi_{m,n}(\mu)=\int_N f_{m,n}(ut_\mu^{-1})\overline{\theta}(u)\,du.$$
Fix a representative $\mu_0\in\Z^3$ for $\mu$. 
For $p\in\Z$, let $\lambda_p=(-m-n,-n,0)+(p,p,p)$. Then the integrand is supported on $\cup_{p\in\Z}G_{\mu_0,\lambda_p}$.
(The $p$ comes from the center $Z$.) 
If the set
$G_{\mu_0,\lambda_p}$ is non-empty, then $f_{m,n}(ut_\mu^{-1})=1$ for all $u\in G_{\mu_0,\lambda_p}$.
Moreover by Lemma~\ref{lem gamma}, $G_{\mu_0,\lambda_p}$ is non-empty
only when $|\mu_0|+|\lambda_0|+3p=0$, so there is at most one $p$ such that this condition is satisfied.  

We observe that for any $p$, 
\begin{equation}\label{eq vanish part}
\int_{{G_{\mu_0,\lambda_p}}\setminus {\Gamma_{\mu_0,\lambda_p}}} \overline{\psi}(x+y)\ dx\ dy\ dz=0.
\end{equation}
Indeed, both $G_{\mu_0,\lambda_p}$ and $\Gamma_{\mu_0,\lambda_p} $ and therefore also their difference is preserved by the variable change $(x,y,z)\mapsto (v_1x,v_2y,v_1v_2z)$ for any
$v_1,~v_2\in \O^*$. It follows that for any $v_1,~v_2\in\O^*$,
\[
\int_{{G_{\mu_0,\lambda_p}}\setminus {\Gamma_{\mu_0,\lambda_p}}} \overline{\psi}(x+y)\ dx\ dy\ dz=
\int_{{G_{\mu_0,\lambda_p}}\setminus {\Gamma_{\mu_0,\lambda_p}}}  \overline{\psi}(v_1x+v_2y)\ dx\ dy\ dz.
\]
The vanishing \eqref{eq vanish part} follows by averaging over $v_1,\,v_2\in \O^*$ since if $G_{\mu_0,\lambda_p}$ is non-empty then
 for $(x,y,z)\in {G_{\mu_0,\lambda_p}}\setminus {\Gamma_{\mu_0,\lambda_p}}$ we have $\min(\val(x),\val(y))\le -2$.

Thus we have 
$$\Psi_{m,n}(\mu)=\int_{\Gamma_{\mu_0,\lambda_p}} \overline{\theta}(u)\,du$$
when $|\mu_0|+|\lambda_0|+3p=0$, and $\Psi_{m,n}(\mu)=0$ if no such $p$ exists.
The evaluation of this integral is straightforward using Lemma~\ref{lem gamma} and is omitted here.
\end{proof}

\subsubsection{An expansion of the relative orbital integrals for basis elements}
In order to write an explicit expansion for basis elements we introduce the following notation.  
For $x\in\Xi_\rel$, let
\begin{equation}\label{I(x;i,j)}
I(x;i,j)=\delta_{B}(d_{i,j})\O(x,\omega(d_{i,j}^{-1})\phi_\circ)\quad i,j\in\Z_{\geq0}, 
\end{equation}
and let $I(x;i,j)=0$ if $i,j\in \Z$ with $\min(i,j)<0$.
Then combining Corollary~\ref{cor genexp-G} and Proposition~\ref{prop psi-G}, we obtain the following expansion.
\begin{corollary}\label{cor exp j-G}
For $x\in \Xi_\rel$ and $m,n\in\Z_{\ge 0}$ we have
\begin{multline*}
\O(x,\omega(f_{m,n})\phi_\circ)=I(x;n,m)-[I(x;n+1,m-2)+I(x;n-2,m+1)] \\ +[1-q-\delta_{n,1}-\delta_{m,1}]I(x;n-1,m-1)\\ +q[I(x;n,m-3)+I(x;n-3,m)]-qI(x;n-2,m-2)].
\end{multline*}
\end{corollary}
\subsection{An expansion of the metaplectic orbital integrals for Hecke functions on $G'$}

Let $f_\circ'=f_{0,0}'$ be the unit element in $\H'$. 
For $i,j\in \Z$ let $c_{i,j}=\diag(\varpi^{-j},\varpi^{j-i},\varpi^i)$. We have  
\begin{equation}\label{eq invcij}
\sec(c_{i,j}^{-1})=\sec(c_{i,j})^{-1}.
\end{equation}  

\subsubsection{An expansion for right $K'$-invariant functions}
Let $\swrz(G')^{K',r}$ be the subspace of right $K'$-invariant functions in $\swrz(G')$.
\begin{lemma}\label{lem rexp'}
Let $f\in \swrz(G')^{K',r}$.  Then 
\begin{equation}\label{eq testexp}
f(g)=\sum_{i,j\in \Z} \delta_{B'}(\sec(c_{i,j}))\int_N f(u\sec(c_{i,j}^{-1}))(L(u\sec(c_{i,j}^{-1}))f_\circ')(g) \ du. 
\end{equation}
where on the right hand side, at most one summand is non-zero for any $g\in G'$.
\end{lemma}
\begin{proof}
It suffices to replace $g\in G'$ by $\sec(g)$ for $g\in \SL_3(F)$.
Since 
\[\sec(gk)=(\kappa(k)\sigma(g,k))^{-1}\sec(g)(k,\kappa(k))
\] 
for $g\in \SL_3(F)$ and $k\in K_1$ we also have
\[
f(\sec(gk))=\kappa(k)\sigma(g,k)f(\sec(g)) ,\ f\in \swrz(G')^{K',r}.
\]
Applying this also for $f=L(\sec(g))f_\circ'\in  \swrz(G')^{K',r}$ we obtain that
\[
(L(\sec(g))f_\circ')(\sec(gk))=\kappa(k)\sigma(g,k)(L(\sec(g))f_\circ')(\sec(g))=\kappa(k)\sigma(g,k)f_\circ'(e)=\kappa(k)\sigma(g,k).
\]
We conclude that
\begin{equation}\label{eq kequiv}
f(\sec(gk))=f(\sec(g))(L(\sec(g))f_\circ')(\sec(gk)),\ \ \ g\in \SL_3(F), \ k\in K_1,\ f\in \swrz(G')^{K',r}.
\end{equation}
Based on the Iwasawa decomposition we have the disjoint union
 $
 \SL_3(F)=\sqcup_{i,j\in\Z} Nc_{i,j}^{-1}K$
or more precisely
\begin{itemize}
\item for every $g\in \SL_3(F)$ there exists a unique pair $(i,j)\in \Z^2$ and some $u\in N$ and $k\in K_1$ such that $g=uc_{i,j}^{-1}k$ and
\item for $u_1,u_2\in N$ and $k_1,k_2\in K_1$ we have $u_1c_{i,j}^{-1}k_1=u_2c_{i,j}^{-1}k_2$ if and only if $u_1^{-1}u_2=c_{i,j}^{-1}k_1k_2^{-1}c_{i,j}\in N\cap  c_{i,j}^{-1}Kc_{i,j}$.
\end{itemize}
Thus for any choice of representatives $[ N/(N\cap c_{i,j}^{-1}Kc_{i,j})]$ for the corresponding quotient space the map $(u,k)\mapsto uc_{i,j}^{-1}k$ from $
\sqcup_{i,j\in\Z}  [N/(N\cap c_{i,j}^{-1}Kc_{i,j})] \times K_1$ to $\SL_3(F)$ is a bijection.
Together with \eqref{eq kequiv} we conclude that for $f\in \swrz(G')^{K',r}$,
\[
f(\sec(g))=\sum_{i,j\in \Z} \sum_{u\in [N/(N\cap c_{i,j}^{-1}Kc_{i,j})]}f(u\sec(c_{i,j}^{-1}))(L(u\sec(c_{i,j}^{-1}))f_\circ')(\sec(g)).
\]
We further note that any summand $f(u\sec(c_{i,j}^{-1}))(L(u\sec(c_{i,j}^{-1}))f_\circ')(\sec(g))$ is independent of the choice of representative $u\in N/(N\cap c_{i,j}^{-1}Kc_{i,j})$. Indeed, if 
$u_0\in N\cap c_{i,j}^{-1}Kc_{i,j}$ then $u_0':=c_{i,j} u_0c_{i,j}^{-1}\in N\cap K$ so that $\kappa(u_0')=1$ and $uu_0 \sec(c_{i,j}^{-1})=u\sec(c_{i,j}^{-1})u_0'$. Consequently $f(uu_0 \sec(c_{i,j}^{-1}))=f(u\sec(c_{i,j}^{-1}))$ and $(L(uu_0 \sec(c_{i,j}^{-1}))f_\circ')(\sec(g))=(L(u\sec(c_{i,j}^{-1}))f_\circ')(\sec(g))$.
Since the volume of $N\cap c_{i,j}^{-1}Kc_{i,j}$ is $\delta_{B'}(\sec(c_{i,j}))^{-1}$ we deduce \eqref{eq testexp}.
The assertion that at most one summand on the right hand side is non-zero follows from the disjointness in the Iwasawa decomposition.
\end{proof}

\subsubsection{An expansion of the metaplectic orbital integrals}
Let 
\[
T^+=\{(t,\zeta):\zeta\in \mu_3,t=\diag(t_1,t_2,t_3)\in T_1,\ \abs{t_1}\le \abs{t_2}\le \abs{t_3}\}.
\]

\begin{lemma}\label{lem J equiv}
For every $x\in \Xi_\rel'$, $g\in G'$, $u\in N$ and $k\in K'$ we have
\begin{itemize}
\item $\O'(x,L(ugk)f_\circ')=\theta(u)\O'(x;L(g)f_\circ')$;
\item $\O'(x;L(t)f_\circ')=0$, $t\in T'\setminus T^+$.
\end{itemize}
\end{lemma}
\begin{proof}
By definition $L(k)f_\circ'=f_\circ'$ and the first equality follows in conjunction with \eqref{eq Nequiv}.
From the first part of the lemma it follows that for $t\in T'$ if $\O'(x,L(t)f_\circ')\ne 0$ then $\theta$ is trivial on $t(N\cap K')t^{-1}$ or equivalently $t\in T^+$.
\end{proof}

Let
\[
\Lambda'=\{(i,j)\in \Z^2: 2i\ge j,\ 2j\ge i\}
\]
and note that $\sec(c_{i,j}^{-1})\in T^+$ if and only if $(i,j)\in \Lambda'$.

Define the linear form
\[
\Whit'(f)=\int_N f(u)\theta(u)\ du,\ \ \ f\in \swrz(G').
\]
Combining \eqref{eq testexp} and Lemma \ref{lem J equiv} we obtain the following corollary.
\begin{corollary}\label{cor genexp}
Let $f\in \swrz(G')^{K',r}$, $x\in \Xi_\rel'$.  Then we have the following expansion of the metaplectic orbital integrals:
\begin{equation}\label{eq l-expansion2}
\O'(x,f)=\sum_{(i,j)\in \Lambda'} \delta_{B'}(\sec(c_{i,j}))\Whit'(R(\sec(c_{i,j}^{-1}))f)\O'(x,L(\sec(c_{i,j})^{-1})f_\circ').
\end{equation}
\end{corollary}
\noindent We remark that there is a similar expansion with left and right interchanged.

\subsubsection{Computation of the Whittaker coefficients for basis elements of $\H'$}
We describe the coefficients in \eqref{eq l-expansion2} for $f=f_{m,n}'$. Set 
\[
\Psi'_{m,n}(i,j)=\Whit'(R(\sec(c_{i,j}^{-1}))f_{m,n}'),\ \ \ (i,j)\in\Lambda',\ m,n\in\Z_{\ge 0}.
\]
\begin{proposition}\label{prop psil} Let $m,\,n\in\Z_{\ge 0}$ and $(i,j)\in \Lambda'$. Then $\Psi'_{m,n}(i,j)=0$ except in the following six cases:
\begin{enumerate}
\item  $\Psi'_{m,n}(2m+n,m+2n)=1$.
\item $\Psi'_{m,n}(2m+n-1,m+2n)=\bar\g$.
\item $\Psi'_{m,n}(2m+n,m+2n-1)=\bar\g$.
\item $\Psi'_{m,n}(2m+n-2,m+2n-1)=q^2.$
\item $\Psi'_{m,n}(2m+n-1,m+2n-2)= q^2.$
\item $\Psi'_{m,n}(2m+n-2,m+2n-2)
= q^2\bar \g.$
\end{enumerate}
\end{proposition}

The proof of this proposition is more complicated than
the analogous computation in Section~\ref{Computation of the psi-nonmetaplectic} above as one must work
in the covering group and keep careful track of the cocycle $\sigma$ and the splitting $\kappa$.  Fortunately,
the computation of the Whittaker coefficients of a specific bi-$K'$-invariant anti-genuine function for the $n$-fold cover of $\GL_3$ (namely, the generating
function attached to a local L-function) was carried out in \cite{MR1703755}, Proposition~3.1.  
The steps required to prove Proposition~\ref{prop psil} are essentially the same -- one carefully
computes the Cartan decomposition of upper triangular matrices in the covering group and uses this information to determine the
contribution to the coefficients $\Psi'_{m,n}$.  
We will use those calculations below.

\begin{proof}  We turn to the proof of the proposition.
We must compute 
\begin{equation}\label{integral formula for the coeffs G'}
\Psi'_{m,n}(i,j)=\Whit'(R(\sec(c_{i,j}^{-1}))f_{m,n}')=\int_N f_{m,n}'(u\sec(c_{i,j}^{-1}))\theta(u)\,du.
\end{equation}
Here $f_{m,n}'$ is anti-genuine and satisfies 
\begin{equation}\label{eq ant gen}
f_{m,n}'(k_1d'_{m,n}k_2)=\kappa(k_1)\kappa(k_2)\sigma(k_1,d'_{m,n}k_2)\sigma(d'_{m,n},k_2),\ \ \ k_1,\,k_2\in K_1.
\end{equation}
Throughout this Section, we use the coordinates \eqref{coords-u} for $u$.

Let 
$$\lambda=(-n-2m,m-n,2n+m),\qquad \mu=(-j,j-i,i).$$
Then the support of the integrand in \eqref{integral formula for the coeffs G'} is $G_{\mu,\lambda}$ (see \eqref{G sub lambda mu}).
The first step in our computation is the following analog of the main step in the proof of Proposition~\ref{prop psi-G}.
\begin{lemma}\label{lem G' coeff reduction}
We have 
\[
\Psi_{m,n}'(i,j)=\int_{\Gamma_{\mu,\lambda}} f_{m,n}'(u\sec(c_{i,j}^{-1}))\psi(x+y)\ dx\ dy\ dz.
\]
\end{lemma}

\begin{proof} We show that 
\[
\int_{G_{\mu,\lambda}\setminus \Gamma_{\mu,\lambda}} f_{m,n}'(u\sec(c_{i,j}^{-1}))\psi(x+y)\ dx\ dy\ dz=0
\]
by an argument similar to the proof of Proposition~\ref{prop psi-G}.  Indeed, for any $v_1, v_2\in \O^*$, let $\eta=\diag(v_1v_2,v_2,1)\in K$.
Note that $f_{m,n}'$ extends uniquely to a bi-$K''$-invariant anti-genuine function on $K''G'$.
Using the bi-$K''$-invariance  of $f'_{m,n}$ and computing the cocycles, we find that for any $v_1,v_2\in\O^*$,
$$f'_{m,n}(u\sec(c_{i,j}^{-1}))=(v_1,\varpi)^j(v_2,\varpi)^i f'_{m,n}(\sec(\eta u \eta^{-1})\sec(c_{i,j}^{-1})).$$
Making the variable change in $u$ and integrating over $v_i$, we get zero since for any $\ell$
$$\int_{|v|=1} (v,\varpi)^\ell \psi(vx)\,dv=0$$ 
when $|x|\geq q^2$.
\end{proof}

To carry out the evaluation, we now treat separately each of the eight cases that appear in Lemma~\ref{lem gamma}.
We need to evaluate $f'_{m,n}$ on various domains of integration, using \eqref{eq ant gen}.  The computations are, as noted above,
essentially the same as those
carried out in \cite{MR1703755}, Section 3; to make the comparison we observe
that in \cite{MR1703755} the authors first made the variable change $z\mapsto xyz$, and the roles of $y$ and $z$ in that paper
are interchanged with their roles here.

{\it Case 1:}~{$i=2m+n$, $j=m+2n$.}
In this case the domain of integration is $x,y,z\in\O$, and it is immediate to see that
$$\Psi_{m,n}'(n+2m,m+2n)=1.$$

{\it Case 2:}~{$i=2m+n-1, j=m+2n$.} Since $(i,j)\in \Lambda'$, this case arises only when $m>0$.
The domain of integration here is 
$
\val(y)=-1, x\in \O,\ z\in xy+\O.
$
In this case a short computation with the metaplectic Cartan decomposition for $\SL_2$ (see \cite[Lemma 3.2]{MR1703755}) shows that
\[
f_{m,n}'(u\sec(c_{i,j}^{-1}))=f_{m,n}'(\sec\begin{pmatrix} \varpi^{m+2n} &  & \\ & \varpi^{m-n-1} & y\varpi^{1-2m-n} \\ & & \varpi^{1-2m-n} \end{pmatrix})=(y,\varpi), \quad \val(y)=-1
\]
and therefore 
$$
\Psi_{m,n}'(2m+n-1,m+2n)=\int_{\val(y)=-1}(y,\varpi)\psi(y) \int_\O  \int_{xy+\O}\ dz \ dx\ dy=\bar\g.
$$

{\it Case 3:}~{$i=2m+n, j=m+2n-1$.}
Here $n>0$.
The domain of integration is given by
$\val(x)=-1,\,y,z\in \O.$
As above, we find that
\[
f_{m,n}'(u\sec(c_{i,j}^{-1}))=f_{m,n}'(\sec\begin{pmatrix} \varpi^{m+2n-1} & \varpi^{m-n+1}x & \\ & \varpi^{m-n+1} &  \\ & & \varpi^{-2m-n} \end{pmatrix})=(x,\varpi), \quad \val(x)=1
\]
so that 
\[
\Psi'_{m,n}(2m+n,m+2n-1)=\int_{\val(x)=-1}(x,\varpi)\psi(x) \, dx=\bar\g.
\]

{\it Case 4:}~{$i=2m+n-1, j=m+2n-1$.}   Since $(i,j)\in \Lambda'$, $m,n>0$ so $\lambda_2>  \lambda_1+1$
and $\lambda_3> \lambda_2+1$. By Lemma~\ref{lem gamma},
the domain of integration is the disjoint union of the three domains
\begin{enumerate}
\item $x,y\in\O$, $\val(z)=-1$,
\item $\val(x)=-1$, $y\in \O$, $\val(z)=-1$, and
\item $x\in \O$, $\val(y)=-1$, $\val(z-xy)=-1$.
\end{enumerate}
On the first piece, similarly to the above, $f'_{m,n}(u\sec(c_{i,j}^{-1}))=(z,\varpi)$.
The contribution of this subdomain is therefore
$$\int_{x,y\in\O, \val(z)=-1}(z,\varpi)\,dz\,dx\,dy=0.$$
For the second piece, after some work with the Cartan decompositions, once again we find that $f'_{m,n}(u\sec(c_{i,j}^{-1}))=(z,\varpi)$.
For the third, $f'_{m,n}(u\sec(c_{i,j}^{-1}))=(z-xy,\varpi)$.
For the second case this is treated in \cite{MR1703755}, Section 3, Case 4b, and for the third case, this is \cite{MR1703755}, Section 3, Case 4a.
For each piece it is then easy to check that the resulting integral is zero.  We conclude that 
$\Psi'_{m,n}(2m+n-1,m+2n-1)=0.$

{\it Case 5:}~{$i=2m+n-2, j=m+2n-1$.}   Here $m>0$.  By Lemma~\ref{lem gamma}, the domain of integration is $\val(x)=\val(y)=-1$, $\val(z-xy)\geq -1$.
On this subdomain, following \cite{MR1703755}, Section 3, Case 5 (and correcting for a missing inverse there) we find that 
$f'_{m,n}(u\sec(c_{i,j}^{-1}))=(\varpi,xz)$.  This gives
$$\Psi'_{m,n}(2m+n-2,m+2n-1)=\int_{|x|=|y|=q,|z-xy|\leq q} (\varpi,xz)\psi(x+y)\,dx\,dy.$$
After changing $z\mapsto xyz$ we obtain an integral that is easy to evaluate.
We conclude that $\Psi'_{m,n}(2m+n-2,m+2n-1)=q\g\overline{\g}=q^2$.

{\it Case 6:}~{$i=2m+n-1, j=m+2n-2$.}   Here $n>0$.
By Lemma~\ref{lem gamma}, the domain of integration is $\val(x)=\val(y)=-1$, $\val(z)\geq -1$.
For this case we follow \cite{MR1703755}, Section 3, Case 6 to obtain
$f'_{m,n}(u\sec(c_{i,j}^{-1}))=(\varpi,x^2y)$.  Thus
$$\Psi'_{m,n}(2m+n-1,m+2n-2)=\int_{\val(x)=\val(y)=-1,\val(z)\geq -1}(\varpi,x^2y)\psi(x+y)\,d(x,y,z)=
q\g\overline{\g}=q^2.$$

{\it Case 7:}~{$i=2m+n-2, j=m+2n-1$.}   Here $m>0$.   By Lemma~\ref{lem gamma}, the domain of integration is
the disjoint union of two subdomains: 
\begin{enumerate}
\item $\val(x)=\val(y)=-1$, $\val(z)=\val(xy-z)=-2$ and
\item $\val(x),\val(y),\val(xy)\ge -1,\ \val(z)=-2$.
\end{enumerate}
For the first domain, following \cite{MR1703755}, Section 3, Case 7, we find that 
$f'_{m,n}(u\sec(c_{i,j}^{-1}))=(z(z-xy),\varpi).$ 
Changing $z\mapsto xyz$ and using Lemma~\ref{lem jac}, we obtain a contribution of
\begin{multline*}
\int_{|x|=|y|=q, |z|=|z-xy|=q^2}(z(z-xy),\varpi)\psi(x+y)\,dx\,dy\\
=q^2\int_{|x|=|y|=q} (\varpi,xy)\psi(x+y)\,dx\,dy \int_{|z|=|z-1|=1}(z(z-1),\varpi)\,dz\\
=q^2 \g^2 \left(q^{-2}\overline{\g}^3\right)=q^2\overline{\g}.
\end{multline*}
 
For the second subdomain, in \cite{MR1703755}, this is broken into three pieces and the Hilbert symbols are evaluated for each of them.
In each case the integral gives zero.  The arguments in this case are identical and are omitted here.  

We conclude that 
$\Psi'_{m,n}(2m+n-2,m+2n-1)=q^2\overline{\g}$.

{\it Case 8:}~{$i-j=m-n$ with $j\leq 2n+m-3$.}  This is our final case.
The  domain here is $\val(x),\val(y)\geq -1$, $\val(z)=i-n-2m$.   Once again the argument is the same as in \cite{MR1703755}, Case 8, and we find
that this contribution is zero as well.
 This concludes the proof of  Proposition~\ref{prop psil}.
 \end{proof}

\subsubsection{An expansion of the metaplectic orbital integrals for basis elements}
In order to give the explicit expansion for basis elements it will be convenient to set the following notation.
For $x\in \Xi_\rel'$, let
\begin{equation}\label{the J^l}
J^l(x;i,j)= \delta_{B'}(\sec(c_{i,j}))\,\O'(x,L(\sec(c_{i,j})^{-1})f_\circ')\qquad (i,j)\in \Lambda',
\end{equation}
and let $J^l(x;i,j)=0$ if $(i,j)\not\in \Lambda'$.
Then combining Corollary~\ref{cor genexp} and Proposition~\ref{prop psil}, we obtain the following expansion.

\begin{corollary}\label{cor exp j}
For $x\in \Xi_\rel'$ and $m,n\in\Z_{\ge 0}$ we have
\begin{multline*}
\O'(x,f_{m,n}')=J^l(x;2m+n,m+2n)+q^2\bar \g J^l(x;2m+n-2,m+2n-2)\\
\bar \g[J^l(x;2m+n-1,m+2n)+J^l(x;2m+n,m+2n-1)]+\\
q^2[J^l(x;2m+n-2,m+2n-1)+J^l(x;2m+n-1,m+2n-2)].
\end{multline*}
\end{corollary}

\subsubsection{Behavior of the orbital integrals under an involution of $\H'$}\label{symmetry-J}
Consider the involution $f^*(g)=\overline{f(g^{-1})}$ on $\swrz(G')$.
Note that
\begin{equation}\label{eq lr}
(L(g)f)^*=R(g)(f^*),\ f\in \swrz(G')
\end{equation}
and in particular the involution preserves $\H'$.
The variable change $(u_1,u_2)\mapsto (u_2,u_1)$ shows that
\begin{equation}\label{eq orb inv}
\O'(x,f^*)=\overline{\O'(x^{-1},f)},\ \ \ f\in \swrz(G').
\end{equation}
\begin{lemma}\label{lem bassym}
We have 
\[
(f_{m,n}')^*=f_{n,m}',\ \ \ m,n\in\Z_{\ge 0}.
\]
\end{lemma}
\begin{proof}
Let $w_0=\left(\begin{smallmatrix}&&1\\&-1&\\1&&\end{smallmatrix}\right)\in K_1$ and note
that $w_0 (d_{m,n}')^{-1}w_0=d_{n,m}'$.
It follows that $(\mu_3 K'\sec(d_{m,n}') K')^{-1}=\mu_3 K'\sec(d_{n,m}' )K'$ so that both $(f_{m,n}')^*$ and $f_{n,m}'$ are functions in $\H'$ supported on $\mu_3 K'd_{n,m}' K'$ and hence determined by their values on $\sec(d_{n,m}')$. We have 
\[
\sec(d_{n,m}')^{-1}=\sec((d_{n,m}')^{-1})=\sec(w_0 d_{m,n}' w_0)=\sigma(w_0,d_{m,n}'w_0)\sigma(d_{m,n}',w_0)\sec(w_0)\sec(d_{m,n}')\sec(w_0).
\]
We observe that $\kappa(w_0)=1$, that is, $\sec(w_0)\in K' $ and that $\sigma(w_0,d_{m,n}'w_0)=1=\sigma(d_{m,n}',w_0)$.  
It follows that $f_{m,n}'(\sec(d_{n,m}')^{-1})=1$ or equivalently that $f_{m,n}^*(\sec(d_{n,m}'))=1=f_{n,m}'(\sec(d_{n,m}'))$. The lemma follows.
\end{proof}

\begin{corollary}\label{cor Jfe}
 Let $m,n\in\Z_{\geq0}$, 
and let $x\in\Xi'_\rel$.  Then
$$
\O'(x,f_{n,m}')=\overline{\O'(x^{-1},f_{m,n}')}.
$$
\end{corollary}

\begin{proof} This follows at once from \eqref{eq orb inv} and Lemma~\ref{lem bassym}.  \end{proof} 

Let $\Psi_{m,n}^r(i,j)=\Whit'(L(\sec(c_{j,i}^{-1}))f_{m,n}')$ (these are the coefficients one would obtain if reversing right and left in Corollary~\ref{cor genexp}).
Then we remark that it follows from the above that $\Psi'_{m,n}(i,j)=\overline{\Psi_{n,m}^r(j,i)}.$  Hence interchanging right and left would require essentially
the same computations.

\section{Formulas for the orbital integrals for Hecke functions on $\PGL_3(F)$: the generic family}\label{sec I comp}
In this Section we evaluate the relative orbital integrals attached to the generic family of relevant orbits.    We shall first
evaluate the orbital integrals $I(\xi(a,b);i,j)$ defined by \eqref{I(x;i,j)}, and then put them together using Corollary~\ref{cor exp j-G}.
When $x=\xi(a,b)$, we shorten the notation by writing $I(a,b;i,j)$ instead of $I(\xi(a,b);i,j)$.   For any $a\in F$, $k\in\Z$,
we recall that $a_k:=\varpi^k a$.

\subsection{Coordinate form of the orbital integral $I(a,b;i,j)$}\label{Coordinate form of the oi for I}
In order to write the orbital integral $I(a,b;i,j)$ in coordinates, in addition to the explication of $\inj(h,\Ad(n))$ for 
$h\in \SL_2(F)$ and $n\in N$ in \cite{FO}, Sections~3, 4, we also recall that for $t=\diag(t_1,t_2,t_3)$ we have
\[
\inj(\Ad(t))=\diag(\tau(t),w_8\,^t\tau(t)^{-1}w_8)
\]
where $\tau(t)=\diag(t_1t_3^{-1}I_2,t_2t_3^{-1}I_2,t_1t_2^{-1}I_2, I_2)$ and $w_8=(\delta_{i,9-j})\in \GL_8(F)$.
We observe that
\[
\abs{\det \tau(t)}^{\frac12}=\abs{t_1t_3^{-1}}^2=\delta_B(t).
\]
Consequently, $I(a,b;i,j)$ equals
\begin{multline*}
\int_{F^4}\int_{F^*} \phi_ \circ[(0,\varpi^{i+j} t,\varpi^jt^{-1}a,\varpi^jt(x+as),\varpi^{i}t^{-1}b,\varpi^{i}t(bs-y),\\
t^{-1}(xb+y\rho^2a),t[(xb+y\rho^2a)s-(xy+z\rho)])] \\
\abs{t}^2\ d^*t \ ds\ \psi[x+y+2a(xy-z)y\rho^2-2bxz\rho] \ dx\ dy\ dz.
\end{multline*}
Changing the order of integration, we find that
\[
I(a,b;i,j)=\int \abs{t}^2 \psi[x+y+2a(xy-z)y\rho^2-2bxz\rho] \ d^*t \ ds\ dx\ dy\ dz
\]
where the domain of integration is
\begin{multline*}
\max(\abs{a_j},\abs{b_i})\le \abs{t}\le q^{i+j}, \abs{x+as}\le q^j\abs{t}^{-1},\abs{y-bs}\le q^i\abs{t}^{-1}, \\ \abs{bx+\rho^2 a y}\le \abs{t},
\abs{z+\rho^2 xy-\rho^2 s(bx+\rho^2 a y)}\le \abs{t}^{-1}.
\end{multline*}
It follows immediately that 
\begin{equation}\label{vanishing I(a,b;i,j)}
I(a,b;i,j)=0\quad \text{when}\quad \max(|a_{i+2j}|,|b_{2i+j}|)>1,
\end{equation}
so from now on we suppose that
$\max(|a_{i+2j}|,|b_{2i+j}|)\leq1$.

The domain of integration above is rather complicated. After some variable changes, we will break it into pieces over which the integration may
be carried out.  In doing so, complicated sums of the cubic exponential sums $\Cu$ and $\Cu^*$ will naturally arise.  Then we will explain how to add
these pieces and simplify in order to obtain a closed formula for the relative orbital integral that allows us to compare it to the metaplectic side.

\subsection{A symmetry for the orbital integral}
Our next step is to observe a symmetry satisfied by the orbital integral under consideration.
Making the variable change $z\mapsto z-\rho^2 xy+\rho^2 s(bx+\rho^2 a y)$ and simplifying, we find that $I(a,b;i,j)$ is equal to
\[
\int \abs{t}^2 \psi[x+y+2bx^2y-2axy^2-2s(b^2x^2+ a^2 y^2-abxy)] \int \psi(-2\rho(bx+\rho ay)z)\ dz\ d^*t \ ds\ dx\ dy
\]
where the inner integral is over $\abs{z}\le \abs{t}^{-1}$ and the outer integral is over
\[
\max(\abs{a_j},\abs{b_i})\le \abs{t}\le q^{i+j}, \abs{x+as}\le q^j\abs{t}^{-1},\abs{y-bs}\le q^i\abs{t}^{-1}, \\ \abs{bx+\rho^2 a y}\le \abs{t}.
\]
After integrating over $z$ the integral becomes
\[
\int \abs{t} \psi[x+y+2bx^2y-2axy^2-2s(b^2x^2+ a^2 y^2-abxy)] \ d^*t \ ds\ dx\ dy
\]
over the domain 
\begin{align*}
&\max(\abs{a_j},\abs{b_i})\le \abs{t}\le q^{i+j}, \abs{x+as}\le q^j\abs{t}^{-1},\abs{y-bs}\le q^i\abs{t}^{-1},\\& \abs{bx+\rho^2 a y}\le \abs{t}, \abs{bx+\rho a y}\le \abs{t}
\end{align*}
 or equivalently
\[
\max(\abs{a_j},\abs{b_i})\le \abs{t}\le q^{i+j}, \abs{x+as}\le q^j\abs{t}^{-1},\abs{y-bs}\le q^i\abs{t}^{-1}, \abs{bx}\le \abs{t}, \abs{a y}\le \abs{t}.
\]
Interchanging $x$ and $y$ this gives the functional equation
\begin{equation}\label{eq feq}
I(a,b;i,j)=I(-b,-a;j,i).
\end{equation}

We now reduce to an integral in 3 variables.  Let 
\[
\I(a,b;i,j)=q^j \abs{a^2b}^{-1}\int  \psi[b^{-1}x+2a^{-1}b^{-1} x^3+a^{-1}y] \ d^*t \ dx\ dy 
\]
where the integral is over the domain $D$ defined by 
\[
\abs{a_j}\le \abs{t}\le q^{i+j},  \abs{x},\abs{y}\le \abs{t}, \abs{x+y}\le q^i\abs{at^{-1}}, \abs{(x+\rho y)(x+\rho^2 y)} \le q^{-j} \abs{at}.
\]
\begin{lemma}\label{lem Itxy}
We have
\begin{equation}\label{eq I-intcases}
I(a,b;i,j)=\begin{cases}
\I(a,b;i,j) & \abs{b_i}\le \abs{a_j} \\
\I(-b,-a;j,i) & \abs{a_j}\le \abs{b_i}.
\end{cases}
\end{equation}
\end{lemma}
\begin{proof}
By \eqref{eq feq} it suffices to prove \eqref{eq I-intcases} when $\abs{b_i}\le \abs{a_j}$. Suppose that this is the case.
After the variable change $x\mapsto b^{-1}x$ and $y\mapsto a^{-1}y$ we have
\begin{multline*}
I(a,b;i,j)=\\
\abs{ab}^{-1}\int \abs{t} \psi[b^{-1}x+a^{-1}y+2a^{-1}b^{-1}x^2y-2a^{-1}b^{-1}xy^2-2s(x^2+ y^2-xy)] \ d^*t \ ds\ dx\ dy
\end{multline*}
over the domain 
\[
\abs{a_j}\le \abs{t}\le q^{i+j}, \abs{x+abs}\le q^j\abs{bt^{-1}},\abs{y-abs}\le q^i\abs{at^{-1}}, \abs{x}\le \abs{t}, \abs{y}\le \abs{t}.
\]
After the variable change $s\mapsto s-(ab)^{-1}x$ the two conditions 
\[
\abs{x+abs}\le q^j\abs{bt^{-1}},~\abs{y-abs}\le q^i\abs{at^{-1}}
\] become 
\[
\abs{a s}\le q^j\abs{t}^{-1},~\abs{x+y-abs}\le q^i\abs{at^{-1}}.
\]
Note that the assumption that $\abs{b_i}\le \abs{a_j}$ together with the condition $\abs{a s}\le q^j\abs{t}^{-1}$ implies that $\abs{abs}\le q^i\abs{at^{-1}}$ and 
consequently we have
\[
I(a,b;i,j)=\abs{ab}^{-1}\int \abs{t} \psi[a^{-1}y+b^{-1}x+2a^{-1}b^{-1}x^3-2s(x^2+ y^2-xy)] \ d^*t \ ds\ dx\ dy
\]
over the domain 
\[
\abs{a_j}\le \abs{t}\le q^{i+j}, \abs{a s}\le q^j\abs{t}^{-1}, \abs{x+y}\le q^i\abs{at^{-1}}, \abs{x}\le \abs{t}, \abs{y}\le \abs{t}.
\]
Integrating over $s$ we get that
\[
I(a,b;i,j)=q^j\abs{a^2b}^{-1}\int  \psi[a^{-1}y+b^{-1}x+2a^{-1}b^{-1}x^3] \ d^*t \ dx\ dy
\]
over the domain 
\[
\abs{a_j}\le \abs{t}\le q^{i+j}, \abs{x+y}\le q^i\abs{at^{-1}}, \abs{x}\le \abs{t}, \abs{y}\le \abs{t}, \abs{x^2+y^2-xy}\le \abs{a_jt}.
\]
Observing that $x^2+y^2-xy=(x+\rho y)(x+\rho^2 y)$ the lemma follows.
\end{proof}

By Lemma~\ref{lem Itxy}, to evaluate the orbital integrals, we must evaluate $\I(a,b;i,j)$ when $\abs{b_i}\le \abs{a_j}$.  
In view of \eqref{vanishing I(a,b;i,j)} our concern is thus the case $|b_{2i+j}|\leq |a_{i+2j}|\leq1$.
We now carry out the evaluation of $\I(a,b;i,j)$ in this case.

\subsection{A decomposition of the domain of integration}
\subsubsection{The decomposition}
Integrating separately over $\abs{t}=q^{i+j-k}$ and summing over $k$ we write
\begin{equation}\label{eq sum over abst}
\I(a,b;i,j)=q^j\abs{a^2b}^{-1}\sum_{k=0}^{i+2j+\val(a)}\I_k(a,b;i,j)
\end{equation}
where
\[
\I_k(a,b;i,j)=\int \psi[a^{-1}y+b^{-1}x+2a^{-1}b^{-1}x^3]  \ dx\ dy
\]
over the domain $D_k$ defined by
\[
\abs{x+y}\le q^{k-j}\abs{a}, \abs{x}\le q^{i+j-k}, \abs{y}\le q^{i+j-k},\abs{(x+\rho y)(x+\rho^2 y)} \le q^{i-k} \abs{a}.
\]

Now we break the domain $D_k$ into pieces. Define the following domains:
\begin{itemize}
\item $D_k^\circ$ is defined by $\abs{x},\abs{y}\le q^{\lfloor \frac{i-(k+\val(a))}2\rfloor}$
\item $D_{k,t}[\ell]$  is defined by $\abs{x}=q^t,\ \abs{y+\rho^\ell x}\le q^{i-k-t-\val(a)}$ and 
\item $D_k^\dagger$ is defined by $\abs{x+y}\le q^{k-j}\abs{a},\abs{x}\le q^{\lfloor \frac{i-(k+\val(a))}2\rfloor}$.
\end{itemize}
Let
\[
v=v_{i,j}(a)=i+2j+\val(a).
\]
Since 
\[
\lfloor \frac{i-(k+\val(a))}2 \rfloor\le i+j-k,\ \ \ k\le v,
\] 
\[
\min(k-j-\val(a),i+j-k)=\begin{cases}
k-j-\val(a) & 2k\le v \\
i+j-k & 2k\ge v
\end{cases}
\]
and 
\[
\min(q^{k-j}\abs{a},q^{\lfloor \frac{i-(k+\val(a))}2\rfloor})=
\begin{cases}
q^{k-j}\abs{a}& 3k\le v \\
q^{\lfloor \frac{i-(k+\val(a))}2\rfloor} & 3k+1 \ge v
\end{cases}
\]
we deduce that
\begin{equation}\label{eq break domain}
D_k=\begin{cases}
D_k^\circ\sqcup \sqcup_{\ell=1}^2 \sqcup_{t=\lfloor \frac{i-(k+\val(a))}2\rfloor+1}^{i+j-k} D_{k,t}[\ell] & \lfloor \frac{v+1}2\rfloor\le k\le v \\
& \\
D_k^\circ\sqcup \sqcup_{\ell=1}^2 \sqcup_{t=\lfloor \frac{i-(k+\val(a))}2\rfloor+1}^{k-j-\val(a)} D_{k,t}[\ell] & \lfloor\frac{v}3\rfloor+1\le k\le \lfloor\frac{v-1}2\rfloor \\
& \\
D_k^\dagger & 0\le k\le \lfloor\frac{v}3\rfloor.
\end{cases}
\end{equation}

\subsubsection{Additional notation for cubic exponential sums}\label{additional-sums}
Below we will make use of the exponential sums in Section~\ref{Cubic exponential notation}
in specific instances, and it is convenient to introduce a shorthand for them.  It will also be efficient to 
replace $k$ by $q^k$ in specifying the domain of integration. Accordingly, 
for fixed $a,b\in F^*$ that we suppress from the notation and for $k\in\Z$, we introduce 
$$\Cu(q^k):=\Cu(b^{-1},2a^{-1}b^{-1};k),\quad \Cu^*(q^k):=\Cu^*(b^{-1},2a^{-1}b^{-1};k)$$
and for $\ell=0,1,2$ (or $\ell\in \Z$),
$$\Cu(q^k)[\ell]:=\Cu(b^{-1}-\rho^\ell a^{-1},2a^{-1}b^{-1};k),\quad \Cu^*(q^k)[\ell]:=\Cu^*(b^{-1}-\rho^\ell a^{-1},2a^{-1}b^{-1};k).$$
We write $\tilde{\Cu}$ and $\tilde{\Cu^*}$ for the same sums but with the pair $(a,b)$ replaced by $(-b,-a)$.
We note that for any $k$, using \eqref{Cu and units} it follows that $\tilde{\Cu}(q^k)[\ell]=\Cu(q^k)(-\ell)$, so
\begin{equation}\label{C sum vs tilde C sum}
\sum_{\ell=0}^2 \Cu(q^k)[\ell]=\sum_{\ell=0}^2 \tilde{\Cu}(q^k)[\ell].
\end{equation}

The following evaluations will be used later.
\begin{lemma}\label{special values of Cu for later}
\begin{enumerate} Suppose $m,n\geq0$.  
\item  If $|a_{2m+n}|=1$, then
$$\Cu(q^{m+n})=\begin{cases} q^{m+n}&\text{when~}|b_{m+2n}|=1\\
q^{m+n-1}(\g_{2a^{-1}b^{-1}}+\overline{\g}_{2a^{-1}b^{-1}})&\text{when~$|b_{m+2n}|=q^{-1}$ and $n\geq1$}\\
q^{m+n-1}&\text{when~$|b_{m+2n}|=q^{-2}$ and $n\geq1$.}
\end{cases}$$
\item If $|a_{2m+n}|=q^{-1}$, then
\begin{multline*}\Cu(q^{m+n-1})=\\
\begin{cases} q^{m+n-1}&\text{when~$|b_{m+2n}|=q^{-1}$ or (\ $|b_{m+2n}|=q^{-2}$ and $n\ge1$)}\\
q^{m+n-2}(\g_{2a^{-1}b^{-1}}+\overline{\g}_{2a^{-1}b^{-1}})&\text{when~$|b_{m+2n}|=q^{-3}$ and $n\geq2$.}
\end{cases}
\end{multline*}
\end{enumerate}
\end{lemma}

\begin{proof} Suppose that $|a_{2m+n}|=1$. If $|b_{m+2n}|=1$, then $|b|^{-1}q^{m+n}=q^{-n}\leq1$ and $|a^{-1}b^{-1}|q^{3m+3n}=1$, hence
the integrand of $\Cu(q^{m+n})$ is identically 1 and we obtain $q^{m+n}$.  If $|b_{m+2n}|=q^{-1}$and $n\geq1$, the evaluation shown follows from \eqref{special Cu sum}.
If $|b_{m+2n}|=q^{-2}$ and $n\geq1$, we use \eqref{C and Cstar} to write $\Cu(q^{m+n})=\Cu^*(q^{m+n})+\Cu(q^{m+n-1})$. Since
 $|b^{-1}|q^{m+n}=q^{2-n}\leq q$ while $|a^{-1}b^{-1}|q^{3(m+n)}=q^2$, we see from the vanishing criterion \eqref{vanishing conditions for Cu star}
 that $\Cu^*(q^{m+n})=0$.  For $\Cu(q^{m+n-1})$ the integrand is identically 1 and so we obtain $q^{m+n-1}$.
If instead $|a_{2m+n}|=q^{-1}$, then when~$|b_{m+2n}|=q^{-1}$ or ($|b_{m+2n}|=q^{-2}$ and $n\ge1$), we have
 $|b^{-1}|q^{m+n-1}\leq1$ and $|a^{-1}b^{-1}|q^{3(m+n-1)}\leq1$.  Thus the
integrand is again constant and the evaluation follows.  If $|b_{m+2n}|=q^{-3}$ then $|b^{-1}|q^{m+n-1}=q^{2-n}$ and $|a^{-1}b^{-1}|q^{3(m+n-1)}=q$.  When $n\geq2$ the evaluation
then follows from \eqref{special Cu sum}.
\end{proof}

\subsubsection{Using the decomposition to write the integral as a sum of terms}
Applying the integration formula
\[
\int_{z+\p^{-m}}\psi(a^{-1}y)\ dy=\begin{cases}
q^m \psi(a^{-1}z) & q^m\le \abs{a}\\
0 & \text{otherwise}
\end{cases}
\]
we obtain the evaluations
\[
\int_{D_k^\circ} \psi[a^{-1}y+b^{-1}x+2a^{-1}b^{-1}x^3]  \ dx\ dy=\\
\begin{cases}
q^{\lfloor \frac{i-(k+\val(a))}2\rfloor}\Cu(q^{\lfloor \frac{i-(k+\val(a))}2\rfloor}) & i+\val(a)-1\le k \\
0 & k \le i+\val(a)-2,
\end{cases}
\]
\[
\int_{D_{k,t}[\ell]} \psi[a^{-1}y+b^{-1}x+2a^{-1}b^{-1}x^3]  \ dx\ dy
=\begin{cases}
\abs{a}q^{i-k-t}\Cu^*(q^t)[\ell] & i-k\le t \\
0 & t \le i-k-1
\end{cases}
\]
and 
\[
\int_{D_k^\dagger} \psi[a^{-1}y+b^{-1}x+2a^{-1}b^{-1}x^3]  \ dx\ dy=\\
\begin{cases}
q^{k-j}\abs{a}\Cu(q^{\lfloor \frac{i-(k+\val(a))}2\rfloor})[0] &  k\le j \\
0 & j+1\le k.
\end{cases}
\]
For example, for the first integration above the $y$ integral gives zero unless $q^{\lfloor \frac{i-(k+\val(a))}2\rfloor}\leq |a|$ and
this is easily seen to be equivalent to the condition $i+\val(a)-1\le k$.

Observe that
\[
\lfloor\frac{i-(k+\val(a))}2\rfloor=\lfloor\frac{v-k}2\rfloor-j-\val(a).
\]
Combining this with \eqref{eq sum over abst} and \eqref{eq break domain}, we conclude that
\begin{multline*}
\abs{ab}\I(a,b;i,j)=
\sum_{k=\max(\lfloor \frac{v}3\rfloor+1,i+\val(a)-1)}^v q^{\lfloor \frac{v-k}2\rfloor}\Cu(q^{\lfloor \frac{v-k}2\rfloor-j}\abs{a})+\\
\sum_{k=0}^{\min(\lfloor\frac{v}3\rfloor,j)}q^k\Cu(q^{\lfloor \frac{v-k}2\rfloor-j}\abs{a})[0] +\sum_{\ell=1}^2\sum_{k=\lfloor\frac{v}3\rfloor+1}^{\lfloor\frac{v-1}2\rfloor } \sum_{t=\max(\lfloor \frac{i-(k+\val(a))}2\rfloor+1,i-k)}^{k-j-\val(a)} q^{i+j-(k+t)}\Cu^*(q^t)[\ell]+\\
\sum_{\ell=1}^2\sum_{k=\lfloor\frac{v+1}2\rfloor}^{v} \sum_{t=\max(\lfloor \frac{i-(k+\val(a))}2\rfloor+1,i-k)}^{i+j-k} q^{i+j-(k+t)}\Cu^*(q^t)[\ell].
\end{multline*}
We write
\[
\abs{ab}\I(a,b;i,j)=A+B[1]+B[2]
\]
where $A=A_1+A_2$ with
\[
A_1=\sum_{k=\max(\lfloor \frac{v}3\rfloor+1,v-2j-1)}^v q^{\lfloor \frac{v-k}2\rfloor}\Cu(q^{\lfloor \frac{v-k}2\rfloor-j}\abs{a}),\  A_2=
\sum_{k=0}^{\min(\lfloor\frac{v}3\rfloor,j)}q^k \Cu(q^{\lfloor \frac{v-k}2\rfloor-j}\abs{a})[0]
\]
and for $\ell=1,2$,
\begin{multline*}
B[\ell]=\\
\left(\sum_{k=\lfloor\frac{v}3\rfloor+1}^{\lfloor\frac{v-1}2\rfloor } \sum_{t=\max(\lfloor \frac{i-(k+\val(a))}2\rfloor+1,i-k)}^{k-j-\val(a)} +
\sum_{k=\lfloor\frac{v+1}2\rfloor}^{v} \sum_{t=\max(\lfloor \frac{i-(k+\val(a))}2\rfloor+1,i-k)}^{i+j-k} \right)q^{i+j-(k+t)}\Cu^*(q^t)[\ell].
\end{multline*}
Note further that for $k=v$, we have $\max(\lfloor \frac{i-(k+\val(a))}2\rfloor+1,i-k)=i+j-k+1$, so this index of summation does not contribute to $B[\ell]$.
\subsection{Evaluation of the contribution $A$}
Changing the index $k\mapsto v-k$ we find that 
\[
A_1=\sum_{k=0}^{\min(v-1-\lfloor\frac{v}3\rfloor,2j+1)} q^{\lfloor\frac{k}2\rfloor}\Cu(q^{\lfloor\frac{k}2\rfloor-j}\abs{a}).
\]
Observe that 
\[
\min\left(v-1-\lfloor\frac{v}3\rfloor,2j+1\right)=\begin{cases} 2j+1 & v\ge 3j+2 \\
v-1-\lfloor\frac{v}3\rfloor & v\le 3j+1.
\end{cases}
\]
Further,  $v-1-\lfloor\frac{v}3\rfloor\equiv0\bmod 2$ if and only if $v\equiv1\bmod 3$ and in that case
\[
\frac12(v-1-\lfloor\frac{v}3\rfloor)=\frac{v-1}3.
\]
Also,
\[
2\lfloor \frac{v-2}3 \rfloor+1=\begin{cases}
v-2-\lfloor\frac{v}3\rfloor & v\equiv1  \bmod 3\\
v-1-\lfloor\frac{v}3\rfloor & \text{otherwise.}
\end{cases}
\]
Since
\[ 
\sum_{k=0}^{2s+1}q^{\lfloor\frac{k}2\rfloor}\Cu(q^{\lfloor\frac{k}2\rfloor-j}\abs{a})=2\sum_{k=0}^s q^k \Cu(q^{k-j}\abs{a})
\]
we conclude that
\[
A_1=\begin{cases}
2\sum_{k=0}^j q^k \Cu(q^{k-j}\abs{a}) & v\ge 3j+2 \\
\delta(v\equiv1\bmod3)q^{\frac{v-1}3}\Cu(q^{\frac{v-1}3-j}\abs{a})+2\sum_{k=0}^{\lfloor \frac{v-2}3\rfloor} q^k \Cu(q^{k-j}\abs{a}) & v\le 3j+1
\end{cases}
\]
where here
$$\delta(v\equiv1\bmod3)=\begin{cases}1&v\equiv1\bmod 3\\0&\text{otherwise.}\end{cases}$$

As for $A_2$, note that
\[
\min(\lfloor\frac{v}3\rfloor,j)=\begin{cases} j & v\ge 3j+2 \\\lfloor\frac{v}3\rfloor & v\le 3j+1. \end{cases}
\]
Thus we have
\[
A_2=\begin{cases} \sum_{k=0}^j q^k \Cu(q^{\lfloor \frac{v-k}2\rfloor-j}\abs{a})[0] & v\ge 3j+2 \\ \sum_{k=0}^{\lfloor\frac{v}3\rfloor}q^k \Cu(q^{\lfloor \frac{v-k}2\rfloor-j}\abs{a})[0] & v\le 3j+1. \end{cases}
\]

\subsection{Evaluation of the contributions $B[\ell]$}
After the index change $t\mapsto i+j-k-t$ we have
\[
B[\ell]=\left(\sum_{k=\lfloor\frac{v}3\rfloor+1}^{\lfloor\frac{v-1}2\rfloor } \sum_{t=v-2k}^{\min(\lfloor \frac{v-k-1}2\rfloor,j)}+
\sum_{k=\lfloor\frac{v+1}2\rfloor}^{v-1} \sum_{t=0}^{\min(\lfloor \frac{v-k-1}2\rfloor,j)} \right)q^t\Cu^*(q^{v-k-t-j}\abs{a})[\ell]=
\]
\[
\sum_{t=0}^j q^t \sum_{k\in S_1(t)\sqcup S_2(t)} \Cu^*(q^{v-k-t-j}\abs{a})[\ell]
\]
where 

\[
S_1(t)=\{k\in\Z\colon \lfloor\frac{v}3\rfloor+1\le k\le \lfloor\frac{v-1}2\rfloor\text{ and } v-2k\le t\le \lfloor \frac{v-k-1}2\rfloor\} 
\]
and 
\[
S_2(t)=\{k\in\Z\colon \lfloor\frac{v+1}2\rfloor\le k\le v-1\text{ and }   t\le \lfloor \frac{v-k-1}2\rfloor \}.
\]

Explicitly, these sets are the integer intervals 
\[
S_1(t)=\left[\max\left(\lfloor\frac{v}3\rfloor+1,\lfloor\frac{v+1-t}2\rfloor\right), \min\left(\lfloor\frac{v-1}2\rfloor,v-1-2t\right)\right]
\]
and
\[
S_2(t)=\left[\lfloor\frac{v+1}2\rfloor,v-1-2t\right].
\]
Since $\lfloor\frac{v+1-t}2\rfloor\le v-1-2t$ if and only if $3t+2\le v$ the interval $S_1(t)$ is empty unless $3t+2\le v$ and when this is the case we have 
\[
\max\left(\lfloor\frac{v}3\rfloor+1,\lfloor\frac{v+1-t}2\rfloor\right)=\lfloor\frac{v+1-t}2\rfloor.
\]
Note further that $S_2(t)$ is non-empty if and only if $4t+2\le v$ and that it is empty if and only if $\min(\lfloor\frac{v-1}2\rfloor,v-1-2t)=v-1-2t$. 
It follows that
\[
S_1(t)\sqcup S_2(t)=\begin{cases} [\lfloor\frac{v+1-t}2\rfloor,v-1-2t]& t\le \lfloor\frac{v-2}3\rfloor \\
\emptyset & \text{otherwise.}
\end{cases}
\]

We conclude that
\[
B[\ell]=\sum_{t=0}^{\min(j, \lfloor\frac{v-2}3\rfloor)}q^t\sum_{k=\lfloor\frac{v+1-t}2\rfloor}^{v-1-2t}\Cu^*(q^{v-k-t-j}\abs{a})[\ell]=\sum_{t=0}^{\min(j, \lfloor\frac{v-2}3\rfloor)}q^t\sum_{k=t+1}^{\lfloor\frac{v-t}2 \rfloor}\Cu^*(q^{k-j}\abs{a})[\ell].
\]
The last equality is obtained after the change $k\mapsto v-t-k$. 

Making use of \eqref{C and Cstar} and noting that the sum telescopes, we conclude that
\[
B[\ell]=\sum_{t=0}^{\min(j, \lfloor\frac{v-2}3\rfloor)}q^t\left(\Cu(q^{\lfloor\frac{v-t}2 \rfloor-j}\abs{a})[\ell]-\Cu(q^{t-j}\abs{a}) \right).
\]
Since 
\[
\min(j, \lfloor\frac{v-2}3\rfloor)=\begin{cases}
j& v\ge 3j+2 \\
\lfloor\frac{v-2}3\rfloor & v\le 3j+1
\end{cases}
\]
we find that 
\[
B[\ell]=
\begin{cases} 
\sum_{t=0}^j q^t \left(\Cu(q^{\lfloor\frac{v-t}2 \rfloor-j}\abs{a})[\ell]-\Cu(q^{t-j}\abs{a}) \right)& v\ge 3j+2\\\
\sum_{t=0}^{\lfloor\frac{v-2}3\rfloor} q^t \left(\Cu(q^{\lfloor\frac{v-t}2 \rfloor-j}\abs{a})[\ell]-\Cu(q^{t-j}\abs{a}) \right)& v\le 3j+1.
\end{cases}
\]
Note also that 
\[
q^{\lfloor\frac{v-k}2 \rfloor-j}\abs{a}=q^{\lfloor\frac{i-\val(a)-k}2 \rfloor}.
\]

\subsection{Evaluation of $\I(a,b;i,j)$}
Putting everything together, 
for all $a,b\in F^\times$, $i,j\geq0$ such that $\max(|a_{i+2j}|,|b_{2i+j}|)\leq1$, we have
\begin{multline}\label{eq longIij}
\abs{ab}\I(a,b;i,j)=\\
\begin{cases}
\sum_{\ell=0}^2 \sum_{k=0}^j q^k  \Cu(q^{\lfloor\frac{i-\val(a)-k}2 \rfloor})[\ell] & 3j+2\le v \\& \\
\sum_{\ell=0}^2 \sum_{k=0}^{\lfloor\frac{v-2}3\rfloor}q^k \Cu(q^{\lfloor\frac{i-\val(a)-k}2 \rfloor})[\ell]  & 0\le v\le 3j+1 \text{ and }v\equiv2\bmod 3\\ & \\
\sum_{\ell=0}^2 \sum_{k=0}^{\lfloor\frac{v-2}3\rfloor}q^k \Cu(q^{\lfloor\frac{i-\val(a)-k}2 \rfloor})[\ell] +2q^{\frac{v-1}3} \Cu(q^{\frac{i-j-2\val(a)-1}3})[\ell] & 0\le v\le 3j+1 \text{ and }v\equiv1\bmod 3\\ & \\
\sum_{\ell=0}^2 \sum_{k=0}^{\lfloor\frac{v-2}3\rfloor}q^k \Cu(q^{\lfloor\frac{i-\val(a)-k}2 \rfloor})[\ell] +q^{\frac{v}3} \Cu(q^{\frac{i-j-2\val(a)}3})[\ell] & 0\le v\le 3j+1 \text{ and }v\equiv0\bmod 3.
\end{cases}
\end{multline}

We simplify this expression.
It follows from \eqref{vanishing condition for Cu} and \ref{eq another Cu vanishing} that if $\abs{b}<\abs{a}$ and $t\in \Z$ then 
$\Cu(q^t)[\ell]=0$ and $\Cu(q^t)=0$ if 
\begin{itemize}
\item either $\abs{b}^{-1}q^t\ge q$ and $\abs{a}^{-1}q^{2t}<1$
\item or $\abs{b}^{-1}q^t\ge q^2$ and $\abs{a}^{-1}q^{2t}=q$.
\end{itemize}
Based on these vanishing properties we deduce that whenever $\abs{b}<\abs{a}$ we have
\begin{equation}\label{eq c=0}
\Cu(q^{\lfloor\frac{-\val(a)-t}2 \rfloor})[\ell]=0=\Cu(q^{\lfloor\frac{-\val(a)-t}2 \rfloor}) 
\end{equation}
if either of the following two conditions holds:
\begin{equation}\label{eq cond1}
\abs{ab^{-2}}\ge q^{t+2} \text{ and either }t\ge 1 \text{ or }(t=0 \text{ and } \val(a)\equiv 1 \bmod 2)
\end{equation}
or\begin{equation}\label{eq cond2}
\abs{ab^{-2}}\ge q^3, t=-1 \text{ and } \val(a)\equiv 1\bmod 2.
\end{equation}
Similarly, if $\epsilon\in \{0,1\}$ is such that $v\equiv \epsilon\bmod 3$ then 
\begin{equation}\label{eq c=0 over 3}
\Cu(q^{\frac{i-j-2\val(a)-\epsilon}3})[\ell]=0 
\end{equation}
whenever
\begin{equation}\label{eq cond3}
\abs{a^2b^{-3}}\ge q^{3+\epsilon+j-i} \text{ and } v-3i\ge 3-2\epsilon.
\end{equation}

Using the above vanishing properties and others to be explained momentarily,
we give the following evaluation of $\I(a,b;i,j)$.  We also study the difference 
\[
\I(a,b;i,j)-q\I(a,b;i-1,j-1), 
\]
a step that will be useful
in allowing us to obtain formulas for the orbital integrals by systematically taking advantage of cancellations. 
\begin{proposition}\label{prop Iij formula}
Let $i,j\in \Z_{\ge 0}$ and $a,b\in F^*$ be such that $\abs{b_{j+2i}}\le \abs{a_{i+2j}}\le 1$. Set $v=i+2j+\val(a)$.
\begin{enumerate}
\item Then
\begin{multline*}
\abs{ab}\I(a,b;i,j)=
\\
\begin{cases}
\sum_{\ell=0}^2 \sum_{k=0}^{\min(i,j)} q^k \Cu(q^{\lfloor\frac{i-\val(a)-k}2 \rfloor})[\ell] & 3\min(i,j)+2\le v \\& \\
\sum_{\ell=0}^2 \sum_{k=0}^{\lfloor\frac{v-2}3\rfloor}q^k \Cu(q^{\lfloor\frac{i-\val(a)-k}2 \rfloor})[\ell] & 0\le v\le 3\min(i,j)+1\text{ and } v\equiv2\bmod 3 \\ & \\
\sum_{\ell=0}^2 \sum_{k=0}^{\lfloor\frac{v-2}3\rfloor}q^k \Cu(q^{\lfloor\frac{i-\val(a)-k}2 \rfloor})[\ell]+2q^{\frac{v-1}3} \Cu(q^{\frac{i-j-2\val(a)-1}3}) & 0\le v\le 3\min(i,j)+1 \text{ and }v\equiv1\bmod 3\\ & \\
\sum_{\ell=0}^2 \sum_{k=0}^{\lfloor\frac{v-2}3\rfloor}q^k \Cu(q^{\lfloor\frac{i-\val(a)-k}2 \rfloor})[\ell] +q^{\frac{v}3} \Cu(q^{\frac{i-j-2\val(a)}3})& 0\le v\le 3\min(i,j)+1 \text{ and }v\equiv0\bmod3.
\end{cases}
\end{multline*}
\item Set
\[
\RI(a,b;i,j)=\begin{cases}\I(a,b;i,j)-q\I(a,b;i-1,j-1) & \min(i,j)\ge 1 \\
\I(a,b;i,j)& \min(i,j)=0.
\end{cases}
\]
Then
\[
\abs{ab}\RI(a,b;i,j)=\begin{cases}
\sum_{\ell=0}^2 \Cu(q^{\lfloor\frac{i-\val(a)}2 \rfloor})[\ell] & 2\le v \\
2\Cu(q^{-j}\abs{a}) & v=1 \\ \Cu(q^{-j}\abs{a}) & v=0.
\end{cases}
\]
\end{enumerate}
\end{proposition}
\begin{proof}
The second part is a straightforward consequence of the first. If either $j\le i$ or $v\le 3i+1$ then the first part is \eqref{eq longIij}. 
For the rest of the proof we assume that $i<j$  and $v\ge 3i+2$. We prove the first part by employing the vanishing conditions above for cubic exponentials.

Note that by assumption $\abs{b}\le q^{i-j}\abs{a}$ and in particular that $\abs{b}<\abs{a}$.
We therefore have
\begin{equation}\label{eq aid ineq}
\abs{ab^{-2}}\ge q^{2(j-i)}\abs{a}^{-1}=q^{v-3i}.
\end{equation}
If $v\ge 3j+2\ge 2i+j+2$, then $q^{v-3i} \ge q^{j-i+2}$ and it follows from \eqref{eq aid ineq}
that the condition \eqref{eq cond1} holds for $t=k-i$ for any $i+1\le k\le j$. We conclude from \eqref{eq c=0} that
\[
\Cu(q^{\lfloor\frac{i-\val(a)-k}2 \rfloor})[\ell] =0, \ \ \ i+1\le k\le j.
\]
If $3i+2\le v\le 3j+1$, then $v-3i\ge \lfloor \frac{v-2}3\rfloor-i+2$ and it similarly follows that condition \eqref{eq cond1} holds for $t=k-i$ for any $i+1\le k\le  \lfloor \frac{v-2}3\rfloor$. 
We conclude from \eqref{eq c=0} that
\[
\Cu(q^{\lfloor\frac{i-\val(a)-k}2 \rfloor})[\ell] =0, \ \ \ i+1\le k\le  \lfloor \frac{v-2}3\rfloor.
\]
Assume further that $v\equiv \epsilon\bmod 3$ for $\epsilon\in \{0,1\}$. In particular then $3i+3+\epsilon\le v$.
We have
\[
\abs{a^2b^{-3}}\ge q^{3(j-i)}\abs{a}^{-1}=q^{v+j-4i}\ge q^{3+\epsilon+j-i}
\]
and therefore condition \eqref{eq cond3} holds. It follows from \eqref{eq c=0 over 3} that $\Cu(q^{\frac{i-j-2\val(a)-\epsilon}3})[\ell]=0$.

Using these vanishing statements together with \eqref{eq longIij}, it follows that $\abs{ab}\I(a,b;i,j)=\sum_{\ell=0}^2 \sum_{k=0}^i q^k \Cu(q^{\lfloor\frac{i-\val(a)-k}2 \rfloor})[\ell]$ whenever $3i+2\le v$.
The proposition follows.
\end{proof}

For convenience, we set $\RI(a,b;i,j)=0$ when the condition $\abs{b_{j+2i}}\le \abs{a_{i+2j}}\le 1$ does not hold.

\subsection{Evaluation of the relative orbital integral for the generic family of orbits}\label{sec: evaluation of I integral}
Let
$$I_{m,n}(a,b):=\O(\xi(a,b),\omega(f_{m,n})\phi_\circ)$$
be the orbital integral attached to the generic family of relevant orbits for the Hecke basis element $f_{m,n}$.  We shall present formulas for this integral in all cases. 

We begin with the following symmetry.
\begin{lemma}\label{prop symmetry G}
Let $a,b\in F^*$.  Then
 $I_{m,n}(a,b)=I_{n,m}(-b,-a)$.
 \end{lemma}
 \begin{proof} This follows immediately by combining Corollary~\ref{cor exp j-G} and the symmetry \eqref{eq feq}.\end{proof}
 
In view of Lemma~\ref{prop symmetry G}, Corollary~\ref{cor exp j-G}, and the vanishing property \eqref{vanishing I(a,b;i,j)},
it thus suffices to determine the integral in the case 
\begin{equation}\label{ab-inequalities}
\abs{b_{m+2n}}\le \abs{a_{n+2m}}\le 1.
\end{equation} 
We suppose that these inequalities are satisfied in the rest of the Section.

We have expressed $I_{m,n}(a,b)$ as a sum of the integrals of the form $I(a,b;i,j)$ in Corollary~\ref{cor exp j-G} and
each summand may be evaluated by combining Lemma~\ref{lem Itxy} and Proposition~\ref{prop Iij formula}.
To utilize these ingredients, 
we first evaluate the integrals $I_{m,0}$, $I_{0,n}$ and 
$I_{m,n}$ when $\min(m,n)=1$.  In each of these cases several terms in Corollary~\ref{cor exp j-G} are automatically zero.
Then we treat the case $\min(m,n)\geq2$. In all cases, we write the expressions arising from Corollary~\ref{cor exp j-G}  in terms of the differences $\RI$ introduced in
Proposition~\ref{prop Iij formula}, part (2).  This allows us to make systematic use of the cancellations established
there.

\subsubsection{Evaluation of $I_{m,0}(a,b)$}
We are in the case $|a|\leq q^{2m}$ and $\abs{b}\le q^{-m}\abs{a}$ by \eqref{ab-inequalities}. When $m=1$, we have 
$$
I_{1,0}(a,b)=\I(a,b;0,1)=\abs{ab}^{-1}\begin{cases}\sum_{\ell=0}^2 \Cu(q^{\lfloor-\frac{\val(a)}2\rfloor})[\ell] & \abs{a}\le 1 \\
2\Cu(q^{-1}\abs{a}) & \abs{a}=q \\
\Cu(q^{-1}\abs{a}) & \abs{a}=q^2.
\end{cases}
$$
Using the properties of the cubic exponential sums $\Cu$ in Section~\ref{Cubic exponential notation}, 
we rewrite this as follows.

\begin{lemma}\label{I when m=1, n=0} We have
$$|ab|I_{1,0}(a,b)=\begin{cases}q&|a_{2}|=|b_1|=1\\
\Cu(q)&|a_2|=1, |b_1|\le q^{-1}\\
2&|a_2|=q^{-1}, |b_1|=q^{-1}\\
3\Cu(1)&|a_2| = q^{-2}\\
\sum_{\ell=0}^2 \Cu(q^{\lfloor-\frac{\val(a)}2\rfloor})[\ell]&|a_2|\leq q^{-4}\\
0&otherwise.
\end{cases}$$
\end{lemma}

\begin{proof}
 If $|a_{2}|=1$, $|b_{1}|\leq1$, we have
$|ab| I_{1,0}(a,b)=\Cu(q).$
By Lemma~\ref{special values of Cu for later}, if $|b_{1}|=1$ then $\Cu(q)=q$.
If $|a_{2}|=q^{-1}$, then
$|ab|I_{1,0}(a,b)=2\Cu(1)$. But $\Cu(1)$ vanishes unless $|b_{1}|=|q|^{-1}$, in which case it is $1$.
If $|a_2|\leq q^{-2}$ then $|ab|I_{1,0}(a,b)=\sum_{\ell=0}^2 \Cu(q^{\lfloor-\frac{\val(a)}2\rfloor})[\ell].$
If $|a_{2}| = q^{-2}$,  then $\val(a)=0$ and the three summands are equal and independent of $\ell$.
If $|a_{2}| =  q^{-3}$ they are zero by the vanishing criterion \eqref{vanishing condition for Cu} for $\Cu$ in Section~\ref{Cubic exponential notation}. The lemma follows.
\end{proof}

For $m\ge 2$, using Corollary~\ref{cor exp j-G} , we have that
$$I_{m,0}(a,b)=I(a,b;0,m)-I(a,b;1,m-2)+qI(a,b;0,m-3).$$
Using
the relation between
the $I$ and $\I$ integrals given in
Lemma~\ref{lem Itxy} and
the definition of the difference $\RI$ in Proposition~\ref{prop Iij formula}, we find that $I_{m,0}(a,b)=\I(a,b;0,m)-\RI(a,b;1,m-2)$. 
Proposition~\ref{prop Iij formula} and the vanishing property \eqref{vanishing I(a,b;i,j)} then give
\[
\abs{ab}\I(a,b;0,m)=\\
\begin{cases}
\sum_{\ell=0}^2 \Cu(q^{\lfloor-\frac{\val(a)}2\rfloor})[\ell]& \abs{a}\le q^{2m-2} \\
2\Cu(q^{-m}\abs{a})& \abs{a}=q^{2m-1} \\
\Cu(q^{-m}\abs{a})& \abs{a}=q^{2m}
\end{cases}
\]
and
\[
\abs{ab}\RI(a,b;1,m-2)=
\begin{cases}
\sum_{\ell=0}^2 \Cu(q^{\lfloor\frac{1-\val(a)}2\rfloor})[\ell] & \abs{a}\le q^{2m-5} \\
2\Cu(q^{2-m}\abs{a}) & \abs{a}=q^{2m-4}\\
\Cu(q^{2-m}\abs{a}) & \abs{a}=q^{2m-3} \\
0 & \abs{a}\ge q^{2m-2}.
\end{cases}
\]
We conclude that for $m\ge 2$ we have
\[
\abs{ab}I_{m,0}(a,b)=\begin{cases}
\sum_{\ell=0}^2 \left( \Cu(q^{\lfloor-\frac{\val(a)}2\rfloor})[\ell]-\Cu(q^{\lfloor \frac{1-\val(a)}2\rfloor})[\ell] \right) & \abs{a}\le q^{2m-5} \\
\Cu(q^{m-2}) & \abs{a}=q^{2m-4}\\
3\Cu(q^{m-2}) -\Cu(q^{m-1})& \abs{a}=q^{2m-3} \\
\sum_{\ell=0}^2 \Cu(q^{\lfloor-\frac{\val(a)}2\rfloor})[\ell]& \abs{a}= q^{2m-2} \\
2\Cu(q^{-m}\abs{a})& \abs{a}=q^{2m-1} \\
\Cu(q^{-m}\abs{a})& \abs{a}=q^{2m}.
\end{cases}
\]

Some of these expressions can be simplified.
First, suppose that  either $\abs{a}\le q^{2m-5}$ or $\abs{a}=q^{2m-3}$. If $\val(a)$ is even then $\lfloor-\frac{\val(a)}2\rfloor=\lfloor \frac{1-\val(a)}2\rfloor$ and therefore 
\[
\Cu(q^{\lfloor-\frac{\val(a)}2\rfloor})[\ell]-\Cu(q^{\lfloor \frac{1-\val(a)}2\rfloor})[\ell]=0.
\]
If $\val(a)$ is odd, then since 
$\abs{b}\le q^{-m}\abs{a}$, it follows that $\abs{ab^{-2}}\ge q^{2m}\abs{a}^{-1} \ge q^3$ and so both condition \eqref{eq cond1} with $t=0$ and 
\eqref{eq cond2} with $t=-1$ hold. It follows from \eqref{eq c=0} that
\[
 \Cu(q^{\lfloor-\frac{\val(a)}2\rfloor})[\ell]=\Cu(q^{\lfloor \frac{1-\val(a)}2\rfloor})[\ell]=0. 
\]
Note further that if $\abs{a}=q^{2m-3}$ then $\lfloor \frac{1-\val(a)}2\rfloor=m-1$ and $\lfloor-\frac{\val(a)}2\rfloor=m-2$.
We conclude that if $m\geq2$ then the following evaluation holds.
\[
\abs{ab}I_{m,0}(a,b)=\begin{cases}
0 & \abs{a}\le q^{2m-5} \\
\Cu(q^{m-2}) & \abs{a}=q^{2m-4}\\
0 & \abs{a}=q^{2m-3} \\
\sum_{\ell=0}^2 \Cu(q^{\lfloor-\frac{\val(a)}2\rfloor})[\ell]& \abs{a}= q^{2m-2} \\
2\Cu(q^{-m}\abs{a})& \abs{a}=q^{2m-1} \\
\Cu(q^{-m}\abs{a})& \abs{a}=q^{2m}.
\end{cases}
\]

Now we use the properties of cubic exponential integrals to rewrite this formula, breaking out some cases
where the exponential sums simplify.  Doing so will ultimately 
allow us to combine the formulas that appear in the different cases for $m,n$.
In cases where the integrand in $\Cu(q^k)$ or $\Cu(q^k)[\ell]$ is identically one in the domain $\p^{-k}$ we often write that the 
cubic exponential integral equals $q^k$ without further justification. 

\begin{proposition}\label{I int (m,0), m ge 2} Suppose that $m\ge2$ and $|b_m|\leq |a_{2m}|$. Then
\[
\abs{ab}I_{m,0}(a,b)=\begin{cases}
q^m& \abs{a_{2m}}=\abs{b_m}=1\\
\Cu(q^m)& \abs{a_{2m}}=1, \abs{b_m}\leq q^{-1}\\
2q^{m-1}& \abs{a_{2m}}=|b_m|=q^{-1}\\
3\Cu(q^{m-1})& \abs{a_{2m}}= q^{-2} \\
\Cu(q^{m-2}) & \abs{a_{2m}}=q^{-4}\\
0 & otherwise.
\end{cases}
\]
\end{proposition}

\begin{proof} Once again we use the information from Section~\ref{Cubic exponential notation}.
If $\abs{a_{2m}}=\abs{b_m}=1$ then $\Cu(q^m)=q^m$.  If $\abs{a_{2m}}=|b_m|=q^{-1}$, then
$\Cu(q^{m-1})=q^{m-1}$.  If $\abs{a_{2m}}=q^{-1}$ and $|b_m|\leq q^{-2}$ then $|b|^{-1}\ge q^{2-m}$ and the criterion \eqref{vanishing condition for Cu} shows that
$\Cu(q^{m-1})=0$.  If $|a_{2m}|=q^{-2}$ then $|a|^{-1}q^{m-1}=q^{1-m}<1$ and it follows that $\Cu(q^{m-1})[\ell]=\Cu(q^{m-1})$ for each $\ell$.
Substituting these values into the previous formula, the result follows.
\end{proof}

\subsubsection{Evaluation of $I_{0,n}(a,b)$} 
We are in the case $\abs{a}\le q^n$ and
$\abs{b}\le q^n\abs{a}$ by \eqref{ab-inequalities}.
For $n=1$, we have
$$
\abs{ab}I_{0,1}(a,b)=\abs{ab}\I(a,b;1,0)=
\begin{cases}
\sum_{\ell=0}^2 \Cu(q^{\lfloor\frac{1-\val(a)}2\rfloor})[\ell]  & \abs{a}\le q^{-1} \\
2\Cu(\abs{a}) & \abs{a}=1 \\
 \Cu(\abs{a}) & \abs{a}=q.
\end{cases}
$$

We separate this out by cases as above.
\begin{lemma}\label{I integral with m=0 and n=1} We have
\begin{equation*}
\abs{ab}I_{0,1}(a,b)=
\begin{cases}
q& \abs{a_1}=|b_2|=1\\
\g_{2a^{-1}b^{-1}}+\overline{\g}_{2a^{-1}b^{-1}}& \abs{a_1}=1, |b_2|=q^{-1}\\
 1 & \abs{a_1}=1, |b_2|=q^{-2}\\
 2 & \abs{a_1}=q^{-1}, |b_2|\in\{q^{-1},q^{-2}\} \\
  2\Cu(1) & \abs{a_1}=q^{-1}, |b_2|\leq q^{-3} \\
3\tilde{\Cu}(1)& \abs{a_1}=|b_2|= q^{-2} \\
3q^{-1}& \abs{a_1}= q^{-2}, |b_2|=q^{-3} \\
\sum_{\ell=0}^2 \tilde{\Cu}(q^{\lfloor\frac{-\val(b)}2\rfloor})[\ell]  & \abs{a_1}=|b_2|\leq q^{-3} \\
\sum_{\ell=0}^2 \Cu(q^{\lfloor\frac{-\val(a)}2\rfloor})[\ell]  & |b_2|<\abs{a_1}\leq q^{-3} \\
0&otherwise
\end{cases}
\end{equation*}
\end{lemma}

\begin{proof} The values when $\abs{a_1}=1$, $|b_2|\in\{1,q^{-1}\}$ are given in Lemma~\ref{special values of Cu for later}.  If  $\abs{a_1}=1$, $|b_2|\leq q^{-2}$, then we use
\eqref{C and Cstar} to write $\Cu(q)=\Cu^*(q)+\Cu(1)$.  We have $\Cu^*(q)=0$ by \eqref{vanishing conditions for Cu star}
while $\Cu(1)=1$ if $|b_2|=q^{-2}$ and $\Cu(1)=0$ if $|b_2|<q^{-2}$ by \eqref{vanishing condition for Cu}.  If $|a_1|=q^{-1}$ and $|b_2|\in\{q^{-1},q^{-2}\}$ then
$\Cu(1)=1$.  If $\abs{a_1}=|b_2|= q^{-2}$ then since $|b|=1$, $\Cu(1)[\ell]=\tilde{\Cu}(1)$ for each $\ell$.
If $\abs{a_1}= q^{-2}$, $|b_2|\leq q^{-3}$, then write $\Cu(1)[\ell]=\Cu^*(1)[\ell]+\Cu(q^{-1})[\ell]$.  The first term is zero by \eqref{vanishing conditions for Cu star} and
the second is $q^{-1}$ when $|b_2|=q^{-3}$ and is zero otherwise by \eqref{vanishing condition for Cu}.  For the remaining cases, we are choosing the expressions shown
for later use. If $\abs{a_1}=|b_2|$ then $1-\val(a)=-\val(b)$.  Also by \eqref{C sum vs tilde C sum}, $\sum_{\ell=0}^2\Cu(q^k)[\ell]=\sum_{\ell=0}^2\tilde{\Cu}(q^k)[\ell]$
for any $k$, so we may replace one by the other. 
Finally, for the value when  $|b_2|<\abs{a_1}\leq q^{-3}$, $\lfloor\frac{1-\val(a)}{2}\rfloor=\lfloor\frac{-\val(a)}{2}\rfloor+\delta$
where $\delta=0$ if $\val(a)$ is even and $\delta=1$ if $\val(a)$ is odd.  In the latter case, $\Cu^*(q^{\lfloor\frac{1-\val(a)}2\rfloor})[\ell]=0$
by \eqref{vanishing conditions for Cu star} and we have $\Cu(q^{\lfloor\frac{1-\val(a)}2\rfloor})[\ell]=\Cu(q^{\lfloor\frac{-\val(a)}2\rfloor})[\ell]$ by \eqref{C and Cstar}.
 Substituting these values into our prior expression proves the lemma.
\end{proof}

For $n\ge 2$,
using Corollary~\ref{cor exp j-G}, we have that
$$I_{0,n}(a,b)=I(a,b;n,0)- I(a,b;n-2,1)+qI(a,b;n-3,0).$$
Once again we use
the relation between
the $I$ and $\I$ integrals given in
Lemma~\ref{lem Itxy} and
the definition of the difference $\RI$ in Proposition~\ref{prop Iij formula}. However, this case is subtly different
from the prior case, since though $|b_{n}|\leq |a|$, it is not always the case that $|b_{n-3}|\leq|a|$, so both cases
that appear in Lemma~\ref{lem Itxy} arise here.   We find that
\[
I_{0,n}(a,b)=\I(a,b;n,0)-\begin{cases}
\RI(a,b;n-2,1) & \abs{b}\le q^{n-3}\abs{a} \\
\RI(-b,-a;1,n-2) & q^{n-2}\abs{a}\le \abs{b}\le q^n\abs{a}.
\end{cases}
\]
Using Proposition~\ref{prop Iij formula} and \eqref{vanishing I(a,b;i,j)},
these terms are evaluated as follows:
\[
\abs{ab}\I(a,b;n,0)=\begin{cases}
\sum_{\ell=0}^2 \Cu(q^{\lfloor\frac{n-\val(a)}2\rfloor})[\ell]  & \abs{a}\le q^{n-2} \\
2\Cu(\abs{a}) & \abs{a}=q^{n-1} \\
 \Cu(\abs{a}) & \abs{a}=q^n;
\end{cases}
\]
for $\abs{b}\le q^{n-3}\abs{a}$ 
\[\abs{ab}\RI(a,b;n-2,1)=
\begin{cases}
\sum_{\ell=0}^2 \Cu(q^{\lfloor\frac{n-\val(a)}2 \rfloor-1})[\ell] &\abs{a}\le q^{n-2} \\
2\Cu(q^{-1}\abs{a}) &\abs{a}=q^{n-1} \\ 
\Cu(q^{-1}\abs{a}) &\abs{a}=q^n;
\end{cases}
\]
and for $q^{n-2}\abs{a}\le \abs{b}\le q^n\abs{a}$
$$
\abs{ab}\RI(-b,-a;1,n-2)=\\\begin{cases}
\sum_{\ell=0}^2 \tilde\Cu(q^{\lfloor\frac{1-\val(b)}2 \rfloor})[\ell] & \abs{b}\le q^{2n-5}\\ 
2\tilde\Cu(q^{n-2}) & \abs{b}=q^{2n-4}\\ 
\tilde\Cu(q^{n-1}) & \abs{b}=q^{2n-3}\\ 
0& \abs{b}\ge q^{2n-2}. 
\end{cases}
$$
(Recall that the sums $\tilde{\Cu}$ and $\tilde{\Cu^*}$ are defined in Section~\ref{additional-sums}.) 

Putting these together, we obtain the following evaluation.
\begin{proposition}  Suppose that $n\geq2$. Then
\begin{multline*}
\abs{ab}I_{0,n}(a,b)=\\\begin{cases}
\sum_{\ell=0}^2 \left(\Cu(q^{\lfloor\frac{n-\val(a)}2\rfloor})[\ell] -\tilde\Cu(q^{\lfloor\frac{1-\val(b)}2 \rfloor})[\ell] \right)& \abs{b}\le q^{2n-5}\text{ and }q^{-n}\abs{b}\le \abs{a}\le q^{2-n}\abs{b} \\
\sum_{\ell=0}^2 \Cu(q^{\lfloor\frac{n-\val(a)}2 \rfloor})[\ell] -2\tilde\Cu(q^{n-2}) & \abs{b}=q^{2n-4}\text{ and }\abs{a}\in \{q^{n-2},q^{n-3},q^{n-4}\}\\
\sum_{\ell=0}^2 \Cu(q^{\lfloor\frac{n-\val(a)}2 \rfloor})[\ell] -\tilde\Cu(q^{n-1}) & \abs{b}=q^{2n-3}\text{ and }\abs{a}\in \{q^{n-2},q^{n-3}\}\\
\sum_{\ell=0}^2 \Cu(q^{\lfloor\frac{n-\val(a)}2 \rfloor})[\ell] & \abs{b}= q^{2n-2}\text{ and }\abs{a}= q^{n-2} \\
\sum_{\ell=0}^2 \Cu^*(q^{\lfloor\frac{n-\val(a)}2 \rfloor})[\ell] & \abs{b}\le q^{n-3}\abs{a}\text{ and }\abs{a}\le q^{n-2} \\
2\Cu(\abs{a})& \abs{b}\in \{q^{2n-1},q^{2n-2}\}\text{ and }\abs{a}=q^{n-1}\\
2\Cu(\abs{a})-\tilde\Cu(q^{n-1}) & \abs{b}=q^{2n-3}\text{ and }\abs{a}=q^{n-1}\\
2\Cu^*(\abs{a}) & \abs{b}\le q^{2n-4}\text{ and }\abs{a}=q^{n-1}\\ 
\Cu(\abs{a}) & q^{n-2}\abs{a}\le \abs{b}\le q^n\abs{a}\text{ and }\abs{a}=q^n \\
\Cu^*(\abs{a}) & \abs{b}\le q^{n-3}\abs{a}\text{ and }\abs{a}=q^n.
\end{cases}
\end{multline*}
\end{proposition}

This evaluation may be simplified by using facts about exponential sums.   Before doing so, we separate out one simplification
that will appear repeatedly in our evaluation.
\begin{lemma}\label{cancellation lemma}
Suppose that  $m,n\geq0$, and suppose that $a,b\in F^*$ satisfy \eqref{ab-inequalities} and $\abs{b_{2n+m}}\le q^{-5}\text{ and }q^{m-n}\abs{b}\le \abs{a}\le q^{2+m-n}\abs{b}$. Then
$$\sum_{\ell=0}^2 \Cu(q^{\lfloor\frac{n-\val(a)}2\rfloor})[\ell] = \tilde\Cu(q^{\lfloor\frac{m+1-\val(b)}2 \rfloor})[\ell].$$
\end{lemma}
\begin{proof} If $\lfloor\frac{n-\val(a)}2\rfloor=\lfloor\frac{m+1-\val(b)}2 \rfloor$ then the lemma follows from \eqref{C sum vs tilde C sum}.  Since $m-\val(b)\leq n-\val(a)\leq 2+m-\val(b)$, this is the case
except when $n-\val(a)=m-\val(b)$ is odd or $n-\val(a)=2+m-\val(b)$ is even.  If $n-\val(a)=m-\val(b)$ is odd then $\lfloor\frac{m+1-\val(b)}2 \rfloor=\lfloor\frac{n-\val(a)}2\rfloor+1$.
Using \eqref{C and Cstar} and \eqref{C sum vs tilde C sum}, we have
\[
\sum_{\ell=0}^2 \left(\Cu(q^{\lfloor\frac{n-\val(a)}2\rfloor})[\ell] -\tilde\Cu(q^{\lfloor\frac{m+1-\val(b)}2 \rfloor})[\ell] \right)=-\sum_{\ell=0}^2 \tilde\Cu^*(q^{\frac{m+1-\val(b)}2})[\ell].
\]
Similarly, if $n-\val(a)=2+m-\val(b)$ is even then $\lfloor\frac{n-\val(a)}2\rfloor=\lfloor\frac{m+1-\val(b)}2 \rfloor+1$ and 
\[
\sum_{\ell=0}^2 \left(\Cu(q^{\lfloor\frac{n-\val(a)}2\rfloor})[\ell] -\tilde\Cu(q^{\lfloor\frac{m+1-\val(b)}2 \rfloor})[\ell] \right)=\sum_{\ell=0}^2 \Cu^*(q^{\frac{n-\val(a)}2})[\ell].
\]
Write $n-\val(a)=m-\val(b)+\delta$, $\delta\in\{0,2\}$, with the parities as above. To complete the proof of the lemma,
we now check the vanishing conditions \eqref{vanishing conditions for Cu star} to show that each $\tilde\Cu^*(q^{k})=0$ 
with $k=\frac{m+1-\val(b)}2+\epsilon$ with $\epsilon=0$ when $\delta=0$ (resp.\ $\epsilon=1$ when $\delta=2$).
We have 
$$|a^{-1}b^{-1}|q^{3k}= q^{\val(a)+\val(b)+3(m+1-\val(b))/2+3\epsilon}=q^{(m+2n+\val(b))/2 +3\epsilon+\delta}\geq q^{5/2}.$$
Also, to check $|b^{-1}-\rho^\ell a^{-1}|\neq |a^{-1}b^{-1}|q^{2k}$, we must show that $|a-\rho^\ell b|\neq q^{2k}$.  Since $q^{2k}=|b|q^{m+1+2\epsilon}=|a|q^{n+1+2\epsilon-\delta}$,
it follows that $\max(|a|,|b|)<q^{2k}$.   Hence the vanishing conditions hold and the lemma is proved.
\end{proof}

We now establish the following simplification.
\begin{corollary}\label{Simplification of I_{0,n}} Suppose that $n\ge2$. Then
\begin{enumerate}
\item\label{part1} $I_{0,n}(a,b)=0$ in the following cases: 
\begin{itemize} 
\item $\abs{a_n}=1$ and $\abs{b_{2n}}\le q^{-3}$.
\item $\abs{a_n}=q^{-1}$ and $\abs{b_{2n}}\le q^{-4}$.
\item $\abs{a_n}\le q^{-2}$ and $\abs{b_{2n}}\le q^{-3}\abs{a_n}$.
\item $\abs{a_n}=q^{-3}$ and $\abs{b_{2n}}\in\{q^{-3},q^{-5}\}$.
\item $\abs{a_n}\le q^{-4}$ and $\abs{b_{2n}}\in\{q^{-1}\abs{a_n}, q^{-2}\abs{a_n}\}$.
\item $\abs{a_n}\le q^{-5}$ and $\abs{b_{2n}}=\abs{a_n}$. 
\end{itemize}
\item\label{part-two}  If $|a_n|=1$ then the following evaluations hold:
\begin{itemize}
\item If $\abs{b_{2n}}=1$,  then $|ab| I_{0,n}(a,b)=q^n$.
\item If $\abs{b_{2n}}=q^{-1}$, then $|ab| I_{0,n}(a,b)=q^{n-1}[\g_{2a^{-1}b^{-1}}+\bar\g_{2a^{-1}b^{-1}}]$.
\item If $\abs{b_{2n}}=q^{-2}$, $|ab| I_{0,n}(a,b)=q^{n-1}$.
\end{itemize}
\item\label{part-three}  If $|a_n|=q^{-1}$ then the following evaluations hold:
\begin{itemize}
\item If $\abs{b_{2n}}\in \{q^{-1},q^{-2}\}$, then $|ab| I_{0,n}(a,b)=2q^{n-1}$.
\item If $\abs{b_{2n}}=q^{-3}$, then $|ab| I_{0,n}(a,b)=q^{n-2}[\g_{2a^{-1}b^{-1}}+\bar\g_{2a^{-1}b^{-1}}]$.
\end{itemize}
\item\label{part4} If $|a_n|=q^{-2}$ then the following evaluations hold:
\begin{itemize}
\item If $|b_{2n}|=q^{-2}$, then $|ab| I_{0,n}(a,b)=3\tilde\Cu(q^{n-1})$.
\item If $|b_{2n}|=q^{-3}$, then $|ab| I_{0,n}(a,b)=2q^{n-2}$.
\item If $|b_{2n}|=q^{-4}$, then $|ab| I_{0,n}(a,b)=q^{n-2}$.
\end{itemize}
\item\label{part-six} If $|a_n|\in\{q^{-3},q^{-4}\}$ and $|b_{2n}|=q^{-4}$ then $|ab| I_{0,n}(a,b)=\tilde\Cu(q^{n-2})$.
\end{enumerate}
\end{corollary}

\begin{proof}
For the first statement in~\ref{part1}, if $\abs{a_n}=1$ and $\abs{b_{2n}}\le q^{-3}$ or equivalently
$\abs{a}=q^n$, $\abs{b}\le q^{n-3}\abs{a}$, then $\abs{b^{-1}a}<\abs{ab}^{-1}\abs{a}^3$ and $\abs{ab}^{-1}\abs{a}^3=q^{2n}\abs{b}^{-1}\ge q^3$, so
the vanishing conditions \eqref{vanishing conditions for Cu star} for $\Cu^*$ are satisfied and hence $\Cu^*(\abs{a})=0$.
Similarly, if $\abs{a_n}=q^{-1}$ and $\abs{b_{2n}}\le q^{-4}$ then $q^2\le \abs{ab}^{-1}\abs{a}^3\ne \abs{b^{-1}a}$ and therefore $\Cu^*(\abs{a})=0$.
If $\abs{a_n}\le q^{-2}$ and $\abs{b_{2n}}\le q^{-3}\abs{a_n}$ then each $\Cu^*(q^{\lfloor\frac{n-\val(a)}2 \rfloor})[\ell]$ that appears in the evaluation is zero.
Indeed, the vanishing criterion \eqref{vanishing conditions for Cu star} is satisfied since $\abs{ab}^{-1}q^{3\lfloor\frac{n-\val(a)}2 \rfloor}\ge q^2$ and $\abs{ab}^{-1}q^{2\lfloor\frac{n-\val(a)}2 \rfloor}\ge 
\max(\abs{a}^{-1},\abs{b}^{-1})$. 

When $\abs{a_n}=q^{-3}$,  we have $\lfloor\frac{n-\val(a)}2\rfloor=n-2$.  If $\abs{b_{2n}}=q^{-3}$ then $\abs{b}^{-1}q^{n-2}=q^{1-n}<1$.
 It follows that 
$\Cu(q^{\lfloor\frac{n-\val(a)}2 \rfloor})[\ell] $ is independent of $\ell$ and equals
$\tilde\Cu(q^{n-2})$. Also, $\abs{a}^{-1}q^{n-2}=q$ and 
$\abs{ab}^{-1}q^{3(n-2)}=1$ and these imply that $\tilde\Cu(q^{n-2})=0$.  We must also show the vanishing of $\tilde\Cu(q^{n-1})$.
Using \eqref{C and Cstar}, we have $\tilde\Cu(q^{n-1})=\tilde\Cu^*(q^{n-1})+\tilde\Cu(q^{n-2})$.  By the vanishing
criterion \eqref{vanishing conditions for Cu star}, the first summand $\tilde\Cu^*(q^{n-1})$ is zero, and we have just seen
that the second is zero as well.  So the vanishing for $\abs{a_n}=\abs{b_{2n}}=q^{-3}$ is established.
 When $\abs{b_{2n}}=q^{-5}$, the vanishing follows from Lemma~\ref{cancellation lemma} with $m=0$.
Similarly, if $\abs{a_n}\le q^{-4}$ and $|b_{2n}|\in \{q^{-1}|a_n|, q^{-2}\abs{a_n}\}$ or if $\abs{a_n}\le q^{-5}$ and $\abs{b_{2n}}=\abs{a_n}$ then the vanishing follows from Lemma~\ref{cancellation lemma} with $m=0$.
This concludes the proof of part~\ref{part1}.

Part~\ref{part-two} follows directly from Lemma~\ref{special values of Cu for later}, as does the first bullet in part~\ref{part-three}. If $\abs{a_n}=q^{-1}$ and $\abs{b_{2n}}=q^{-3}$,
the value $\Cu(\abs{a})=q^{n-2}[\g_{2a^{-1}b^{-1}}+\bar\g_{2a^{-1}b^{-1}}]$ is established there as well.
We similarly have $\tilde\Cu(q^{n-1})=q^{n-2}[\g_{2a^{-1}b^{-1}}+\bar\g_{2a^{-1}b^{-1}}]$.  Part~\ref{part-three} follows.

Turning to part~\ref{part4}, if $\abs{a_n}= q^{-2}$ then $\lfloor\frac{n-\val(a)}2\rfloor=n-1$. 
If $\abs{b_{2n}}\in\{q^{-2},q^{-3}\}$ then $q^{n-1}\abs{b}^{-1}\le q^{n-2}\le 1$ and it follows that $ \Cu(q^{\lfloor\frac{n-\val(a)}2 \rfloor})[\ell] =\tilde\Cu(q^{n-1})$  is independent of $\ell$.
The simplified formula follows in these cases. If $\abs{b_{2n}}\le q^{-4}$ then we see that $q^{3(n-1)}\abs{ab}^{-1}=q^{2n-1}\abs{b}^{-1}\ge q^3$ and furthermore that
$q^{3(n-1)}\abs{ab}^{-1}>q^{n-1}\max(\abs{a}^{-1},\abs{b}^{-1})$. It therefore follows from \eqref{vanishing conditions for Cu star} that $\Cu^*(q^{\lfloor\frac{n-\val(a)}2 \rfloor})[\ell] =0$. 
This takes care of the vanishing for the case $\abs{b_{2n}}\le q^{-5}$. It remains to treat the case $\abs{b_{2n}}=q^{-4}$. 
In this case we further deduce from \eqref{C and Cstar} that $\Cu(q^{\lfloor\frac{n-\val(a)}2 \rfloor})[\ell] =\Cu(q^{n-2})[\ell]$ and 
then easily observe that $\Cu(q^{n-2})[\ell]=\tilde\Cu(q^{n-2})=q^{n-2}$. The simplified evaluations of part~\ref{part4} follow. 

For part~\ref{part-six}, 
if $\abs{a_n}\in\{q^{-3},q^{-4}\}$ and $\abs{b_{2n}}=q^{-4}$, then $\lfloor\frac{n-\val(a)}2\rfloor=n-2$ and $\abs{b}^{-1}q^{n-2}\le 1$. It follows that 
$\Cu(q^{\lfloor\frac{n-\val(a)}2 \rfloor})[\ell] $ is independent of $\ell$ and equals
$\tilde\Cu(q^{n-2})$. Thus here
$|ab| I_{0,n}(a,b)=\sum_{\ell=0}^2 \Cu(q^{\lfloor\frac{n-\val(a)}2 \rfloor})[\ell] -2\tilde\Cu(q^{n-2}) =\tilde\Cu(q^{n-2}).$
This concludes the proof of the corollary.
\end{proof}

\subsubsection{Evaluation of $I_{m,n}(a,b)$ when $\min(m,n)=1$}
We continue to suppose that \eqref{ab-inequalities} holds.  Thus we are in the case $|a|\leq q^{n+2m}$ and $|b|\leq q^{n-m}|a|$. Since $\min(m,n)=1$, 
applying Corollary~\ref{cor exp j-G} and arguing as before, we find that
$$I_{1,1}(a,b)=\RI(a,b;1,1)-\RI(a,b;0,0),$$
for $m\geq2$ we have
\begin{equation}\label{recursion for m,1}
I_{m,1}=\RI(a,b;1,m)-\RI(a,b;2,m-2),
\end{equation}
and for $n\ge2$ we have
\begin{equation}\label{recursion for 1,n}
I_{1,n}(a,b)=\RI(a,b;n,1)-\begin{cases} \RI(a,b;n-2,2) & \abs{b}\le q^{n-4}\abs{a}\\ \RI(-b,-a;2,n-2) & q^{n-3}\abs{a}\le \abs{b}\le q^{n-1}\abs{a}.\end{cases}
\end{equation}

From Proposition~\ref{prop Iij formula} and \eqref{vanishing I(a,b;i,j)} we obtain the evaluations
\[
\abs{ab}\RI(a,b;1,1)=\begin{cases} \sum_{\ell=0}^2\Cu(q^{\lfloor \frac{1-\val(a)}2\rfloor})[\ell] & \abs{a}\le q \\
2\Cu(q^{-1}\abs{a}) & \abs{a}=q^2 \\
\Cu(q^{-1}\abs{a}) & \abs{a}=q^3
\end{cases}
\]
and
\[
\abs{ab}\RI(a,b;0,0)=\begin{cases} \sum_{\ell=0}^2\Cu(q^{\lfloor \frac{-\val(a)}2\rfloor})[\ell] & \abs{a}\le q^{-2} \\
2\Cu(\abs{a}) & \abs{a}=q^{-1} \\
\Cu(\abs{a}) & \abs{a}=1 \\
0 & q\le \abs{a}\le q^3.
\end{cases}
\]
We conclude that
\begin{equation}\label{I integral with m=n=1}
\abs{ab}I_{1,1}(a,b)=
\begin{cases} \sum_{\ell=0}^2\left(\Cu(q^{\lfloor \frac{1-\val(a)}2\rfloor})[\ell]-\Cu(q^{\lfloor \frac{-\val(a)}2\rfloor})[\ell]\right)& \abs{a}\le q^{-2} \\
\sum_{\ell=0}^2\Cu(q^{\lfloor \frac{1-\val(a)}2\rfloor})[\ell]-2\Cu(\abs{a}) & \abs{a}=q^{-1}\\
\sum_{\ell=0}^2\Cu(q^{\lfloor \frac{1-\val(a)}2\rfloor})[\ell]-\Cu(\abs{a}) & \abs{a}=1\\
\sum_{\ell=0}^2\Cu(q^{\lfloor \frac{1-\val(a)}2\rfloor})[\ell] & \abs{a}=q\\
2\Cu(q^{-1}\abs{a}) & \abs{a}=q^2 \\
\Cu(q^{-1}\abs{a}) & \abs{a}=q^3.
\end{cases}
\end{equation}

For $m\ge 2$ we have
\[
\abs{ab}\RI(a,b;1,m)=\begin{cases} \sum_{\ell=0}^2\Cu(q^{\lfloor \frac{1-\val(a)}2\rfloor})[\ell] & \abs{a}\le q^{2m-1} \\
2\Cu(q^{-m}\abs{a}) & \abs{a}=q^{2m} \\
\Cu(q^{-m}\abs{a}) & \abs{a}=q^{2m+1}
\end{cases}
\]
and
\[
\abs{ab}\RI(a,b;2,m-2)=\begin{cases} 
\sum_{\ell=0}^2\Cu(q^{\lfloor \frac{2-\val(a)}2\rfloor})[\ell] & \abs{a}\le q^{2m-4} \\
2\Cu(q^{2-m}\abs{a}) & \abs{a}=q^{2m-3} \\
\Cu(q^{2-m}\abs{a}) & \abs{a}=q^{2m-2} \\
0 & q^{2m-1}\le \abs{a}\le q^{2m+1}.
\end{cases}
\]
Substituting into \eqref{recursion for m,1}, this gives the following evaluation.
\begin{proposition}\label{case I em, one} Suppose that $m\ge2$. Then
\[
\abs{ab}I_{m,1}(a,b)=
\begin{cases} 
\sum_{\ell=0}^2\left(\Cu(q^{\lfloor \frac{1-\val(a)}2\rfloor})[\ell]-\Cu(q^{\lfloor \frac{2-\val(a)}2\rfloor})[\ell]\right)& \abs{a}\le q^{2m-4} \\
\sum_{\ell=0}^2\Cu(q^{\lfloor \frac{1-\val(a)}2\rfloor})[\ell]-2\Cu(q^{2-m}\abs{a}) & \abs{a}=q^{2m-3}\\
\sum_{\ell=0}^2\Cu(q^{\lfloor \frac{1-\val(a)}2\rfloor})[\ell]-\Cu(q^{2-m}\abs{a}) & \abs{a}=q^{2m-2}\\
\sum_{\ell=0}^2\Cu(q^{\lfloor \frac{1-\val(a)}2\rfloor})[\ell] & \abs{a}=q^{2m-1}\\
2\Cu(q^{-m}\abs{a}) & \abs{a}=q^{2m} \\
\Cu(q^{-m}\abs{a}) & \abs{a}=q^{2m+1}.
\end{cases}
\]
\end{proposition}

Note that in some cases the formulas for $m=1$ and $m\ge2$ are different. 
As in the prior cases, we may rewrite this expression and that of \eqref{I integral with m=n=1} by using Lemma~\ref{special values of Cu for later} and the facts about 
cubic exponential integrals in Section~\ref{Cubic exponential notation}. We obtain the following.
\begin{corollary}\label{I integral for any m and n=1}
Suppose that $m\geq1$.  Then 
$$\abs{ab}I_{m,1}(a,b)=
\begin{cases} 
q^{m+1}&|a_{2m+1}|=|b_{m+2}|=1\\
q^m(\g_{2a^{-1}b^{-1}}+\overline{\g}_{2a^{-1}b^{-1}})&|a_{2m+1}|=1, |b_{m+2}|=q^{-1}\\
q^{m}&|a_{2m+1}|=1, |b_{m+2}|=q^{-2}\\
2q^{m}&|a_{2m+1}|=q^{-1}, |b_{m+2}|\in\{q^{-1},q^{-2}\}\\
2\Cu(q^m)&|a_{2m+1}|=q^{-1}, |b_{m+2}|\leq q^{-3} \\
3q^{m-1}(\g_{2a^{-1}b^{-1}}+\overline{\g}_{2a^{-1}b^{-1}})&|a_{2m+1}|=|b_{m+2}|=q^{-2}\\
3q^{m-1}&|a_{2m+1}|=q^{-2}, |b_{m+2}|=q^{-3}\\
2q^{m-1}&|a_{2m+1}|=|b_{m+2}|=q^{-3}\\
2\Cu(q^{m-1})&|a_{2m+1}|=q^{-3}, |b_{m+2}|\leq q^{-4}\\
q^{m-2}&|a_{2m+1}|=|b_{m+2}|=q^{-4}\\
0&otherwise.
\end{cases}
$$
\end{corollary}
\begin{proof}
This follows by a case-by-case analysis similar to the simplifications in prior cases.
For example,  to show 
the vanishing if $|b_{m+2}|\leq |a_{2m+1}|\leq q^{-5}$, when $m\ge2$  by Proposition~\ref{case I em, one} we must show
that the quantities 
$$\Cu(q^{\lfloor \frac{1-\val(a)}2\rfloor})[\ell]-\Cu(q^{\lfloor \frac{2-\val(a)}2\rfloor})[\ell]$$ are zero.  If $\val(a)$ is odd this is apparent,
while by \eqref{C and Cstar}, in the remaining case we must show that $\Cu^*(q^{1-\frac{\val(a)}2})[\ell]$ is zero.  Note that since $|b_{m+2}|\leq |a_{2m+1}|$ and $m\geq2$, we have $|b|<|a|$.
To apply \eqref{vanishing conditions for Cu star}, we note that  $q^{1-\frac{\val(a)}2}\abs{b}^{-1}=q\abs{a}^{\frac12}\abs{b}^{-1}$ and $q^{3(1-\frac{\val(a)}2)}\abs{ab}^{-1}=q^3\abs{a}^{\frac12}\abs{b}^{-1}$. 
It follows from \eqref{vanishing conditions for Cu star} that $\Cu^*(q^{1-\frac{\val(a)}2})[\ell]=0$ unless
 $q^2\abs{a}^\frac12\le \abs{b}$, but this contradicts the assumption $|b_{m+2}|\leq |a_{2m+1}|\leq q^{-5}$. 
When $m=1$ (so $\abs{b}\le \abs{a}\le q^{-2}$) the argument is similar, but using \eqref{I integral with m=n=1}. The vanishing when $\val(a)$ is even is immediate,
and when $\val(a)$ is odd 
follows from \eqref{C and Cstar} and the vanishing of $\Cu^*(q^{\frac{1-\val(a)}2})[\ell]$, which 
is proved by checking the criterion \eqref{vanishing conditions for Cu star} exactly as when $m\ge2$. We omit the remaining details as they are
similar to the prior cases.
 \end{proof}

Finally, for $n\ge 2$ the terms in \eqref{recursion for 1,n} are given by
\[
\abs{ab}\RI(a,b;n,1)=\begin{cases} \sum_{\ell=0}^2\Cu(q^{\lfloor \frac{n-\val(a)}2\rfloor})[\ell] & \abs{a}\le q^n \\
2\Cu(q^{-1}\abs{a}) & \abs{a}=q^{n+1} \\
\Cu(q^{-1}\abs{a}) & \abs{a}=q^{n+2};
\end{cases}
\]
for $|b|\leq q^{n-4}|a|$
\[
\abs{ab}\RI(a,b;n-2,2)=\begin{cases} \sum_{\ell=0}^2\Cu(q^{\lfloor \frac{n-2-\val(a)}2\rfloor})[\ell] & \abs{a}\le q^{n} \\
2\Cu(q^{-2}\abs{a}) & \abs{a}=q^{n+1} \\
\Cu(q^{-2}\abs{a}) & \abs{a}=q^{n+2};
\end{cases}
\]
and for $q^{n-3}|a|\leq |b| \leq q^{n-1}|a|$
\[
\abs{ab} \RI(-b,-a;2,n-2) =\begin{cases}
\sum_{\ell=0}^2 \tilde\Cu(q^{\lfloor\frac{2-\val(b)}2 \rfloor})[\ell] & \abs{b}\le q^{2n-4} \\ 
2\tilde\Cu(q^{2-n}\abs{b}) & \abs{b}=q^{2n-3} \\ 
\tilde\Cu(q^{2-n}\abs{b}) & \abs{b}=q^{2n-2} \\
0& \abs{b}\ge q^{2n-1}. 
\end{cases}
\]
We obtain the following evaluation.
\begin{proposition}\label{prop 1,n with n ge 2} For $n\geq2$, 
\begin{multline*}
\abs{ab}I_{1,n}(a,b)=\\\begin{cases}
\sum_{\ell=0}^2 \left(\Cu(q^{\lfloor\frac{n-\val(a)}2\rfloor})[\ell] - \tilde\Cu(q^{\lfloor\frac{2-\val(b)}2 \rfloor})[\ell] \right)& \abs{b}\le q^{2n-4}\text{ and }q^{1-n}\abs{b}\le \abs{a}\le q^{3-n}\abs{b} \\
\sum_{\ell=0}^2 \Cu(q^{\lfloor\frac{n-\val(a)}2 \rfloor})[\ell] -2\tilde\Cu(q^{2-n}\abs{b}) & \abs{b}=q^{2n-3}\text{ and }\abs{a}\in\{q^{n-2},q^{n-1},q^{n}\} \\
\sum_{\ell=0}^2 \Cu(q^{\lfloor\frac{n-\val(a)}2 \rfloor})[\ell] -\tilde\Cu(q^{2-n}\abs{b}) & \abs{b}=q^{2n-2}\text{ and }\abs{a}\in\{q^{n-1},q^{n}\} \\
\sum_{\ell=0}^2 \Cu(q^{\lfloor\frac{n-\val(a)}2 \rfloor})[\ell] & \abs{b}= q^{2n-1}\text{ and }\abs{a}= q^{n} \\
\sum_{\ell=0}^2 \Cu^*(q^{\lfloor\frac{n-\val(a)}2 \rfloor})[\ell] & \abs{b}\le q^{n-4}\abs{a}\text{ and }\abs{a}\le q^{n} \\
2\Cu(q^{-1}\abs{a})& \abs{b}\in \{q^{2n-1},q^{2n}\}\text{ and }\abs{a}=q^{n+1}\\
2\Cu(q^{-1}\abs{a})-\tilde\Cu(q^{2-n}\abs{b}) & \abs{b}=q^{2n-2}\text{ and }\abs{a}=q^{n+1}\\
2\Cu^*(q^{-1}\abs{a}) & \abs{b}\le q^{2n-3}\text{ and }\abs{a}=q^{n+1}\\ 
\Cu(q^{-1}\abs{a}) & q^{n-3}\abs{a}\le \abs{b}\le q^{n-1}\abs{a}\text{ and }\abs{a}=q^{n+2} \\
\Cu^*(q^{-1}\abs{a}) & \abs{b}\le q^{n-4}\abs{a}\text{ and }\abs{a}=q^{n+2}.
\end{cases}
\end{multline*}
\end{proposition}

We once again rewrite this expression by using Section~\ref{Cubic exponential notation}. We obtain the following.
\begin{corollary}\label{I integral for n ge 2 and m=1}
Suppose that $n\geq 2$.  Then 
$$
\abs{ab}I_{1,n}(a,b)=
\begin{cases} 
q^{n+1}& |a_{2+n}|=|b_{1+2n}|=1 \\
q^n(\g_{2a^{-1}b^{-1}}+\overline{\g}_{2a^{-1}b^{-1}})& |a_{2+n}|=1, |b_{1+2n}|=q^{-1}\\
q^n&|a_{2+n}|=1, |b_{1+2n}|=q^{-2}\\
2q^n&|a_{2+n}|=q^{-1}, |b_{1+2n}|\in\{q^{-1},q^{-2}\}\\
q^{n-1}(\g_{2a^{-1}b^{-1}}+\overline{\g}_{2a^{-1}b^{-1}})&|a_{2+n}|=q^{-1}, |b_{1+2n}|=q^{-3}\\
3q^{n-1}(\g_{2a^{-1}b^{-1}}+\overline{\g}_{2a^{-1}b^{-1}})&|a_{2+n}|=|b_{1+2n}|=q^{-2}\\
2q^{n-1}&|a_{2+n}|=q^{-2}, |b_{1+2n}|=q^{-3}\\
q^{n-1}&|a_{2+n}|=q^{-2}, |b_{1+2n}|=q^{-4}\\
2q^{n-1}&|a_{2+n}|=|b_{1+2n}|=q^{-3}\\
q^{n-2}(\g_{2a^{-1}b^{-1}}+\overline{\g}_{2a^{-1}b^{-1}})&|a_{2+n}|=q^{-3}, |b_{1+2n}|=q^{-4}\\
q^{n-2}&|a_{2+n}|=|b_{1+2n}|=q^{-4}\\
0&otherwise.
\end{cases}
$$
\end{corollary}

\begin{proof} We first note that the following terms in Proposition~\ref{prop 1,n with n ge 2} are zero by the properties of cubic integrals:
\begin{itemize}
\item
$\sum_{\ell=0}^2 \left(\Cu(q^{\lfloor\frac{n-\val(a)}2\rfloor})[\ell] - \tilde\Cu(q^{\lfloor\frac{2-\val(b)}2 \rfloor})[\ell] \right)\quad \text{for~}
\abs{b}\le q^{2n-4}\text{ and }q^{1-n}\abs{b}\le \abs{a}\le q^{3-n}$
\item $\sum_{\ell=0}^2 \Cu^*(q^{\lfloor\frac{n-\val(a)}2 \rfloor})[\ell] \quad \text{for~} \abs{b}\le q^{n-4}\abs{a}\text{ and }\abs{a}\le q^{n}$
\item $2\Cu^*(q^{-1}\abs{a}) \quad\text{for~} \abs{b}\le q^{2n-3}\text{ and }\abs{a}=q^{n+1}$
\item $\Cu^*(q^{-1}\abs{a}) \quad\text{for} \abs{b}\le q^{n-4}\abs{a}\text{ and }\abs{a}=q^{n+2}.$
\end{itemize}
The first vanishes by Lemma~\ref{cancellation lemma} with $m=1$. 
For the other three terms, one sees that each $\Cu^*$ is zero by criterion \eqref{vanishing conditions for Cu star}.

The rest of the proof is similar to the prior cases. For example, if $|a_{2+n}|=1$, $|b_{1+2n}|\in\{1,q^{-1},q^{-2}\}$, the value is $\Cu(q^{n+1})$
and this is given in Lemma~\ref{special values of Cu for later}.  
If $|a_{2+n}|=q^{-1}$, $|b_{1+2n}|\in\{q^{-1},q^{-2}\}$ we obtain $2\Cu(q^n)=2q^n$ by the same lemma, 
while when $|a_{2+n}|=q^{-1}$, $|b_{1+2n}|=q^{-3}$ we obtain $2\Cu(q^n)-\tilde{\Cu}(q^n)$,
each evaluated by \eqref{special Cu sum}.
The other cases are similar and we omit the details.\end{proof}

\subsubsection{The evaluation of $I_{m,n}(a,b)$ when $\min(m,n)\ge 2$}
Again, by \eqref{ab-inequalities}, $|a|\leq q^{n+2m}$ and $|b|\leq q^{n-m}|a|$. 
In this case, using Corollary~\ref{cor exp j-G} as above, we may write the orbital integral in terms of four differences:
\begin{multline*}
I_{m,n}(a,b)=\RI(a,b;n,m)-\RI(a,b;n+1,m-2)+\RI(a,b;n-1,m-1)-\\
\begin{cases}
\RI(-b,-a;m+1,n-2) & q^{-2}\abs{a_m}\le \abs{b_n}\le \abs{a_m}\\ 
\RI(a,b;n-2,m+1) & \abs{b_n}\le q^{-3}\abs{a_m}. \end{cases}
\end{multline*}
We have the following evaluations for these contributions:
\[
\abs{ab}\RI(a,b;n,m)=\begin{cases} \sum_{\ell=0}^2\Cu(q^{\lfloor \frac{n-\val(a)}2\rfloor})[\ell] & \abs{a}\le q^{n+2m-2} \\
2\Cu(q^{-m}\abs{a}) & \abs{a}=q^{n+2m-1} \\
\Cu(q^{-m}\abs{a}) & \abs{a}=q^{n+2m};
\end{cases}
\]
\[
\abs{ab}\RI(a,b;n+1,m-2)=\begin{cases} \sum_{\ell=0}^2\Cu(q^{\lfloor \frac{n+1-\val(a)}2\rfloor})[\ell] & \abs{a}\le q^{n+2m-5} \\
2\Cu(q^{2-m}\abs{a}) & \abs{a}=q^{n+2m-4} \\
\Cu(q^{2-m}\abs{a}) & \abs{a}=q^{n+2m-3}\\
0& \abs{a}\ge q^{m+2n-2};
\end{cases}
\]
\[
\abs{ab}\RI(a,b;n-1,m-1)=\begin{cases} \sum_{\ell=0}^2\Cu(q^{\lfloor \frac{n-1-\val(a)}2\rfloor})[\ell] & \abs{a}\le q^{n+2m-5} \\
2\Cu(q^{1-m}\abs{a}) & \abs{a}=q^{n+2m-4} \\
\Cu(q^{1-m}\abs{a}) & \abs{a}=q^{n+2m-3}\\
0& \abs{a}\ge q^{m+2n-2};
\end{cases}
\]
for $q^{-2}\abs{a_m}\le \abs{b_n}\le \abs{a_m}$
\[
\abs{ab}\RI(-b,-a;m+1,n-2) 
=\begin{cases}
\sum_{\ell=0}^2 \tilde\Cu(q^{\lfloor\frac{m+1-\val(b)}2 \rfloor})[\ell] & \abs{b}\le q^{m+2n-5}
\\
2\tilde\Cu(q^{2-n}\abs{b}) & \abs{b}=q^{m+2n-4}
\\
\tilde\Cu(q^{2-n}\abs{b}) & \abs{b}=q^{m+2n-3}
\\
0& \abs{b}\ge q^{m+2n-2};
\end{cases}
\]
and for $ \abs{b_n}\le q^{-3}\abs{a_m}$
\[
\abs{ab}\RI(a,b;n-2,m+1)=\begin{cases} \sum_{\ell=0}^2\Cu(q^{\lfloor \frac{n-2-\val(a)}2\rfloor})[\ell] & \abs{a}\le q^{n+2m-2} \\
2\Cu(q^{-m-1}\abs{a}) & \abs{a}=q^{n+2m-1} \\
\Cu(q^{-m-1}\abs{a}) & \abs{a}=q^{n+2m}.
\end{cases}
\]

Combining these expressions,  we conclude that the following evaluation holds.
\begin{proposition}\label{prop the awful case} Suppose $m,n\geq2$, $a,b\in F^*$, and the inequalities \eqref{ab-inequalities} hold.  Then $I_{m,n}(a,b)$ is evaluated as follows, and is zero in all other cases.
\begin{enumerate}
\item\label{this is one} 
$\Cu(q^{-m}\abs{a})$ when $\abs{a}=q^{n+2m}, \ \abs{b}\in \{q^{m+2n-2},q^{m+2n-1},q^{m+2n}\}$.
\item\label{this is two}
$\Cu^*(q^{-m}\abs{a})$ when $\abs{a}=q^{n+2m}, \ \abs{b}\le q^{m+2n-3}$.
\item\label{this is three} $2\Cu(q^{-m}\abs{a})$ when $\abs{a}=q^{n+2m-1},\ \abs{b}\in \{q^{m+2n-2},q^{m+2n-1}\}$.
\item\label{this is four} $2\Cu(q^{-m}\abs{a})-\tilde\Cu(q^{2-n}\abs{b})$ when $\abs{a}=q^{n+2m-1},\ \abs{b}= q^{m+2n-3}$.
\item\label{this is five} $2\Cu^*(q^{-m}\abs{a})$ when $\abs{a}=q^{n+2m-1}, \ \abs{b}\le q^{m+2n-4}$.
\item\label{this is six} $\sum_{\ell=0}^2\Cu(q^{\lfloor \frac{n-\val(a)}2\rfloor})[\ell]$ when $\abs{a}=q^{n+2m-2},\ \abs{b}=q^{m+2n-2}$.
\item\label{this is seven} $\sum_{\ell=0}^2\Cu(q^{\lfloor \frac{n-\val(a)}2\rfloor})[\ell] -\tilde\Cu(q^{2-n}\abs{b})$ when $\abs{a}=q^{n+2m-2},\ \abs{b}=q^{m+2n-3}$.
\item\label{this is eight} $\sum_{\ell=0}^2\Cu(q^{\lfloor \frac{n-\val(a)}2\rfloor})[\ell] -2\tilde\Cu(q^{2-n}\abs{b})$ when $\abs{a}=q^{n+2m-2},\ \abs{b}=q^{m+2n-4}$.
\item\label{this is nine} $\sum_{\ell=0}^2\Cu^*(q^{\lfloor \frac{n-\val(a)}2\rfloor})[\ell]$ when $\abs{a}=q^{n+2m-2},\ \abs{b}\le q^{m+2n-5}$.
\item\label{this is ten} $\sum_{\ell=0}^2\Cu(q^{\lfloor \frac{n-\val(a)}2\rfloor})[\ell] -\Cu^*(q^{2-m}\abs{a})-\tilde\Cu(q^{2-n}\abs{b})$ when $\abs{a}=q^{n+2m-3},\ \abs{b}=q^{m+2n-3}$.
\item\label{this is eleven} $\sum_{\ell=0}^2\Cu(q^{\lfloor \frac{n-\val(a)}2\rfloor})[\ell] -\Cu^*(q^{2-m}\abs{a})-2\tilde\Cu(q^{2-n}\abs{b})$ when $\abs{a}=q^{n+2m-3},\ \abs{b}=q^{m+2n-4}$.
\item\label{this is twelve} $\sum_{\ell=0}^2\left(\Cu(q^{\lfloor \frac{n-\val(a)}2\rfloor})[\ell]-\tilde\Cu(q^{\lfloor\frac{m+1-\val(b)}2 \rfloor})[\ell] \right)-\Cu^*(q^{2-m}\abs{a})$ when $\abs{a}=q^{n+2m-3}$, $\abs{b}=q^{m+2n-5}$.
\item\label{this is thirteen} $\sum_{\ell=0}^2\Cu^*(q^{\lfloor \frac{n-\val(a)}2\rfloor})[\ell]-\Cu^*(q^{2-m}\abs{a})$ when $\abs{a}=q^{n+2m-3},\ \abs{b}\le q^{m+2n-6}$.
\item\label{this is fourteen} $\sum_{\ell=0}^2\Cu(q^{\lfloor \frac{n-\val(a)}2\rfloor})[\ell] -2\Cu^*(q^{2-m}\abs{a})-2\tilde\Cu(q^{2-n}\abs{b})$ when $\abs{a}=q^{n+2m-4},\ \abs{b}=q^{m+2n-4}$.
\item\label{this is fifteen} $\sum_{\ell=0}^2\left(\Cu(q^{\lfloor \frac{n-\val(a)}2\rfloor})[\ell]-\tilde\Cu(q^{\lfloor\frac{m+1-\val(b)}2 \rfloor})[\ell] \right) -2\Cu^*(q^{2-m}\abs{a})$ when $\abs{a}=q^{n+2m-4}$, $\abs{b}\in \{q^{m+2n-5},q^{m+2n-6}\}.$
\item\label{this is sixteen} $\sum_{\ell=0}^2\Cu^*(q^{\lfloor \frac{n-\val(a)}2\rfloor})[\ell]-2\Cu^*(q^{2-m}\abs{a})$ when $\abs{a}=q^{n+2m-4},\ \abs{b}\le q^{m+2n-7}$.
\item\label{this is seventeen} $\sum_{\ell=0}^2\left(\Cu(q^{\lfloor \frac{n-\val(a)}2\rfloor})[\ell]-\Cu^*(q^{\lfloor \frac{n+1-\val(a)}2\rfloor})[\ell] -\tilde\Cu(q^{\lfloor\frac{m+1-\val(b)}2 \rfloor})[\ell] \right)$ when $\abs{a}\le q^{n+2m-5}$, $q^{-2}\abs{a_m}\le \abs{b_n}\le \abs{a_m}$.
\item\label{this is eighteen}  $\sum_{\ell=0}^2\left(\Cu^*(q^{\lfloor \frac{n-\val(a)}2\rfloor})[\ell]-\Cu^*(q^{\lfloor \frac{n+1-\val(a)}2\rfloor})[\ell] \right)$ when $\abs{a}\le q^{n+2m-5},\ \abs{b_n}\le q^{-3}\abs{a_m}$.
\end{enumerate}
\end{proposition}

Though this evaluation is complicated, we are able to simplify it substantially due to the
identities in Section~\ref{Cubic exponential notation}.  We obtain the following.
\begin{corollary}\label{I integral for n ge 2 and m ge 2}
Suppose that $m,n\geq 2$.  Then 
$$
\abs{ab}I_{m,n}(a,b)=
\begin{cases} 
q^{m+n}& |a_{2m+n}|=|b_{m+2n}|=1 \\
q^{m+n-1}(\g_{2a^{-1}b^{-1}}+\overline{\g}_{2a^{-1}b^{-1}})& |a_{2m+n}|=1, |b_{m+2n}|=q^{-1}\\
q^{m+n-1}&|a_{2m+n}|=1, |b_{m+2n}|=q^{-2}\\
2q^{m+n-1}&|a_{2m+n}|=q^{-1}, |b_{m+2n}|\in\{q^{-1},q^{-2}\}\\
q^{m+n-2}(\g_{2a^{-1}b^{-1}}+\overline{\g}_{2a^{-1}b^{-1}})&|a_{2m+n}|=q^{-1}, |b_{m+2n}|=q^{-3}\\
3q^{m+n-2}(\g_{2a^{-1}b^{-1}}+\overline{\g}_{2a^{-1}b^{-1}})&|a_{2m+n}|=|b_{m+2n}|=q^{-2}\\
2q^{m+n-2}&|a_{2m+n}|=q^{-2}, |b_{m+2n}|=q^{-3}\\
q^{m+n-2}&|a_{2m+n}|=q^{-2}, |b_{m+2n}|=q^{-4}\\
2q^{m+n-2}&|a_{2m+n}|=|b_{n+2n}|=q^{-3}\\
q^{m+n-3}(\g_{2a^{-1}b^{-1}}+\overline{\g}_{2a^{-1}b^{-1}})&|a_{2m+n}|=q^{-3}, |b_{m+2n}|=q^{-4}\\
q^{m+n-3}&|a_{2m+n}|=|b_{n+2n}|=q^{-4}\\
0&otherwise.
\end{cases}
$$
\end{corollary}

\begin{proof} The values for $|a_{2m+n}|=1$, $|b_{m+2n}|\in\{1,q^{-1},q^{-2}\}$ in Proposition~\ref{prop the awful case}, contribution \eqref{this is one} are given by Lemma~\ref{special values of Cu for later},
while if $|b_{m+2n}|<q^{-2}$ it is immediate from \eqref{vanishing conditions for Cu star} that $\Cu^*(q^{m+n})=0$, so the contribution \eqref{this is two} is zero.  The values for 
$|a_{2m+n}|=q^{-1}$ in \eqref{this is three} and \eqref{this is four} are given in Lemma~\ref{special values of Cu for later} or (for the differences) as in prior cases; note that contribution \eqref{this is five} is identically
zero.  The value when $|a_{2m+n}|=|b_{m+2n}|=q^{-2}$ is given by  \eqref{this is six}, and each $\Cu(q^{m+n-1})[\ell]=q^{m+n-2}(\g_{2a^{-1}b^{-1}}+\overline{\g}_{2a^{-1}b^{-1}})$
 by equation \eqref{special Cu sum}.  The value when $|a_{2m+n}|=q^{-2}$, $|b_{m+2n}|=q^{-3}$ (resp.\ $|b_{m+2n}|=q^{-4}$) is given by contribution \eqref{this is seven} (resp.\ \eqref{this is eight}),
 and one sees as above that the terms $\Cu$, $\tilde{\Cu}$ there are each $q^{m+n-2}$.  When $|a_{2m+n}|=q^{-2}$, $|b_{m+2n}|<q^{-4}$, we get zero using \eqref{this is nine}
 since $\Cu^*(q^{m+n-1})[\ell]=0$ for each $\ell$ (the vanishing criterion holds since $|a^{-1}b^{-1}|q^{3(m+n-1)}\ge q^2$ and $|a|=q^{2m+n-2}\neq q^{2(m+n-1)}$).
 
 When $|a_{2m+n}|=|b_{n+2n}|=q^{-3}$, the value is expressed in Proposition~\ref{prop the awful case}, contribution \eqref{this is ten}: $\sum_\ell \Cu(q^{m+n-1})[\ell]-\Cu^*(q^{m+n-1})-\tilde{\Cu}(q^{m+n-1})$.  By the
 vanishing criteria, $\Cu^*(q^{m+n-1})[\ell]=\Cu^*(q^{m+n-1})=\tilde{\Cu}^*(q^{m+n-1})=0$.  Then using equation \eqref{C and Cstar}, the contribution becomes $\sum_\ell \Cu(q^{m+n-2})[\ell]-\tilde{\Cu}(q^{m+n-2})$.
 It is direct to check that each term here is $q^{m+n-2}$ and the evaluation $2q^{m+n-2}$ follows.  The contribution when $|a_{2m+n}|=q^{-3}$, $|b_{m+2n}|=q^{-4}$, given
 by Proposition~\ref{prop the awful case}, \eqref{this is eleven}.  There $\Cu^*(q^{m+n-1})=0$ and the remaining terms are each evaluated in terms of Gauss sums by
 equation \eqref{special Cu sum} to get the result shown.
For contribution \eqref{this is twelve}, we note first that by the vanishing criterion for $\Cu^*$, in this case $\Cu^*(q^{n+m-1})=\Cu^*(q^{n+m-1})[\ell]=0$ for each $\ell$.  Combining this with equation \eqref{C and Cstar}, 
contribution \eqref{this is twelve} equals $\sum_\ell (\Cu(q^{m+n-2})[\ell]-\tilde{\Cu}(q^{m+n-2})[\ell])$, and this is zero by equation \eqref{C sum vs tilde C sum}.  Contribution
\eqref{this is thirteen} is also zero; this follows immediately since each $\Cu^*$ there is zero. 

The case $|a_{2m+n}|=|b_{n+2n}|=q^{-4}$ is Proposition~\ref{prop the awful case}, contribution \eqref{this is fourteen}.
 Substituting in, this is 
 $\sum_\ell \Cu(q^k)[\ell]-2\Cu^*(q^k)-2\tilde{\Cu}(q^k)$ with $k=m+n-2$. 
 Since $|a|^{-1}q^k<1$ and $|b|^{-1}q^k\leq 1$ while $|a^{-1}b^{-1}|q^{3k}=q^2$, we see that $\Cu^*(q^k)=0$. Using equation \eqref{C and Cstar} we find that
$\Cu(q^k)[\ell]=\tilde{\Cu}(q^k)=\Cu(q^{k-1})=q^{k-1}$ (the last equality holding since the integrand is $1$).  

To conclude the proof, we observe that the remaining contributions
\eqref{this is fifteen}, \eqref{this is sixteen}, \eqref{this is seventeen}, and \eqref{this is eighteen} are all zero.  Contribution~\eqref{this is fifteen} vanishes since
if $\abs{a}=q^{n+2m-4}$, $\abs{b}\in \{q^{m+2n-5},q^{m+2n-6}\}$, then $\Cu^*(q^{m+n-2})=\Cu^*(q^{m+n-2})[\ell]=\tilde{\Cu}^*(q^{m+n-2})[\ell]=0$ for each $\ell$ (using $|ab|^{-1}q^{3(m+n-2)}\in \{q^3,q^4\}$).
Using equation \eqref{C and Cstar}  the contribution reduces to $\sum_\ell\Cu(q^{m+n-3})[\ell]-3\tilde{\Cu}(q^{m+n-3})[\ell]$, and this vanishes by equation \eqref{C sum vs tilde C sum}.
Contribution \eqref{this is sixteen} is zero since $\Cu^*(q^{m+n-2})[\ell]=\Cu^*(q^{m+n-2})=0$.  Indeed when $\abs{a}=q^{n+2m-4}$, $\abs{b}\le q^{m+2n-7}$, one has
$|a|^{-1}q^{m+n-2}=q^{2-m}\leq1$ so $\Cu^*(q^{m+n-2})[\ell]=\Cu^*(q^{m+n-2})$ for each $\ell$.  Then since $|a^{-1}b^{-1}|q^{3(m+n-2)}\ge q^5$
and $|a^{-1}|q^{2(m+n-2)}=q^n\ne1$, it follows that $\Cu^*(q^{m+n-2})=0$. 
For contribution \eqref{this is seventeen} we first observe that $\Cu^*(q^{\lfloor \frac{n+1-\val(a)}2\rfloor})[\ell]$ vanishes.  This follows from
the criterion \eqref{vanishing conditions for Cu star} exactly as in the proof of Lemma~\ref{cancellation lemma}.
The remaining terms cancel by this same lemma, where in using it we switch $a$ and $b$ (this interchanges $\Cu$ and $\tilde{\Cu})$ and replace $(m,n)$ by $(n-1,m+1)$.
Contribution \eqref{this is eighteen} vanishes immediately unless $n-\val(a)$ is odd, and
in that case since each $\Cu^*$ is zero.   Indeed, we have $\val(a)\geq -(2m+n)$, and $|a^{-1}b^{-1}|\geq q^{3+m-n+2\val(a)}$. Using this, the vanishing conditions
\eqref{vanishing conditions for Cu star} for 
$\Cu^*(q^{\lfloor \frac{n-\val(a)}2\rfloor})[\ell]$ and $\Cu^*(q^{\lfloor \frac{n+1-\val(a)}2\rfloor})[\ell]$ are again satisfied.
This concludes the proof of the corollary.
 \end{proof}

\subsubsection{Consolidated formulas for the orbital integral $I_{m,n}(a,b)$}

In this Section we consolidate and simplify the formulas we have obtained above, so that they may easily be compared to the formulas we will obtain on the metaplectic
side.  In doing so, we present the answer in the notation of Section~\ref{Cubic exponential notation}.  Let $u=(54)^{-1}\in \O^*$.
Then $\Cu(u^{-1}a,-2u^{-2}b;k)=\Cu(-3a,b;k)$ for all $a,b\in F$, $k\in \Z$
and $\g_{ua}=\g_{2^{-1}a}$ for all $a\in F^*$.

\begin{proposition}\label{Final-form-I-OI} Let $m,n\geq0$ with $m+n\geq1$. For $a,b\in F^*$ such that $\abs{b_{m+2n}}\le \abs{a_{n+2m}}$ we have
$q^{2(m+n)}I_{m,n}(-ua,ub)=$
$$
\begin{cases}
1 &  \abs{b_{m+2n}}=\abs{a_{n+2m}}=1 \\ 
\g_{ab}+\bar\g_{ab} &  n\ge 1,\ \abs{b_{m+2n}}=q^{-1}, \abs{a_{n+2m}}=1\\
q & n\ge 1,\ \abs{b_{m+2n}}=q^{-2}, \abs{a_{n+2m}}=1\\
\abs{b}^{-1}\Cu(-3b^{-1},a^{-1}b^{-1};m) & n=0,  \abs{b_{m+2n}}\le q^{-1},\ \abs{a_{n+2m}}=1\\
2q & \abs{b_{m+2n}}= \abs{a_{n+2m}}=q^{-1} \\    
2q^2 & n\ge 1,\ \abs{b_{m+2n}}=q^{-2},\ \abs{a_{n+2m}}=q^{-1} \\  
q^2[\g_{ab}+\bar \g_{ab}] & n\ge 2,\ \abs{b_{m+2n}}= q^{-3},\,\abs{a_{n+2m}}=q^{-1}\\
2q^2\abs{b}^{-1}\Cu(-3b^{-1},a^{-1}b^{-1};m) & n=1,\,\ \abs{b_{m+2n}}\le q^{-3},\,\abs{a_{n+2m}}=q^{-1}\\
3q^2[\g_{ab}+\bar \g_{ab}] & m,n\ge 1,\ \abs{b_{m+2n}}= \abs{a_{n+2m}}=q^{-2} \\
3q^2\abs{a}^{-1}\Cu(-3a^{-1},a^{-1}b^{-1};n-1)  & m=0,\ n\ge1, \ \abs{b_{m+2n}}= \abs{a_{n+2m}}=q^{-2} \\
3q^2\abs{b}^{-1} \Cu(-3b^{-1},a^{-1}b^{-1};m-1) & n=0,\ m\ge1, \ \abs{b_{m+2n}}\le \abs{a_{n+2m}}=q^{-2} \\
2q^3 & n\ge 2,\ \abs{b_{m+2n}}=q^{-3},\ \abs{a_{n+2m}}=q^{-2} \\
q^4 & n\ge 2,\ \abs{b_{m+2n}}=q^{-4},\ \abs{a_{n+2m}}=q^{-2} \\
3q^3 & n=1,\ \abs{b_{m+2n}}=q^{-3},\, \abs{a_{n+2m}}=q^{-2} \\
2q^4& m,n\ge 1,\ \abs{b_{m+2n}}=\abs{a_{n+2m}}=q^{-3} \\
q^5 & m,n\ge 1,\ \abs{b_{m+2n}}=\abs{a_{n+2m}}=q^{-4} \\
q^4[\g_{ab}+\bar\g_{ab}] & n\ge 2,\ m\ge 1,\ \abs{b_{m+2n}}=q^{-4},\ \abs{a_{n+2m}}=q^{-3}\\
q^4\abs{a}^{-1}\Cu(-3a^{-1},a^{-1}b^{-1};n-2)  & n\ge 2,\ m=0,\ \abs{b_{m+2n}}=q^{-4},\ \abs{a_{n+2m}}\in \{q^{-3},q^{-4}\}\\
2q^4\abs{b}^{-1}\Cu(-3b^{-1},a^{-1}b^{-1};m-1)& n=1,m\ge 1,\ \abs{b_{m+2n}}\le q^{-4},\ \abs{a_{n+2m}}= q^{-3}\\
q^2\abs{ab}^{-1}\sum_{\ell=0}^2 \Cu(-3(a^{-1}+\rho^\ell b^{-1}),a^{-1}b^{-1};\lfloor-\frac{\val(b)}2\rfloor) & n=1,\ m=0,\ ,\ \abs{b_{m+2n}}=\abs{a_{n+2m}}\le q^{-3}  \\
q^2\abs{ab}^{-1}\sum_{\ell=0}^2 \Cu(-3(b^{-1}+\rho^\ell a^{- 1}),a^{-1}b^{-1};\lfloor-\frac{\val(a)}2\rfloor) & n=1,\ m=0 ,\ \abs{b_{m+2n}}<\abs{a_{n+2m}}\le q^{-3}  \\
q^4\abs{b}^{-1}\Cu(-3b^{-1},a^{-1}b^{-1};m-2)& n=0,\ m\ge 2 ,\ \abs{b_{m+2n}}\le \abs{a_{n+2m}}= q^{-4}  \\
q^2\abs{ab}^{-1}\sum_{\ell=0}^2 \Cu(-3(b^{-1}+\rho^\ell a^{-1}),a^{-1}b^{-1};\lfloor-\frac{\val(a)}2\rfloor) & n=0,\ m=1 ,\ \abs{b_{m+2n}}\le \abs{a_{n+2m}}\le q^{-4}  \\
0&otherwise.
\end{cases}
$$
\end{proposition}

\begin{proof}
This follows by checking that the formulas presented above combine to give the evaluation shown.  
For $m=1, n=0$, the evaluation is in Lemma~\ref{I when m=1, n=0}.
For $m>1$, $n=0$ it is in Proposition~\ref{I int (m,0), m ge 2}.  For $m=0$, $n=1$, it is given by Lemma~\ref{I integral with m=0 and n=1},
and $m=0$, $n>1$, by Corollary~\ref{Simplification of I_{0,n}}. Here the expression above uses
that $\g_{a^{-1}b^{-1}}+\overline{\g}_{a^{-1}b^{-1}}=\g_{ab}+\overline{\g}_{ab}$.
For $m\geq1$, $n=1$  the result is in Corollary~\ref{I integral for any m and n=1} and for $m=1$, $n\ge2$ it is in Corollary~\ref{I integral for n ge 2 and m=1}.
Finally, the case $m\ge2$, $n\ge2$ is handled in Corollary~\ref{I integral for n ge 2 and m ge 2}.  Combining these cases completes the proof.
\end{proof}

\section{Formulas for the orbital integrals for Hecke functions on $\PGL_3(F))$: the non-generic orbits}\label{sec I comp - small}

In this Section we evaluate the relative orbital integrals for the remaining orbits.   We treat
the three isolated orbits
\[
\xi[1]=(0,0,0,1,0,\rho^2,0,0),\quad \xi[2]=(0,0,0,1,0,\rho,\frac12\rho^2,0),\quad  \xi[3]=(0,1,0,0,0,0,0,0,0)
\]
and the two one parameter families $\Xi_i=\{\xi_i(a): a\in F^*\}$, $i=1,2$ where 
\[
\xi_1(a)=(0,1,0,a,-\frac12 \rho^2a^{-2},0,0,0)\ \ \ \text{and}\ \ \ \xi_2(a)=(0,1,\frac12 \rho a^{-2}, 0,0,a,0,0).
\]

\subsection{The isolated orbits}

For each isolated orbit $\xi[\ell]$, $\ell=1,2,3$, we first compute
$$I(\xi[\ell];i,j)=\delta_{B}(d_{i,j})\O(\xi[\ell],\omega(d_{i,j}^{-1})\phi_\circ).
$$
For each $\ell$, a direct computation shows that
\begin{equation}\label{degenerate I(*;i,j) integral}
I(\xi[\ell];i,j)=\sum_{n=0}^{\min(i,j)}q^n.
\end{equation}
For example, using \cite{FO}, Lemma 4.3, we have 
$$H_{\xi[1]}=\left\{\begin{pmatrix} 1 & x \\ 0 & 1\end{pmatrix}:x\in F\right\}\times\{n\in N: y=-x\}.$$  Using the formula in \cite{FO}, Section 4.2.2 and
the action of $\omega(d_{i,j}^{-1})$ (see Section~\ref{Coordinate form of the oi for I}), we find that 
\begin{multline*}I(\xi[1];i,j)=\int_F\int_{F^*}\phi_\circ[(0,0,0,\varpi^jt,0,\rho^2\varpi^i t,0,xt)]|t|^2\,d^*t\,\psi(x)\,dx\\
=\int_{|t|\leq \min(q^i,q^j)}|t|^2\int_{|x|\leq|t|^{-1}}\psi(x)\,dx\,d^*t=\int_{1\leq|t|\leq \min(q^i,q^j)}|t|\,d^*t=\sum_{n=0}^{\min(i,j)}q^n.
\end{multline*}
Similarly, using \cite{FO}, Sections 4.2.3 and 4.2.5, resp., we obtain
$$I(\xi[2];i,j)=\int_{F^*}\int_F \phi_\circ[(0,0,0,\varpi^j t,0,\rho \varpi^i t, {\tfrac12}\rho^2t^{-1},a)]\, da\,|t|\,d^*t$$
and
$$I(\xi[3];i,j)=\int_{F^3}\int_{F^*} \phi_\circ[(0,\varpi^{i+j}t,0,\varpi^j xt,0,-\varpi^i yt,0,zt)]|t|^2\psi(x+y)\,d^*t\,d(x,y,z).$$
The evaluations of these integrals are straightforward and again give \eqref{degenerate I(*;i,j) integral}.

The value of the orbital integral is then obtained by combining \eqref{degenerate I(*;i,j) integral} with Corollary~\ref{cor exp j-G}.  For example,
writing $s(m)=(q^{m+1}-1)/(q-1)$, if $m\geq1$ then
\begin{multline*}
\O(\xi[\ell],\omega(f_{m,1})\phi_\circ)=\\
s(1)-(\delta_{m,2}+\delta_{m,3}s(1)+\delta_{m\ge4}s(2))-(\delta_{m,1}(1+q)+\delta_{m\ge2}q)+q(\delta_{m,3}+\delta_{m\ge4}s(1))=0.
\end{multline*}
The other cases are similar.  We obtain the following evaluation.
\begin{proposition}\label{isolated-cells-I}
The orbital integral attached to each isolated orbit $\xi[\ell]$, $\ell=1,2,3$, is given by
\begin{equation*}
\O(\xi[\ell],\omega(f_{m,n})\phi_\circ)=\begin{cases} 1&m+n\leq1\\
0&otherwise.
\end{cases}
\end{equation*}
\end{proposition}

\subsection{The orbital integrals for the one parameter families}

For these evaluations, we again use Corollary~\ref{cor exp j-G}. 
For each one parameter family $\xi_\ell(a)$, $a\in F^*$, we must compute the orbital integral $I(\xi_\ell(a);i,j).$

\subsubsection{Coordinate forms and symmetry for the one-parameter orbital integrals}\label{subsubsec: coordinates and symmetr-degenerate}
Let $\eta_{i,j}$ be the characteristic function of 
$$\p^{-(i+j)}\times\p^{-(i+j)}\times \p^{-j}\times\p^{-j}\times\p^{-i}\times\p^{-i}\times \O\times \O.$$
As in \cite{FO}, Section 11.1, and incorporating the action of $\omega(d_{i,j}^{-1})$ (see Section~\ref{Coordinate form of the oi for I}), we have
\begin{multline*}
I(\xi_1(a);i,j)=
\int_{F^*}\int_{F^3} \eta_{i,j}(0,t^{-1},0,(a+x)t^{-1},a^{-2}t,-t^{-1}y,a^{-2}t x,t^{-1}(\rho ay-\rho^2 xy- z)))\\
\psi(a^{-2}x z-\rho a^{-1}xy-\rho^2a^{-1}z+x+y) \ dx\ dy\ dz\abs{t}^{-2}\ d^*t.
\end{multline*}
Similarly,
\begin{multline*}
I(\xi_2(a);i,j)=
\int_{F^*}\int_{F^3} \eta_{i,j}(0,t^{-1},{\tfrac12}\rho a^{-2}t,t^{-1}x,0,(a-y)t^{-1},{\tfrac12}a^{-2}ty,t^{-1}(xa-xy-z\rho)]\\
\psi(-(a^{-2} y+\rho a^{-1})z+a^{-2}xy^2-a^{-1}xy+x+y) \ dx\ dy\ dz\abs{t}^{-2}\ d^*t
\end{multline*}

We emphasize that these expressions are derived from the adjoint embedding of $\PGL_3$ into $\SO_8$ composed with a further embedding of $\SL_2\times\SO_8$ into $\Sp_{16}$.
This composition depends on the choice of primitive cubic root of unity $\rho\in F$.   See the explicit formulas in \cite{FO}, Section 3.1. 
We write $\O_\rho$ (resp.\ $I_\rho$) for the orbital integral $\O$ (resp.\ $I$) when it is important to record this dependence.
Changing $(x,y,z)$ to $(-y,-x,-z+xy)$ and replacing $\rho$ by $\rho^2$ in the last expression gives
\begin{multline*}
I_{\rho^2}(\xi_2(a);i,j)=
\int_{F^*}\int_{F^3} \eta_{i,j}(0,t^{-1},a^{-2}t,t^{-1}y,0,(a+x)t^{-1},a^{-2}tx,t^{-1}(ya-\rho xy-z\rho^2)]\\
\psi(-a^{-2} xz+\rho^2 a^{-1}z+\rho a^{-1}xy-x-y) \ dx\ dy\ dz\abs{t}^{-2}\ d^*t.
\end{multline*}
We conclude that
\[
\O_{\rho^2}(\xi_2(a),\eta_{i,j})=\overline{\O_{\rho}(\xi_1(a),\eta_{j,i})}.
\]
Then using Corollary~\ref{cor exp j-G} we obtain
\begin{equation}\label{degenerate orbit I symmetry}
\O_{\rho^2}(\xi_2(a),\omega(f_{m,n})\phi_\circ)=\overline{\O_{\rho}(\xi_1(a),\omega(f_{n,m})\phi_\circ)}.
\end{equation}
We remark that it is possible to formulate a symmetry property for these orbital integrals similar to that in Section~\ref{symmetry-J} and to use it to establish \eqref{degenerate orbit I symmetry},
but it is more involved in the relative situation and the above argument is much easier.

Because of the equality \eqref{degenerate orbit I symmetry}, it suffices to directly compute only one of the one-parameter relative orbital integrals.  We proceed to do so.  Though the computation is less involved than the two-parameter case, we arrive at complicated expressions.  Once again, 
a recursive argument (Section~\ref{sec: the recursive step, one parameter integral}) allows us to compute them in full.

\subsubsection{Reduction to a double integral}

From now on we write $I(a;i,j)$ in place of $I(\xi_1(a);i,j)$. 
The change of variables $z\mapsto z+(\rho a-\rho^2 x)y$ shows that
\begin{multline*}
I(a;i,j)=
\int_{F^*}\int_{F^3} \eta_{i,j}(0,t^{-1},0,(a+x)t^{-1},a^{-2}t,t^{-1}y,a^{-2}t x,t^{-1}z)\times \\
\psi((a^{-2}x -\rho^2a^{-1})z+\rho a^{-1} x(1-\rho a^{-1} x)y+x)  \ dx\ dy\ dz\abs{t}^{-2}\ d^*t.
\end{multline*}
Note that
\begin{multline*}
\eta_{i,j}(0,t^{-1},0,(a+x)t^{-1},a^{-2}t,t^{-1}y,a^{-2}t x,t^{-1}z)=\\ \triv_{\p^{-i}}(t^{-1}y)\triv_\O(t^{-1}z)\eta_{i,j}(0,t^{-1},0,(a+x)t^{-1},a^{-2}t,0,a^{-2}t x,0)
\end{multline*}
where for any set $R\subset F$, $\triv_R$ denotes the characteristic function of $R$, and
\[
\int_F \triv_\O(t^{-1}z)\psi(\alpha z)\ dz=\abs{t}\triv_\O(\alpha t)\ \ \ \text{and} \int_F \triv_{\p^{-i}}(t^{-1}y)\psi(\beta y)\ dy=q^i\abs{t}\triv_{\p^i}(\beta t).
\]
Apply this to $\alpha=a^{-2}x -\rho^2a^{-1}$ and $\beta=\rho a^{-1} x(1-\rho a^{-1} x)$.
After integrating over $y$ and $z$ we have
\begin{multline*}
I(a;i,j)=q^i
\int_{F^*}\int_{F} \triv_{\p^{-(i+j)}}(t^{-1})\triv_{\p^{-j}}((a+x)t^{-1})\triv_{\p^{-i}}(a^{-2}t)\triv_{\O}(a^{-2}t x) \\ 
\triv_\O(ta^{-2}(a-\rho x))\triv_{\p^i}(ta^{-2}x(a-\rho x))
\psi(x)  \ dx\ d^*t.
\end{multline*}

After the variable changes $t\mapsto at$ and $x\mapsto ax$ we see that
\[
I(a;i,j)=q^i\abs{a}\int_\D\psi(ax)\ dx\ d^*t
\]
where the domain $\D\subseteq F\times F^*$ is defined by $(x,t)\in \D$ whenever
\[
q^{-(i+j)}\abs{a}^{-1}\le \abs{t}\le \min(1,q^i\abs{a}),\ \abs{1+x}\le q^j\abs{t},\ \abs{tx}\le 1,\ \abs{tx(1-\rho x)}\le q^{-i}\abs{a}^{-1}.
\]
This domain is empty unless 
\begin{equation}\label{condition-on-a}
q^{-(2i+j)}\leq |a|^2
\end{equation} 
or $2\val(a)\leq 2i+j$ so we assume this from now on.

To carry out the integration, we will repeatedly use the simple observation that 
\begin{equation}\label{simple integral}
\int_{b+\p^k}\psi(ax)\ dx=\delta(\abs{a}\le q^k)q^{-k}\psi(ab),
\end{equation}
and the evaluation of 
\[
L(e;q^a,q^b):=\int_{q^a\le \abs{t}\le q^b}\abs{t}^e\ d^*t
\]
for $e\in \{-1,0,1\}$.  This is
\begin{equation}\label{eq lab}
L(e;q^a,q^b)=\delta(a\le b)\begin{cases}
\frac{q^{b+1}-q^a}{q-1} & e=1\\
b-a+1 & e=0 \\
\frac{q^{1-a}-q^{-b}}{q-1} & e=-1.
\end{cases}
\end{equation}
We will also need the integral
\begin{align}
L_n(q^a,q^b)&=\int_{q^a\le \abs{t}\le q^b}q^{\lfloor\frac{n+\val(t)}2\rfloor}\ d^*t\notag\\
&=\delta(a\leq b)[2q^{\lfloor\frac{n-b}2\rfloor}\frac{q^{\lfloor{\frac{n-a}2}\rfloor-\lfloor{\frac{n-b}2\rfloor}+1}-1}{q-1}
-\delta_{n-a\text{~even}}q^{\lfloor\frac{n-a}2\rfloor}-\delta_{n-b~\text{odd}}q^{\lfloor\frac{n-b}2\rfloor}].
\label{L-en-integral}
\end{align}

\subsubsection{Calculation of $I(a;i,j)$ when $|a_{-i}|\leq1$.}

We consider first the case $\abs{a}\le q^{-i}$. 
\begin{lemma}\label{compute I when a sub -i leq1}
Suppose $|a_{-i}|\leq1$ and \eqref{condition-on-a} holds.  Then
$$I(a;i,j)=\begin{cases}\frac{q^{\lfloor\frac{2i+j-2\val(a)}{3}\rfloor+1}-1}{q-1}+
q^i|a|L_{\val(a)-i}(q^{\lfloor\frac{\val(a)+3-i-2j}3\rfloor},q^i|a|)
&|a|^2\le q^{2+i-j}\\
3\frac{q^{i+1}-1}{q-1}&|a|^2\geq q^{3+i-j}.
\end{cases}$$
\end{lemma}

\begin{proof}
In this case the domain $\D$ simplifies to
\[
q^{-(i+j)}\abs{a}^{-1}\le \abs{t}\le q^i|a|,\ \abs{x}\le q^j\abs{t},\ \abs{tx}\le 1,\ \abs{tx^2}\le q^{-i}\abs{a}^{-1}.
\]
In this domain, $|t^{-1}|\geq q^{\lfloor{\frac{\val(a)-i+\val(t)}{2}}\rfloor}$, so the condition $|tx|\leq1$ is superfluous.

We have
\[
\int_{\abs{x}\le q^j\abs{t},\ \abs{tx^2}\le q^{-i}\abs{a}^{-1}}\psi(ax)\ dx=\begin{cases}
q^j\abs{t} & \abs{t}\le q^{-j}\abs{a}^{-1},\ \abs{t}^3\le q^{-i-2j}\abs{a}^{-1}\\
q^{\lfloor \frac{\val(t)+\val(a)-i}2\rfloor} & q^{-i-1}\abs{a}\le \abs{t},\ q^{1-i-2j}\abs{a}^{-1}\le \abs{t}^3\\
0 & \text{otherwise}
\end{cases}
\]
and therefore
\begin{multline*}
I(a;i,j)=q^{i+j}\abs{a}L(1;q^{-i-j}\abs{a}^{-1},\min(q^{-j}\abs{a}^{-1},q^{\lfloor\frac{\val(a)-i-2j}3\rfloor},q^i|a|))\\+
q^i|a|L_{\val(a)-i}(\max(q^{-i-j}\abs{a}^{-1},q^{-i-1}\abs{a},q^{\lfloor\frac{\val(a)+3-i-2j}3\rfloor}),q^{i}\abs{a}).
\end{multline*}
For the first term,  in our region
$$\min(q^{-j}\abs{a}^{-1},q^{\lfloor\frac{\val(a)-i-2j}3\rfloor},q^i|a|)=\begin{cases} q^{\lfloor\frac{\val(a)-i-2j}3\rfloor}&|a|^2\leq q^{2+i-j}\\
q^{-j}|a|^{-1}
&|a|^2\geq q^{3+i-j}
\end{cases}
$$
and for
the second term 
$$\max(q^{-i-j}\abs{a}^{-1},q^{-i-1}\abs{a},q^{\lfloor\frac{\val(a)+3-i-2j}3\rfloor})=\begin{cases} q^{\lfloor\frac{\val(a)+3-i-2j}3\rfloor}&|a|^2\le q^{2+i-j}\\
 q^{-i-1}|a|&|a|^2\geq q^{3+i-j}.
\end{cases}
$$
(For the last max, note that $q^{\lfloor\frac{\val(a)+3-i-2j}3\rfloor}=q^{-i-1}|a|$ when $|a|^2=q^{3+i-j}$ or $|a|^2=q^{2+i-j}$.)
Simplifying using \eqref{eq lab} and \eqref{L-en-integral}, the result follows. 
\end{proof}

\subsubsection{Calculation when $|a_{-i}|>1$: decomposition of the domain}\label{Case a sub minus i greater than one-1}
Suppose now that $|a_{-i}|>1$. Then condition \eqref{condition-on-a} holds.   

We write $\D$ as the disjoint union of four subdomains:
\begin{itemize}
\item $\D^\circ=\{(x,t)\in\D: \abs{1+x}=\abs{1-\rho x}=1\}$
\item $\D^{-1}=\{(x,t)\in\D: \abs{1+x}<1\}$
\item $\D^{\rho^2}=\{(x,t)\in\D: \abs{1-\rho x}<1\}$
\item $\D^{>1}=\{(x,t)\in\D: \abs{x}>1\}$
\end{itemize}
and write $I^\star(a;i,j)=q^i\abs{a}\int_{\D^\star}\psi(ax)\ dx\ d^*t$ where $\star\in\{\circ,-1,\rho^2,>1\}$.  

\subsubsection{Computation of $I^\circ(a;i,j)$}\label{Case a sub minus i greater than one-2}

The first contribution is given as follows.
\begin{lemma}\label{compute D circ}
Suppose $|a_{-i}|>1$.  Then
\[
I^\circ(a;i,j)=\begin{cases}
\frac{q^i|a|-1}{q-1}+q^{i-1}\abs{a}[\val(a)+j+1-i](q-2)&  \abs{a}\le \min(1,q^{j-i})\\
\frac{q^{i+1}-1}{q-1}-q^i (j-i)[\psi(-a)+\psi(\rho^2a)]&q=\abs{a}\text{ and }j\ge i+1\\
\frac{q^{\min(i,j)+1}-1}{q-1}& \text{otherwise.}
\end{cases}
\]
\end{lemma}
Here the last expression holds when ($q=\abs{a}$ and $j\le i$) or $q^2\le \abs{a}$.

\begin{proof}
Suppose $|a_{-i}|>1$.  Then $D^\circ$ is defined by
\[
\abs{1+x}=\abs{1-\rho x}=1,q^{-j}\le  \abs{t}\le 1,\  \abs{tx}\le q^{-i}\abs{a}^{-1}.
\]
Since the condition $\abs{1+x}=\abs{1-\rho x}=1$ rules out two classes of $x$ modulo $\p$, we have 
\begin{equation}\label{I circ inclusion exclusion}
I^\circ(a;i,j)=I_{\le 1}(a;i,j)-I_{-1}(a;i,j)-I_{\rho^2}(a;i,j)
\end{equation}
where we set $\E_{\le 1}$ to be the domain defined by
\[
\abs{x}\le 1,\ q^{-j}\le \abs{t}\le 1,\ \abs{tx}\le q^{-i}\abs{a}^{-1}
\]
and for $u\in \O^*$ we set $\E_u$ to be the domain defined by
\[
x\in u+\p, q^{-j} \le \abs{t}\le q^{-i}\abs{a}^{-1}
\]
and $I_{\star}(a;i,j)=q^i\abs{a}\int_{\E_\star}\psi(ax)\ dx\ d^*t$ for $\star\in \{\le 1\}\sqcup \O^\times$.

To compute $I_u$, we first integrate over $x$ using \eqref{simple integral} and then over $t$.  This gives
\[
I_u(a;i,j)=\delta(\abs{a}\le q)q^{i-1}\abs{a}\psi(au)L(0;q^{-j},q^{-i}|a|^{-1}).
\]
For $I_{\le 1}$ we first compute
\[
\int_{\abs{x}\le \min(1,q^{-i}\abs{at}^{-1})}\psi(ax)\ dx=\begin{cases}
1&\abs{a}\le 1, \abs{t}\le q^{-i}\abs{a}^{-1}\\
q^{-i}\abs{at}^{-1} & \max(q^{-i},q^{1-i}\abs{a}^{-1})\le \abs{t}\\
0 & \text{otherwise}
\end{cases}
\]
so that
$$
I_{\le 1}(a;i,j)=\delta(\abs{a}\le 1)q^i \abs{a}L(0;q^{-j},q^{-i}\abs{a}^{-1} )+\\
L(-1;\max(q^{-j},q^{-i},q^{1-i}\abs{a}^{-1}),1).
$$
Substituting the above expressions into \eqref{I circ inclusion exclusion} and evaluating using \eqref{eq lab}, we obtain the expression of the lemma.
\end{proof}

\subsubsection{Computation of $I^{-1}(a;i,j)$}\label{Case a sub minus i greater than one-3} The second contribution is given as follows.

\begin{lemma}\label{compute D -1}
Suppose $|a_{-i}|>1$.  Then
$$
I^{-1}(a;i,j)=\psi(-a)\times
\begin{cases}\frac{q^i\abs{a}-1}{q-1} +q^{i-1}|a|(\val(a)+j+1-i)&  \abs{a}\le \min(1,q^{j-i})\\
\frac{q^{i+1}-1}{q-1}+q^i (j-i)& q=\abs{a}\text{ and }j\ge i+1\\
\frac{q^{\min(i,j)+1}-1}{q-1}& \text{otherwise.} 
\end{cases}
$$
\end{lemma}

\begin{proof}
Note that $\D^{-1}$ is given by
\[
\abs{1+x}\le \min(q^{-1},q^j\abs{t}),\ q^{-(i+j)}\abs{a}^{-1}\le \abs{t}\le q^{-i}\abs{a}^{-1}.
\]
We have
\[
\int_{\abs{1+x}\le \min(q^{-1},q^j\abs{t})}\psi(ax)\ dx=\begin{cases}
\psi(-a)q^j\abs{t} & \abs{t}\le \min(q^{-1-j},q^{-j}\abs{a}^{-1}) \\
\psi(-a)q^{-1} & \abs{a}\le q\text{ and }\abs{t}\ge q^{-j} \\
0 & \text{otherwise.}
\end{cases}
\]
and therefore
\begin{multline*}
I^{-1}(a;i,j)=q^{i+j}\abs{a}\psi(-a)L(1;q^{-i-j}\abs{a}^{-1},\min(q^{-i}\abs{a}^{-1},q^{-j}\abs{a}^{-1},q^{-1-j})\\
+ \delta(|a|\leq q)q^{i-1}|a|\psi(-a) L(0;q^{-j},q^{-i}|a|^{-1}).
\end{multline*}
Applying \eqref{eq lab} and simplifying, we obtain the expression shown. 
\end{proof}

\subsubsection{Computation of $I^{\rho^2}(a;i,j)$}\label{Case a sub minus i greater than one-4}
 The third evaluation is similar to the previous one.

\begin{lemma}\label{compute D rho squared}
Suppose $|a_{-i}|>1$.  Then
$$I^{\rho^2}(a;i,j)=
\psi(\rho^2 a)\times
\begin{cases}\frac{q^i\abs{a}-1}{q-1} +q^{i-1}|a|(\val(a)+j+1-i)& q^{1-i}\le \abs{a}\le \min(1,q^{j-i})\\
\frac{q^{i+1}-1}{q-1}+q^i (j-i)& q=\abs{a}\text{ and }j\ge i+1\\
\frac{q^{\min(i,j)+1}-1}{q-1}& \text{otherwise.} 
\end{cases}
$$
\end{lemma}

\begin{proof}
The domain $\D^{\rho^2}$ is given by
\[
\abs{x-\rho^2}\le \min(q^{-1},q^{-i}\abs{at}^{-1}),\ q^{-j} \le \abs{t}\le 1.
\]
For this case, we have
\[
\int_{\abs{x-\rho^2}\le \min(q^{-1},q^{-i}\abs{at}^{-1})}\psi(ax)\ dx=\begin{cases}
\psi(\rho^2 a)q^{-i}\abs{at}^{-1} & \max(q^{1-i}|a|^{-1},q^{-i})\le \abs{t}\\
\psi(\rho^2 a)q^{-1} & \abs{a}\le q\text{ and }\abs{t}\le q^{-i}\abs{a}^{-1}\\
0 & \text{otherwise}
\end{cases}
\]
and therefore
\begin{multline*}
I^{\rho^2}(a;i,j)=\psi(\rho^2a)L(-1;\max(q^{-i},q^{-j},q^{1-i}|a|^{-1}),1)\\
+\delta(|a|\leq q) q^{i-1}|a|\psi(\rho^2 a)L(0;q^{-j},q^{-i}|a|^{-1}).
\end{multline*}
Again applying \eqref{eq lab} and simplifying gives the result shown.
\end{proof}

\subsubsection{Computation of $I^{>1}(a;i,j)$}\label{Case a sub minus i greater than one-5}

We turn to the fourth and final contribution. 
\begin{lemma}\label{compute D >1}
Suppose $|a_{-i}|>1$. 
\begin{multline*}
I^{>1}(a;i,j)
=q^{i+j}\abs{a}L(1;q^{1-j},  \min(q^{-i-2}\abs{a}^{-1},q^{-j}\abs{a}^{-1},q^{\lfloor\frac{\val(a)-i-2j}3\rfloor})
\\ + q^i|a| L_{\val(a)-i}(\max(q^{1-j},q^{-i-1}\abs{a},q^{\lfloor\frac{3+\val(a)-i-2j}3\rfloor}),q^{-i-2}\abs{a}^{-1})\\
-\delta(\abs{a}\le 1)q^i\abs{a}L(0;q^{1-j},q^{-i-2}\abs{a}^{-1}).
\end{multline*}
All terms vanish if $j\leq 4$.
\end{lemma}

\begin{proof}
When $|a_{-i}|>1$, the domain $\D^{>1}$ is given by
\[
q\le \abs{x}\le \min(q^j\abs{t},\abs{t}^{-1},q^{\lfloor\frac{\val(a)+\val(t)-i}2\rfloor})),\ q^{-(i+j)}\abs{a}^{-1}\le \abs{t}\le 1.
\]
Moreover, $|t|^{-1}<q^{\lfloor\frac{\val(a)+\val(t)-i}2\rfloor}$ implies that $|t|\geq q^2|a_{-i}|$ and this can not hold in this domain.
We observe that
\begin{multline*}
\int_{q\le \abs{x}\le \min(q^j\abs{t},q^{\lfloor\frac{\val(a)+\val(t)-i}2\rfloor}))}\psi(ax)\ dx=\delta(q^{1-j}\le \abs{t}\le q^{-i-2}\abs{a}^{-1})\times  \\ 
\left[\int_{\abs{x}\le \min(q^j\abs{t},q^{\lfloor\frac{\val(a)+\val(t)-i}2\rfloor})}\psi(ax)\ dx-\int_{\abs{x}\le 1}\psi(ax)\ dx\right].
\end{multline*}
Note that the domain
$
q^{1-j} \le \abs{t}\le q^{-i-2}\abs{a}^{-1}
$
 is empty and therefore $I^{>1}(a;i,j)=0$ unless $ \abs{a_{-i}}\le q^{j-3}$. Since $|a_{-i}|>1$ we must have $j\geq4$.
We suppose that this is the case from now on.

Clearly,
$
\int_{\abs{x}\le 1}\psi(ax)\ dx=\delta(\abs{a}\le 1)
$
while
$$
\int_{\abs{x}\le \min(q^j\abs{t},q^{\lfloor\frac{\val(a)+\val(t)-i}2\rfloor})}\psi(ax)\ dx=
\begin{cases}
q^{j}\abs{t} &  \abs{t}\le q^{-j}\abs{a}^{-1},\abs{t}^3\le q^{-(i+2j)}\abs{a}^{-1}\\
q^{\lfloor\frac{\val(a)+\val(t)-i}2\rfloor} &q^{-i-1}\abs{a}\le \abs{t}, q^{1-i-2j}\abs{a}^{-1}\le \abs{t}^3\\
0 & \text{otherwise}.
\end{cases}
$$
The result follows.   
\end{proof}

\subsubsection{Computation of $I(a;i,0)$ and $I(a;0,j)$}\label{edge-cases}
Using the above formulas, it is straightforward to obtain the following evaluations.
\[
I(a;i,0)=\begin{cases}
1+\psi(-a)+\psi(\rho^2 a) & \abs{a}>1\\
3 & q^{1-i}\le \abs{a}\le 1 \\
1& \abs{a}=q^{-i}\\
0 & \abs{a}\le q^{-i-1}.
\end{cases}
\]

\[
I(a;0,j)=\begin{cases}
1+ \psi(-a)+\psi(\rho^2a) & \abs{a}>1 \\
 3 & q^{1-\lfloor\frac{j}2\rfloor}\le \abs{a}\le 1 \text{ or equivalently }j-2\ge 2\val(a)\ge 0\\
 2 & \abs{a}=q^{-\lfloor\frac{j}2\rfloor}\text{ and }j \text{ is odd or equivalently }j-1=2\val(a) \\ 
 1& \abs{a}=q^{-\lfloor\frac{j}2\rfloor}\text{ and }j \text{ is even or equivalently }j=2\val(a) \\
 0 & \abs{a}\le q^{-(1+\lfloor\frac{j}2\rfloor)} \text{ or equivalently }j+1\le 2\val(a).
\end{cases}
\]

\subsubsection{The recursive step}\label{sec: the recursive step, one parameter integral}
For $i,j\in \Z_{\ge 0}$ let 
\begin{equation}\label{recursive-difference-one-para-orbit}
R(a;i,j)=I(a;i,j)-qI(a;i-1,j-1)
\end{equation}
where we set $I(a;i-1,j-1)=0$ if $0\in\{i,j\}$.
In this case we will show the following result.
\begin{proposition}\label{remarkable recursion}
If $i,j\geq0$ then
$$R(a;i,j)=I(a;0,2i+j).$$
\end{proposition}

\begin{proof}  If $i=0$ this is a tautology, while if $j=0$ it follows from the formulas for $I(a;i,0)$ and $I(a;0,2i)$ above.  Suppose $i,j\geq1$. The right hand side is evaluated in Section \ref{edge-cases}. We evaluate the left hand side. 
If $|a_{-i}|\leq1$ then both $I(a;i,j)$ and $I(a;i-1,j-1)$ are evaluated in Lemma~\ref{compute I when a sub -i leq1}.
If $|a|^2\geq q^{3+i-j}$ it is immediate that $R(a;i,j)=3$ since if $p(n)=(q^n-1)/(q-1)$ then $p(n)-qp(n-1)=1$.
If $q^{3-2i-j}\leq |a|^2\leq q^{2+i-j}$ then both terms in the difference contribute. 
A variable change shows that 
$$L_n(q^a,q^b)=L_{n-1}(q^{a-1},q^{b-1}).$$
Using this and cancelling, we find that 
\begin{multline*}R(a;i,j)=1+q^i|a|\left(L_{\val(a)-i}(q^{\lfloor\frac{\val(a)+3-i-2j}3\rfloor},q^i|a|)-L_{\val(a)-i+1}(q^{\lfloor\frac{\val(a)+3-i-2j}3\rfloor+1},q^{i-1}|a|)\right)\\
=1+q^i|a|L_{\val(a)-i}(q^{i-1}|a|,q^i|a|)=3.
\end{multline*}
It remains to treat the cases $|a|^2=q^{-2i-j+\epsilon}$ with $\epsilon\in\{0,1,2\}$.  Here $I(a;i-1,j-1)=0$ (from \eqref{condition-on-a}).
The formula for $I(a;i,j)$ may be evaluated using \eqref{L-en-integral} and gives $3$ when $\epsilon=2$, $2$ when $\epsilon=1$ and $1$ when $\epsilon=0$.

If $|a_{-i}|=q$ then $|a_{-(i-1)}|=1$, so we compute $I(a;i-1,j-1)$ by Lemma~\ref{compute I when a sub -i leq1}.  Since $i,j\geq1$, we obtain
$$I(a;i-1,j-1)=\begin{cases}\frac{q^{\lfloor\frac{j-1}{3}\rfloor+1}-1}{q-1}+L_{0}(q^{\lfloor\frac{5-2j}3\rfloor},1)&j\leq 3i\\
3\frac{q^i-1}{q-1}&j>3i.
\end{cases}$$
To compute $I(a;i,j)$ we use Lemmas~\ref{compute D circ} to \ref{compute D >1}.  It follows easily that 
$$I^{\circ}(a;i,j)+I^{-1}(a;i,j)+I^{\rho^2}(a;i,j)=3+qj.$$
If $j\leq 3$ then $I^{>1}(a;i,j)=0$; combining these formulas gives $R(a;i,j)=3$.  
If $j\geq4$ then (since $L(0;q^{1-j},q^{-3})=j-3)$ Lemma~\ref{compute D >1} gives 
\begin{multline}\label{I greater than one in a special case} 
I^{>1}(a;i,j)=\\
q^{j+1}L(1;q^{1-j},  \min(q^{-3},q^{i-j-1},q^{\lfloor\frac{-1-2j}3\rfloor})
+q L_{-1}(\max(q^{1-j},q^{-2i},q^{\lfloor\frac{2-2j}3\rfloor}),q^{-3})
-(j-3)q.
\end{multline}
To evaluate the remaining terms in \eqref{I greater than one in a special case}, we observe that
$$\min(q^{-3},q^{i-j-1},q^{\lfloor\frac{-1-2j}3\rfloor})=\begin{cases}q^{\lfloor\frac{-1-2j}3\rfloor}&j\leq 3i\\
q^{i-j-1}&j>3i\end{cases}$$
and
$$\max(q^{1-j},q^{-2i},q^{\lfloor\frac{2-2j}3\rfloor})=\begin{cases}q^{\lfloor\frac{2-2j}3\rfloor}&j\leq 3i\\
q^{-2i}&j>3i.
\end{cases}$$
If $i=1$ the two terms $L(1;*)$ and $L_{-1}(*)$ in \eqref{I greater than one in a special case} are zero by the $\delta(a\leq b)$ in \eqref{eq lab} and \eqref{L-en-integral} and $R(a;i,j)=3$ as desired.
In the remaining cases one sees that $R(a;i,j)=3$ 
directly by using the formulas
\eqref{eq lab} and \eqref{L-en-integral} for $L(1;*)$ and $L_{-1}(*)$ in \eqref{I greater than one in a special case} and simplifying. If $j>3i$ this follows directly.  When $j\leq 3i$ it
is easiest to note that
$$L_{-1}(q^{\lfloor\frac{2-2j}3\rfloor}),q^{-3})=L_0(q^{\lfloor\frac{5-2j}3\rfloor}),q^{-2})$$ 
so $$qL_{-1}(q^{\lfloor\frac{2-2j}3\rfloor}),q^{-3})-qL_{0}(q^{\lfloor\frac{5-2j}3\rfloor},1)=-qL_0(q^{-1},1)=-2q.$$
Then evaluating $L(1,*)$ by \eqref{eq lab} and combining rational functions, one obtains this case of the recursion.
This completes the proof for the case $|a_{-i}|=q$.

Finally if $|a_{-i}|>q$ then both terms are evaluated by the computations of subsections~\ref{Case a sub minus i greater than one-2} to \ref{Case a sub minus i greater than one-5}.
One sees easily from the formulas we have obtained that 
$$
I^\star(a;i,j)-qI^\star(a;i-1,j-1)=\begin{cases}1&\text{when}~\star~\text{is}~\circ\\
\psi(-a)&\text{when}~\star~\text{is}~-1\\
\psi(\rho^2 a)&\text{when}~\star~\text{is}~\rho^2\\
0&\text{when}~\star~\text{is}~>1.
\end{cases}
$$
Hence $R(a;i,j)=1+\psi(-a)+\psi(\rho^2 a)$ in this case, as claimed.
\end{proof}

\subsubsection{Computation of the orbital integrals $\O(\xi_1(a),\omega(f_{m,n})\phi_\circ)$}

We now use this information to complete the evaluation of the orbital integrals $\O(\xi_1(a),\omega(f_{m,n})\phi_\circ)$.  First, we
express these integrals in terms of the differences $R(a;i,j)$.  This is accomplished by the following lemma.

\begin{lemma}\label{OO in terms of R}
 For $m,n\geq0$, $a\in F^*$, we have
\begin{multline*}\O(\xi_1(a),\omega(f_{m,n})\phi_\circ)=\\
\begin{cases}
R(a;n,0)-\delta_{n\ge 2}R(a;n-2,1)&m=0\\
R(a;0,m)-\delta_{m\ge 2}R(a;1,m-2)&n=0\\
R(a;n,1)-\delta_{n\ge 2} R(a;n-2,2)-\delta_{n,1}R(a;0,0)&m=1, n\geq1\\
R(a;1,m)-R(a;2,m-2) &n=1, m\geq2\\
R(a;n,m)-R(a;n+1,m-2)+R(a;n-1,m-1)-R(a;n-2,m+1)&m,n\ge2.
\end{cases}
\end{multline*}
\end{lemma}

\begin{proof}  The statement follows directly from Corollary~\ref{cor exp j-G} together with the definition of $R(a;i,j)$ in \eqref{recursive-difference-one-para-orbit}.
\end{proof}

Finally, combining Proposition~\ref{remarkable recursion} and Lemma~\ref{OO in terms of R} we arrive at the following evaluation.

\begin{proposition}\label{formula for one parameter I orbital integral}
Let $m,n\geq0$ with $m+n\ge1$. 
Suppose $a\in F^*$.   Then
$$\O(\xi_1(a),\omega(f_{m,n})\phi_\circ)=\begin{cases}1&|a_{-n}|=1, m=0\\
3&|a_{-n}|=q,  m=0\\
1+\psi(-a)+\psi(\rho^2 a)&|a_{-1}|=q^{2}, m=0, n=1\\
1&|a_{-n}|=q^{2}, m=0, n>1\\
2&|a_{-n}|=1, m=1\\
2&|a_{-n}|= q, m=1, n>0\\
1+\psi(-a)+\psi(\rho^2 a)&|a_{-(2m+n)}|\geq q^{3}, m+n=1\\
0&\text{otherwise.}\end{cases}$$
\end{proposition} 

\begin{proof} Note that Proposition~\ref{remarkable recursion} implies that 
$$R(a;n,m)=R(a;n+1,m-2)\quad \text{and}\quad R(a;n-1,m-1)=R(a;n-2,m+1).$$  Lemma~\ref{OO in terms of R} then implies that 
$\O(\xi_1(a),\omega(f_{m,n})\phi_\circ)=0$ if $m\ge2$. 
The evaluations when $m=0$ and $m=1$ follow directly by substituting the evaluation of $R(a;i,j)$ into the expressions given in Lemma~\ref{OO in terms of R} and using Proposition \ref{remarkable recursion} and Section \ref{edge-cases}.
\end{proof}

Finally, since the functions $f_{m,n}$ are a basis for $\H$,
as a consequence of this evaluation (and the case $m=n=0$ in \cite{FO}) we have the relation between the orbital integrals for different choices of the cube root of unity $\rho$.

\begin{corollary}\label{change rho to rho squared in OO} For $f\in\H$, $a\in F^*$,
$$\O_{\rho^2}(\xi_1(a),\omega(f)\phi_\circ)=\O_\rho(\xi_1(-\rho a),\omega(f)\phi_\circ).$$
\end{corollary}

\section{Formulas for the big cell orbital integrals for Hecke functions on the cubic cover of $\SL_3(F)$}\label{sec J comp}
In this Section we evaluate the orbital integrals attached to the big cell on the metaplectic group $G'$, that is, the generic family of relevant orbits.    We shall first
evaluate the integrals $J^l(\gamma(a,b);i,j)$ defined by \eqref{the J^l}, and then put them together using Corollary~\ref{cor exp j} in order to evaluate the complete big cell
orbital integral.
We shorten the notation by writing $J^l(a,b;i,j)$ (resp.\ $J_{m,n}(a,b)$) instead of $J^l(\gamma(a,b);i,j)$ (resp.\ $\O'(\gamma(a,b),f'_{m,n})$).
The integrals $J^l(a,b;0,0)$ were evaluated in \cite{FO}, Proposition~8.1, and we recall the result in Section~\ref{FO-paper1} below.

We note the following symmetry.
\begin{lemma}\label{lemma Jfe2} For all $a,b\in F^*$, $m,n\in\Z_{\geq0}$, $\overline{J_{n,m}(b,a)}=J_{m,n}(a,b)$,
\end{lemma}

\begin{proof} This follows directly from Corollary~\ref{cor Jfe}, upon observing that
$\gamma(b,a)=\gamma(a,b)^{-1}$.\end{proof} 

In view of Lemma~\ref{lemma Jfe2}, it is sufficient to carry out the computation of $J_{m,n}(a,b)$ in  the case that $\abs{b_{m+2n}}\le \abs{a_{n+2m}}$.
We suppose this from now on.

\subsection{Computation of the integrals $J^l(a,b;i,j)$}
In order to compute $J_{m,n}(a,b)$ for $m,n\in \Z_{\ge 0}$ and $a,b\in F^*$ such that $\abs{b_{m+2n}}\le \abs{a_{n+2m}}$, 
we will apply the expansion in Corollary~\ref{cor exp j} and explicit formulas for $J^l(a,b;i,j)$.
Due to the shifts in indices that appear in Corollary~\ref{cor exp j}, we will require formulas for
$J^l(a,b;i,j)$ when $(i,j)\in \Lambda'$ and $a,b\in F^*$ such that $\abs{b_j}\le q\abs{a_i}$.  We will establish the following result.

\begin{proposition}\label{prop jl}
Let $(i,j)\in\Lambda'$, $a,b\in F^*$ and suppose that $\abs{b_j}\le q\abs{a_i}$.  Then $(b,a)J^l(a,b;i,j)$ is equal to:
\[
\begin{cases}
(\varpi,ab)^{i+j} & \abs{b_j}=\abs{a_i}=1\\
\abs{b}^{-1}\Cu(-3b^{-1},a^{-1}b^{-1};j) & \abs{b_j}\le q^{-1},\ \abs{a_i}=1, \ i=2j \\
(\varpi,ab)^{i+j+1}\g & \abs{b_j}=q^{-1},\ \abs{a_i}=1,\ i<2j \\
(\varpi,ab)^{i+j+1}\g & \abs{b_j}=1,\ \abs{a_i}=q^{-1},\ j<2i\\
\abs{a}^{-1} \Cu(-3a^{-1},a^{-1}b^{-1};i) & \abs{b_j}=1,\ \abs{a_i}=q^{-1},\ j=2i \\
q& \abs{b_j}=\abs{a_i}=q^{-1},\ (j=2i>0 \text{ or }i=2j>0)\\
2q& \abs{b_j}=\abs{a_i}=q^{-1},\ i=j=0\\
q^2(\varpi,ab)^{i+j} & \abs{b_j}= q^{-2},\ \abs{a_i}=q^{-1},\ i\le 2(j-1)\\ 
\g q\abs{b}^{-1}\Cu(-3b^{-1},a^{-1}b^{-1};j-1)&  \abs{b_j}\le q^{-2},\ \abs{a_i}=q^{-1},\ i=2j-1\\ 
\g q\abs{a}^{-1}\Cu(-3a^{-1},a^{-1}b^{-1};i-1) &  \abs{b_j}= q^{-1},\ \abs{a_i}=q^{-2},\ j=2i-1\\
q^2(\varpi,ab)^{i+j} &   \abs{b_j}= q^{-1},\ \abs{a_i}=q^{-2},\ 2i-j\ge 2 \\
q^2\abs{a}^{-1}\Cu(-3a^{-1},a^{-1}b^{-1};i-1) &  \abs{b_j}=\abs{a_i}=q^{-2},\  j=2i>0\\
q^2\abs{b}^{-1}\Cu(-3b^{-1},a^{-1}b^{-1};j-1) &  \abs{b_j}\le \abs{a_i}=q^{-2},\ i=2j>0\\
(\varpi,ab)^{i+j+1}\g q^2 &   \abs{b_j}=\abs{a_i}=q^{-2},\ 2j-i, 2i-j\ge 1 \\
q^2\abs{a}^{-1}\Cu(-3a^{-1},a^{-1}b^{-1};i-1) &  \abs{b_j}=q^{-2},\ \abs{a_i}=q^{-3},\ j=2i>0\\
\abs{ab}^{-1}\sum_{\ell=0}^2 \Cu(-3(b^{-1}+\rho^\ell a^{-1}),a^{-1}b^{-1};\lfloor-\frac{\val(a)}2\rfloor) & \abs{b_j}\le \abs{a_i}\le q^{-2},\ i=j=0\\
\abs{ab}^{-1}\sum_{\ell=0}^2 \Cu(-3(a^{-1}+\rho^\ell b^{-1}),a^{-1}b^{-1};\lfloor-\frac{\val(b)}2\rfloor) & \abs{b_j}=q \abs{a_i}\le q^{-2},\ i=j=0\\
0 & \text{otherwise.}
\end{cases}
\]
\end{proposition}

We turn to the proof of this result.  As for the relative orbital integrals, we must write the integrals here in coordinates,
obtaining an integral over a subspace of $F^6$,
simplify, and then break the domain into subdomains over which the integration may be carried out.

\subsection{Coordinate form of the integral $J^l$ and the integral $J^r$}
For $a,b\in F^*$, let
\[
A[a,b]=\{(u_1,u_2)\in N\times N: u_1g_{a,b}u_2\in K_1\}
\]
(recall that $g_{a,b}$ is given in Section~\ref{ss orbreps}).

\begin{lemma}\label{lem int for J}
For $a,b\in F^*$ and $(i,j)\in\Lambda'$ we have
\[
J^l(a,b;i,j)=(\varpi,b)^j(\varpi,a)^{i-j}\int_{A[a_i,b_j]} \kappa(u_1g_{a_i,b_j}u_2)\theta(c_{i,j}^{-1}u_1c_{i,j})\theta(u_2)\ du_1\ du_2.
\]
\end{lemma}
\begin{proof}
After the variable change $u_1\mapsto \sec(c_{i,j})^{-1}u_1^{-1}\sec(c_{i,j})$ we have
\[
J^l(a,b;i,j)=\int_{N\times N} f_\circ (u_1\sec(c_{i,j})\sec(g_{a,b})u_2) \theta(c_{i,j}^{-1}u_1c_{i,j})\theta(u_2)\ du_1\ du_2.
\]
Note that $\sec(c_{i,j})\sec(g_{a,b})=(\varpi,b)^{-j}(\varpi,a)^{j-i}g_{a_i,b_j}$. Since
\[
f_\circ (u_1\sec(c_{i,j})\sec(g_{a,b})u_2)=(\varpi,b)^j(\varpi,a)^{i-j}\kappa(u_1g_{a_i,b_j}u_2)
\]
the lemma follows.
\end{proof}

For our purposes, we only need to compute the integral $J^l(a,b;i,j)$ when $\abs{b_j}\le q\abs{a_i}$.
If $\abs{b_j}\le \abs{a_i}$ we compute it directly. For the case $\abs{b_j}= q\abs{a_i}$ we use a modified integral
\[
J^r(a,b;i,j)=(a,\varpi)^i(b,\varpi)^{j-i}\int_{A[a_i,b_j]} \kappa(u_1g_{a_i,b_j}u_2)\theta(u_1)\theta(c_{j,i}^{-1}u_2c_{j,i})\ du_1\ du_2
\]
and the following functional equation.

\begin{corollary}\label{cor jij}
We have 
\[
J^r(b,a;j,i)=\overline{J^l(a,b;i,j)}.
\]
\end{corollary}
\begin{proof}
We follow \cite[Lemma 8.2]{FO} and its proof. Since $g_{b,a}=g_{a,b}^{-1}$ for all $a,b\in F^*$ we have
\[
J^r(b,a;j,i)=(b,\varpi)^j(a,\varpi)^{i-j}\int_{A[b_j,a_i]} \kappa(u_1g_{a_i,b_j}^{-1}u_2)\theta(u_1)\theta(c_{i,j}^{-1}u_2c_{i,j})\ du_1\ du_2.
\]
Lemma 8.2 in \cite{FO} and its proof give that the variable change $(u_1,u_2)\mapsto (u_2^{-1},u_1^{-1})$ maps $A[b_j,a_i]$ to $A[a_i,b_j]$ and that
\[
\kappa(u_2^{-1}g_{a_i,b_j}^{-1}u_1^{-1})=\overline{\kappa(u_1 g_{a_i,b_j}u_2)}.
\]
Since also 
\[
\theta(u_2^{-1})\theta(c_{i,j}^{-1}u_1^{-1}c_{i,j})=\overline{\theta(u_2)\theta(c_{i,j}^{-1}u_1c_{i,j})},
\]
applying this change of variables and Lemma \ref{lem int for J}, the corollary follows. 
\end{proof}

For $a,b\in F^*$ and $(i,j)\in\Lambda'$ set
\[
\J^l(a,b;i,j)=\int_{A[a,b]} \kappa(u_1g_{a,b}u_2)\theta(c_{i,j}^{-1}u_1c_{i,j})\theta(u_2) \ du_1\ du_2
\]
and
\[
\J^r(a,b;i,j)=\int_{A[a,b]}
 \kappa(u_1g_{a,b}u_2)\theta(u_1)\theta(c_{j,i}^{-1}u_2c_{j,i}) \ du_1\ du_2
\]
so that
\begin{equation}\label{eq normJ}
J^l(a,b;i,j)=(\varpi,b)^j(\varpi,a)^{i-j}\J^l(a_i,b_j;i,j)\end{equation}
and
\begin{equation}\label{eq normJr}
J^r(a,b;i,j)=(a,\varpi)^i(b,\varpi)^{j-i}\J^r(a_i,b_j;i,j).
\end{equation}
We obtain Proposition \ref{prop jl} by computing the integrals $\J^l(a,b;i,j)$ whenever $\abs{b}\le \abs{a}$ and $\J^r(a,b;i,j)$ whenever $q\abs{b}=\abs{a}$ together with an application of Corollary \ref{cor jij}. 

Let
\[
u_i=\begin{pmatrix} 1 & x_i & z_i \\ & 1 & y_i \\ & & 1\end{pmatrix},\ \ \ i=1,2.
\]
Then
\[
\theta(c_{i,j}^{-1}u_1c_{i,j})=\psi(\varpi^{2j-i}x_1+\varpi^{2i-j}y_1)\ \ \ \text{and}\ \ \ \theta(c_{j,i}^{-1}u_2c_{j,i})=\psi(\varpi^{2i-j}x_2+\varpi^{2j-i}y_2).
\]
For the explication of the domain $A[a,b]$ we also record that
\[
u_1g_{a,b}u_2=\begin{pmatrix} az_1 & az_1x_2-a^{-1}bx_1 & az_1z_2 -a^{-1}bx_1y_2+b^{-1}  \\ ay_1 & ay_1x_2 -a^{-1}b & ay_1z_2-a^{-1}by_2 \\ a & ax_2 & az_2\end{pmatrix}.
\]
In order to unify the notation for $\J^l$ and $\J^r$, for $d\in \{l,r\}$ write
\begin{equation}\label{eq s12}
s_1=\begin{cases}
2j-i & d=l\\
0 & d=r
\end{cases},\ \ \ 
s_2=\begin{cases}
0& d=l\\
2i-j & d=r
\end{cases}
\end{equation}
and
\begin{equation}\label{eq t12}
t_1=\begin{cases}
2i-j & d=l\\
0 & d=r
\end{cases},\ \ \ 
t_2=\begin{cases}
0 & d=l\\
2j-i & d=r.
\end{cases}
\end{equation}
Since $(i,j)\in \Lambda'$, we have $s_1,s_2,t_1,t_2\in \Z_{\ge 0}$, and
\[
\J^d(a,b;i,j)=\int_{A[a,b]} \kappa(u_1g_{a,b}u_2)\psi(\varpi^{s_1}x_1+\varpi^{t_1}y_1+\varpi^{s_2}x_2+\varpi^{t_2}y_2) \,du_1\,du_2.
\]
We will compute these integrals by relating them to the Kloosterman integrals $\Kl$ defined in Section~\ref{Kloosterman sums with cubic characters}.

\subsection{The computation of $\J^l$ and $\J^r$ for $\abs{a}\ge 1$}

\begin{lemma}\label{lem ja=1}
For $(i,j)\in \Lambda'$ we have
\[
\J^l(a,b;i,j)=\begin{cases}
0 & \abs{a}>1 \\ 
1 & \abs{b}=\abs{a}=1 \\
\abs{b}^{-1}\Kl(a^{-1}b;b^{-1},\varpi^{2j-i}ab^{-1})  & \abs{b}\le q^{-1},\ \abs{a}=1
\end{cases}
\]
and 
\[
\J^r(a,b;i,j)=\begin{cases}
0 & \abs{a}>1 \\ 
\abs{b}^{-1}\Kl(a^{-1}b;\varpi^{2j-i}b^{-1},ab^{-1}) & \abs{b}\le q^{-1},\ \abs{a}=1.
\end{cases}
\]
\end{lemma}
\begin{proof}
If $\abs{a}>1$ then $A[a,b]$ is empty and the vanishing in both formulas is immediate. If $\abs{b}=\abs{a}=1$ then the domain $A[a,b]$ is $(N\cap K')^2$, $\kappa(u_1g_{a,b}u_2)=1$ by \cite[Lemma 8.7(1)]{FO} and clearly also $\theta(c_{i,j}^{-1}u_1c_{i,j})\theta(u_2)=1$ identically on this domain. The formula for $\J^l(a,b;i,j)$ in this case follows. Assume now that $\abs{a}=1$ and $\abs{b}\le q^{-1}$.  
Applying \cite[Lemma 8.7(2)]{FO} to evaluate $\kappa(u_1g_{a,b}u_2)$ and our notation \eqref{eq s12} and \eqref{eq t12}, for $d\in \{l,r\}$ we have
\[
\J^d(a,b;i,j)=(b,a)\int (a^{-1}b,y_2)\psi(\varpi^{s_1}x_1+\varpi^{t_1}y_1+\varpi^{s_2}x_2+\varpi^{t_2}y_2)\ d(x_1,y_1,z_1,x_2,y_2,z_2)
\]
where the integral is over the domain given by
\[
y_1,x_2,z_1,z_2, bx_1,by_2,b^{-1}(1-a^{-1}b^2x_1y_2)\in \O.
\]
After integrating over $y_1,x_2,z_1,z_2$ we have
\[
\J^d(a,b;i,j)=(b,a)\int (a^{-1}b,y_2)\psi(\varpi^{s_1}x_1+\varpi^{t_2}y_2)\ d(x_1,y_2) , \ \ \ d\in \{l,r\} 
\]
where integration is over the domain given by
\[
bx_1,by_2,b^{-1}(1-a^{-1}b^2x_1y_2)\in \O
\]
or equivalently
\[
by_2\in \O^*, x_1\in ab^{-2}y_2^{-1}+\O.
\]
Integrating over $x_1$ we have
\[
\J^d(a,b;i,j)=(b,a)\int_{b^{-1}\O^*} (a^{-1}b,y_2)\psi(\varpi^{s_1}ab^{-2}y_2^{-1}+\varpi^{t_2}y_2)\ dy_2  , \ \ \ d\in \{l,r\} .
\]
The variable change $y_2\mapsto b^{-1}y_2$ completes the proof of the lemma.
\end{proof}

\subsection{Decomposition of the domain of integration for $\abs{a}<1$}
Let $a,b\in F^*$ be such that $\abs{b}\le \abs{a}\le q^{-1}$.
According to  \cite[Lemma 8.7]{FO} the domain $A[a,b]$ trisects into disjoint subdomains $A_t[a,b]$, $t=1,2,3$ defined by the intersection of $A[a,b]$ with the $(t)$-th condition:
\begin{enumerate}
\item $\abs{ay_1}=1$
\item $\abs{ax_2}=1$ 
\item $\abs{ay_1},\,\abs{ax_2}<1$.
\end{enumerate}
Accordingly, we write
\[
\J_t^l(a,b;i,j)=\int_{A_t[a,b]} \kappa(u_1g_{a,b}u_2)\theta(c_{i,j}^{-1}u_1c_{i,j})\theta(u_2)\ du_1\ du_2 
\]
and
\[
\J_t^r(a,b;i,j)=\int_{A_t[a,b]} \kappa(u_1g_{a,b}u_2)\theta(u_1)\theta(c_{j,i}^{-1}u_2c_{j,i})\ du_1\ du_2
\]
so that
\[
\J^d(a,b;i,j)=\sum_{t=1}^3 \J_t^d(a,b;i,j),\ \ \ d\in\{l,r\}.
\]
\subsection{The computation of the contributions $\J_1^d$ and $\J_2^d$}

\begin{lemma}\label{lem j1}
Let $a,b\in F^*$ be such that $\abs{b}\le \abs{a}\le q^{-1}$. Then
\[
\J_1^l(a,b;i,j)=\abs{ab}^{-1}\Kl(a^{-1};\varpi^{2i-j}a^{-1},0)\,\Kl(b;b^{-1},\varpi^{2j-i}ab^{-1})
\]
and
\[
\J_1^r(a,b;i,j)=\abs{ab}^{-1}\Kl(a^{-1};a^{-1},0)\,\Kl(b;\varpi^{2j-i}b^{-1},ab^{-1}).
\]
\end{lemma}
\begin{proof}
Applying \cite[Lemma 8.7(3)(a)]{FO} to evaluate $\kappa(u_1g_{a,b}u_2)$  we have
\begin{multline*}
\J_1^d(a,b;i,j)=\\
(b,a)\int (y_1y_2,ab^{-1})(y_2,y_1)\psi(\varpi^{s_1}x_1+\varpi^{t_1}y_1+\varpi^{s_2}x_2+\varpi^{t_2}y_2)\ d(x_1,y_1,z_1,x_2,y_2,z_2) 
\end{multline*}
for $d\in \{l,r\}$ (see  \eqref{eq s12} and \eqref{eq t12}) where the integral is over
\[
ay_1\in \O^*,\ az_1,\,az_2,\, x_2,\,ay_1z_2-a^{-1}by_2 ,\,a^{-1}bx_1,\,  az_1z_2-a^{-1}bx_1y_2+b^{-1}\in \O,
\]
or equivalently
\[
ay_1,\,by_2\in \O^*, \,az_1,\, x_2\in\O,\,z_2\in a^{-2}by_1^{-1}y_2+\O, x_1\in ab^{-2}y_2^{-1}+a^2b^{-1}y_2^{-1}z_1z_2+a\O.
\]
Note that in the domain of integration we have 
\[
\psi(\varpi^{s_1}x_1)=\psi(\varpi^{s_1}ab^{-2}y_2^{-1})\ \ \ \text{and}\ \ \ \psi(\varpi^{s_2}x_2)=1.
\]
After integrating over $x_1$ and $x_2$ we also integrate over $z_1$ and $z_2$ to obtain
\[
\J_1^d(a,b;i,j)=(b,a)\int_{b^{-1}\O^*} \int_{a^{-1}\O^*}(y_1y_2,ab^{-1})(y_2,y_1)\psi(\varpi^{s_1}ab^{-2}y_2^{-1}+\varpi^{t_1}y_1+\varpi^{t_2}y_2)\ dy_1\ dy_2,
\]
for $d\in\{l,r\}$.
The variable change $y_1\mapsto a^{-1}y_1$, $y_2\mapsto b^{-1}y_2$ together with the basic properties of the Hilbert symbol in \S\ref{sss Hilb} complete the proof of the lemma.
\end{proof}

\begin{lemma}\label{lem j2}
Let $a,b\in F^*$ be such that $\abs{b}\le \abs{a}\le q^{-1}$. Then
\[
\J_2^l(a,b;i,j)=\abs{ab}^{-1}(b,a)\Kl(b^{-1};\varpi^{2j-i}b^{-1},ab^{-1})\Kl(a;a^{-1},0)
\]
and
\[
\J_2^r(a,b;i,j)=\abs{ab}^{-1}(b,a)\Kl(b^{-1};b^{-1},\varpi^{2j-i}ab^{-1})\Kl(a;\varpi^{2i-j}a^{-1},0).
\]
\end{lemma}
\begin{proof}
Applying \cite[Lemma 8.7(3)(b)]{FO} to evaluate $\kappa(u_1g_{a,b}u_2)$ we have
\begin{multline*}
\J_2^d(a,b;i,j)=\\
(b,a)\int (b,x_2)(z_2x_2^{-1}-y_2,b^{-1}ax_2)\psi(\varpi^{s_1}x_1+\varpi^{t_1}y_1+\varpi^{s_2}x_2+\varpi^{t_2}y_2)\ d(x_1,y_1,z_1,x_2,y_2,z_2) 
\end{multline*}
for $d\in \{l,r\}$ (see  \eqref{eq s12} and \eqref{eq t12}) where the integral is over
\[
ax_2\in \O^*,\ az_1,\,az_2,\, y_1,\, a^{-1}by_2 ,\,az_1x_2-a^{-1}bx_1,\,  az_1z_2-a^{-1}bx_1y_2+b^{-1}\in \O
\]
or equivalently
\[
ax_2,\,bx_1\in \O^*, \,az_2,\, y_1\in\O,\,z_1\in a^{-2}bx_1x_2^{-1}+\O, y_2\in ab^{-2}x_1^{-1}+a^2b^{-1}x_1^{-1}z_1z_2+a\O.
\]
In the domain of integration $y_2\in ab^{-2}x_1^{-1}+\O$ and 
\[
y_2-x_2^{-1}z_2\in b^{-2}ax_1^{-1}+(b x_1)^{-1}a^2(z_1- a^{-2}bx_1x_2^{-1})z_2+a\O=b^{-2}ax_1^{-1}+a\O
\]
so that
\[
(b,x_2)(z_2x_2^{-1}-y_2,b^{-1}ax_2)=(b,x_2)(b^{-2}ax_1^{-1},b^{-1}ax_2)=(a,b)(x_2,x_1)(b^{-1}a,x_1x_2).
\] 
After integrating over $y_1$ and $y_2$ we also integrate over $z_1$ and $z_2$ to obtain
\begin{multline*}
\J_2^d(a,b;i,j)=\\
\int_{b^{-1}\O^*} \int_{a^{-1}\O^*}(x_2,x_1)(b^{-1}a,x_1x_2)\psi(\varpi^{s_1}x_1+\varpi^{s_2}x_2+\varpi^{t_2}ab^{-2}x_1^{-1})\ dx_2\ dx_1,\ \ \ d\in\{l,r\}.
\end{multline*}
The variable change $x_2\mapsto a^{-1}x_2$, $x_1\mapsto b^{-1}x_1$ together with the basic properties of the Hilbert symbol in \S\ref{sss Hilb} complete the proof of the lemma.
\end{proof}
We now add these expressions and
apply  the basic properties in \S\ref{ss prem}.  For example, since $|\varpi^{2j-i}a|<1$, the Kloosterman sum $\Kl(b;b^{-1},\varpi^{2j-i}ab^{-1})$ that appears in the evaluation of $\J_1^l$ is zero if $|b|<q^{-1}$,
by \eqref{eq klvan}.
This reduces to the case $|a|=|b|=q^{-1}$.  Then the sum $\Kl(a^{-1};\varpi^{2i-j}a^{-1},0)$ is zero unless $j=2i$, by \eqref{eq klvol}.  Using \eqref{eq kl as gauss}, \eqref{the-gauss-sum},
\eqref{eq abs gauss} and simplifying, we see that $\J_1^l(a,b;i,j)=(a,b)q\delta_{j,2i}$.
Analyzing $\J_2^l$ similarly we find that for $\abs{b}\le \abs{a}\le q^{-1}$, 
\[
\J_2^l(a,b;i,j)=
\begin{cases}
(a,b)q\delta_{i,2j} &|a|=|b|=q^{-1}\\
\g\abs{b}^{-1}(b\varpi,a) \,\Kl(b^{-1};\varpi b^{-1},ab^{-1})& 2j-i=1\text{ and }\abs{b}<\abs{a}=q^{-1} \\
q^2(a,b) & 2j-i\ge 2,\ \abs{b}=q^{-2},\abs{a}=q^{-1}\\
0& \text{otherwise.}
\end{cases}
\]
When $2j-i=1$ and $\abs{b}<\abs{a}=q^{-1}$, by \eqref{eq di}, we further have $\Kl(b^{-1};\varpi b^{-1},ab^{-1})=(b^{-1},\varpi a^{-1})\Cu(-3ab^{-1},\varpi^{-1}a^2b^{-1};0)$.

Analyzing $\J_1^r$ and $\J_2^r$ in a similar way, we obtain the following result.
\begin{proposition}\label{j1 plus j2} For $\abs{b}\le \abs{a}\le q^{-1}$, we have
\begin{multline*}
(b,a)[\J_1^l(a,b;i,j)+\J_2^l(a,b;i,j)]=\\ 
\begin{cases}
q[\delta_{j,2i}+\delta_{i,2j}] & \abs{b}=\abs{a}=q^{-1} \\
\g \abs{b}^{-1}(\varpi,ab)\Cu(-3ab^{-1},\varpi^{-1}a^2b^{-1};0) & 2j-i=1\text{ and }\abs{b}<\abs{a}=q^{-1} \\
q^2 & 2j-i\ge 2,\ \abs{b}=q^{-2},\abs{a}=q^{-1}\\
0& \text{otherwise.}
\end{cases}
\end{multline*}
and 
\begin{multline*}
(b,a)[\J_1^r(a,b;i,j)+\J_2^r(a,b;i,j)]=\\
\begin{cases}
q[\delta_{j,2i}+\delta_{i,2j}] & \abs{b}=\abs{a}=q^{-1} \\
\bar{\g} \abs{b}^{-1}(ab,\varpi)\Cu(-3ab^{-1},\varpi^{-1}a^2b^{-1};0) & 2j-i=1\text{ and }\abs{b}<\abs{a}=q^{-1} \\
q^2 & 2j-i\ge 2,\ \abs{b}=q^{-2},\abs{a}=q^{-1}\\
0& \text{otherwise.}
\end{cases}
\end{multline*}
\end{proposition}

\subsection{Computation of the contributions $\J_3^d$} In order to evaluate $\J_3^d$, $d\in\{l,r\}$, we shall first carry out several integrations.  We require the following information from \cite{FO}.
\subsubsection{Information about subdomains}
For $F^5$, we use coordinates $(x_1,y_1,x_2,y_2,z_2)$, and  for any set $D\subseteq F^5$, let $D'$ be its subset with  $y_1\ne 0$.
For $a,b\in F^*$ such that  $\abs{b}\le \abs{a}<1$ let $D[a,b]$ be the domain in $ F^5$ defined by 
\[
ay_1,ax_2\in\p, \ az_2\in \O^*,\ \abs{bx_1y_2-1}=\abs{ab^{-1}},\ y_1-a^{-1}by_2,x_2-a^{-1}bx_1\in \O.
\]
Denote by $D_0[a,b]$ the subset of elements in $D[a,b]$ such that either $y_1\in\O$ or $x_2\in\O$ and let $D_{>0}[a,b]=D[a,b]\setminus D_0[a,b]$ be its complement. 
According to \cite[Lemma 8.15]{FO} we have 
\begin{itemize}
\item $by_2,bx_1\in \p$ whenever $(x_1,y_1,x_2,y_2,z_2)\in D[a,b]$, 
\item the set $D_0[a,b]$ is empty unless $\abs{a}=\abs{b}$ 
\item and 
\begin{multline}\label{eq lem8.15}
(y_1,ab^{-1})(y_1(ax_2y_2-1), ay_1x_2-a^{-1}b)(b,(ax_2y_2-1)z_2)=\\\begin{cases}
(b,z_2) & (x_1,y_1,x_2,y_2,z_2)\in D_0'[a,b] \\
(b,z_2)(bx_1y_2-1,x_1) & (x_1,y_1,x_2,y_2,z_2)\in D_{>0}'[a,b] .
\end{cases}
\end{multline}
\end{itemize}
It easily follows that if $(x_1,y_1,x_2,y_2,z_2)\in D_0'[a,b]$ then $bx_1y_2\in\p$ and therefore
\begin{equation}\label{eq lem8.15u}
(y_1,ab^{-1})(y_1(ax_2y_2-1), ay_1x_2-a^{-1}b)(b,(ax_2y_2-1)z_2)=
(b,z_2)(bx_1y_2-1,x_1).
\end{equation}
This provides a uniform formula for all $(x_1,y_1,x_2,y_2,z_2)\in D'[a,b] $. It is also straightforward that $D[a,b]$ is characterized by the conditions
\begin{equation}\label{eq dab}
by_2,bx_1\in\p, \ az_2\in \O^*,\ \abs{bx_1y_2-1}=\abs{ab^{-1}},\ y_1-a^{-1}by_2,x_2-a^{-1}bx_1\in \O.
\end{equation}
The properties of the subdomains of $F^5$ recalled above are applied in the next lemma. 
\subsubsection{Reduction to two variable integrals}
\begin{lemma}\label{lem z1int}
For $a,b\in F^*$ such that $\abs{b}\le \abs{a}<1$ we have
\begin{multline*}
\J_3^l(a,b;i,j)=\abs{a}^{-2}\int \abs{x_1}^{-1} \Kl(b; b^{-1}y_2x_1^{-1},\varpi^{2j-i}ab^{-1}(y_2-1)^{-1}x_1)(y_2-1,x_1)\\
\psi(\varpi^{2i-j}a^{-1}y_2x_1^{-1}+a^{-1}bx_1)\ d(x_1,y_2) 
\end{multline*}
and 
\begin{multline*}
\J_3^r(a,b;i,j)=\abs{a}^{-2}
\int \abs{x_1}^{-1}\Kl(b;\varpi^{2j-i}b^{-1}y_2x_1^{-1},ab^{-1}(y_2-1)^{-1}x_1)(y_2-1,x_1)\\ 
\psi(a^{-1}y_2x_1^{-1}+\varpi^{2i-j}a^{-1}bx_1)\ d(x_1,y_2)
\end{multline*}
where the integrals are over the domain given by
\begin{equation}\label{eq Sdomain}
\abs{y_2}<\abs{x_1}<\abs{b}^{-1},\ \abs{y_2-1}=\abs{ab^{-1}}.
\end{equation}
\end{lemma}

\begin{proof}
The domain $A_3[a,b]$ is given by the conditions
\[
ay_1,ax_2\in\p, \ az_1,az_2, ay_1x_2, ay_1z_2-a^{-1}by_2, az_1x_2-a^{-1}b x_1, a z_1z_2-a^{-1}bx_1y_2 + b^{-1}\in \O.
\]
Let $(u_1,u_2)\in A_3[a,b]$ be such that $y_1\ne 0$ and set $g=u_1g_{a,b}u_2$.
It follows from \cite[Lemma 8.7 (3)(c)]{FO} that
\[
 \kappa(g)=(b,a)(y_1,ab^{-1})(ay_1(x_2y_2-z_2), ay_1x_2-a^{-1}b)(b,x_2y_2-z_2)
\]
and that $ay_1x_2-a^{-1}b\in \O^*$. 
Furthermore, $az_1,az_2\in \O^*$ (see \cite[Lemma 8.13]{FO}). 
Consequently, 
\[
 \kappa(g)=(b,a)(y_1,ab^{-1})(z_2^{-1}y_1(x_2y_2-z_2), ay_1x_2-a^{-1}b)(b,x_2y_2-z_2)
\]
and after the variable change 
\[
y_2\mapsto az_2 y_2,\ x_1\mapsto az_1x_1
\]
we have
\begin{multline*}
\J_3^d(a,b;i,j)=(b,a)\int (y_1,ab^{-1})(y_1(ax_2y_2-1), ay_1x_2-a^{-1}b)(b,(ax_2y_2-1)z_2)\\
\psi(\varpi^{s_1}az_1x_1+\varpi^{t_1}y_1+\varpi^{s_2}x_2+\varpi^{t_2}az_2y_2)\ d(x_1,y_1,z_1,x_2,y_2,z_2) ,\ \ \ d\in\{l,r\}
\end{multline*}
(see  \eqref{eq s12} and \eqref{eq t12}) where
the integral is over the domain given by $y_1\ne 0$ and
\[
ay_1,ax_2\in\p, \ az_1, az_2\in \O^*, \  ay_1x_2, y_1-a^{-1}by_2,x_2-a^{-1}bx_1, a z_1z_2(1-bx_1y_2)+b^{-1}\in \O
\]
or equivalently
\[
(x_1,y_1,x_2,y_2,z_2)\in D'[a,b], \  ay_1x_2\in \O, \  z_1\in (az_2)^{-1}(bx_1y_2-1)^{-1}b^{-1}+a^{-1}b \O.
\]
Note that in this domain $bx_1\in ax_2+a\O\subseteq \p\subseteq \O$ and therefore
\[
\psi(\varpi^{s_1} a z_1x_1)=\psi(\varpi^{s_1}(b z_2)^{-1}(bx_1y_2-1)^{-1}x_1)
\]
so that the integrand is independent of $z_1$. In the proof of \cite[Lemma 8.14]{FO} it is shown that every element of $D[a,b]$ satisfies $ay_1x_2\in \O$. Applying \eqref{eq lem8.15u} and integrating over $z_1$ we obtain that
\begin{multline*}
\J_3^d(a,b;i,j)=(b,a)\abs{a^{-1}b}\int_{D'[a,b]} (b,z_2)(bx_1y_2-1,x_1)\\
\psi(\varpi^{s_1}(b z_2)^{-1}(bx_1y_2-1)^{-1}x_1+\varpi^{t_1}y_1+\varpi^{s_2}x_2+\varpi^{t_2}az_2y_2)\ d(x_1,y_1,x_2,y_2,z_2), \ \ \ d\in\{l,r\}.
\end{multline*}
Note that in $D[a,b]$ we have
\[
\psi(\varpi^{t_1} y_1)=\psi(\varpi^{t_1} a^{-1}by_2) \ \ \ \text{and}\ \ \ \psi(\varpi^{s_2} x_2)=\psi(\varpi^{s_2} a^{-1}bx_1)
\]
so that the integrand is independent of $y_1\in a^{-1}by_2+\O$ and $x_2\in a^{-1}bx_1+\O$. Applying the characterization \eqref{eq dab} of $D[a,b]$ and integrating over $y_1$ and $x_2$ we obtain that
\begin{multline*}
\J_3^d(a,b;i,j)=(b,a)\abs{a^{-1}b}\int (b,z_2)(bx_1y_2-1,x_1)\\
\psi(\varpi^{s_1}(b z_2)^{-1}(bx_1y_2-1)^{-1}x_1+\varpi^{t_1}a^{-1}by_2+\varpi^{s_2}a^{-1}bx_1+\varpi^{t_2}az_2y_2)\ d(x_1,y_2,z_2) ,\ \ \ d\in\{l,r\}
\end{multline*}
where the integral is over the domain defined by
\[
by_2,bx_1\in\p, \ az_2\in \O^*,\ \abs{bx_1y_2-1}=\abs{ab^{-1}}.
\]
Changing the variable $z_2\mapsto a^{-1}z_2$ and integrating over $z_2$ we have
\begin{multline*}
\J_3^d(a,b;i,j)=\abs{a^{-2}b}\int (bx_1y_2-1,x_1) \Kl(b;  \varpi^{t_2}y_2,\varpi^{s_1}ab^{-1}(bx_1y_2-1)^{-1}x_1)\\
\psi(\varpi^{t_1}a^{-1}by_2+\varpi^{s_2}a^{-1}bx_1)\ d(x_1,y_2) ,\ \ \ d\in\{l,r\}
\end{multline*}
where the integral is over the domain defined by
\[
by_2,bx_1\in\p, \ \abs{bx_1y_2-1}=\abs{ab^{-1}}.
\]
After the variable change $y_2\mapsto (bx_1)^{-1}y_2$
the lemma follows.
\end{proof}

\subsubsection{Decomposition of the two-variable domain}
Let $S$ be the domain of $(x_1,y_2)\in F^2$ satisfying \eqref{eq Sdomain}. Consider the subsets
\[
S_0=\{(x_1,y_2)\in S: \text{either } \abs{x_1}\le 1 \text{ or } \abs{y_2}\le \abs{bx_1}\} \ \ \ \text{and}\ \ \ S'=S\setminus S_0.
\]
Since $|ab^{-1}|\geq1$, it is straightforward that $S_0$ is empty unless $\abs{a}=\abs{b}$ in which case $S_0=S_1\cup S_2$ 
(not a disjoint union) where 
\[
S_1=\{(x_1,y_2)\in S: \abs{x_1}\le 1\} \ \ \ \text{and}\ \ \ S_2=\{(x_1,y_2)\in S: \abs{y_2}\le |bx_1|\}. 
\]
Furthermore, $S'$ is the disjoint union $S'=\sqcup_{k=1}^{\val(a)-1} S_k'$ where 
\begin{multline*}
S_k'=\{(x_1,y_2)\in S: \abs{x_1}=\abs{b}^{-1}q^{-k}, \ q^{-k}<\abs{y_2},\ \abs{y_2-1}=\abs{ab^{-1}}\}=\\
\begin{cases}
\{(x_1,y_2)\in F^2: \abs{x_1}=\abs{b}^{-1}q^{-k},\ \abs{y_2}=\abs{ab^{-1}}\}& \abs{b}<\abs{a}\\
\{(x_1,y_2)\in F^2: \abs{x_1}=\abs{b}^{-1}q^{-k}, \ q^{-k}<\abs{y_2},\ \abs{y_2-1}=1\} & \abs{b}=\abs{a}.
\end{cases} 
\end{multline*}
(If $\val(a)=1$ the set $S'$ is empty.) For $T\subseteq F^2$ let 
\begin{multline*}
\J_{3,T}^d(a,b;i,j)=\abs{a}^{-2}\int_{T} \abs{x_1}^{-1} \Kl(b;\varpi^{t_2}b^{-1}y_2x_1^{-1}, \varpi^{s_1}ab^{-1}(y_2-1)^{-1}x_1)(y_2-1,x_1)\\ 
\psi(\varpi^{t_1}a^{-1}y_2x_1^{-1}+\varpi^{s_2}a^{-1}bx_1)\ d(x_1,y_2).
\end{multline*}
Then it follows from Lemma \ref{lem z1int} that if $\abs{b}<\abs{a}$ then
\begin{equation}\label{sum-3-va neq vc}
\J_3^d(a,b;i,j)=\sum_{k=1}^{\val(a)-1} \J_{3,S_k'}^d(a,b;i,j),\ \ \ d\in \{l,r\},
\end{equation}
while if $\abs{b}=\abs{a}$ then 
\[
\J_3^l(a,b;i,j)=\J_{3,S_0}^l(a,b;i,j)
+\sum_{k=1}^{\val(a)-1} \J_{3,S_k'}^l(a,b;i,j)
\]
(here we treat only $\J_3^l$ as the evaluation of $\J_3^r$ when $|b|=|a|$ is not needed in the evaluation of the big cell
orbital integrals). 
Moreover, by standard inclusion-exclusion, 
\begin{equation}\label{inclusion-exclusion}
\J_{3,S_0}^l(a,b;i,j)=\J_{3,S_1}^l(a,b;i,j)+\J_{3,S_2}^l(a,b;i,j)-\J_{3,S_1\cap S_2}^l(a,b;i,j).
\end{equation}
We now compute each summand separately.

\subsubsection{Evaluation of the term $\J_{3,S_0}^l(a,b;i,j)$}  We prove the following evaluation.

\begin{lemma}\label{lem j0}
Let $\abs{b}= \abs{a}<1$. Then
\[
 \abs{a}(b,a)\J_{3,S_0}^l(a,b;i,j)=\begin{cases} 2 & 3\nmid\val(a)>1 \text{ and }i=j=0 \\ 1+q^{-1} & 3|\val(a) \text{ and }i=j=0 \\ 1 & 3\nmid\val(a)>1\text{ and either }2i=j>0\text{ or }2j=i>0\\ q^{-1}& 3|\val(a)\text{ and either }2i=j>0\text{ or }2j=i>0   \\ q^{-1}-1 & 3|\val(a) \text{ and }2i-j,2j-i>0\\ 0 & \text{either }\val(a)=1\text{ or }(3\nmid\val(a)>1\text{ and }2i-j,2j-i>0). \end{cases}
\]
\end{lemma}
\begin{proof}
In the domain $S_1$, we have that $(y_2-1,x_1)=1$ and
$$\Kl(b;\varpi^{t_2}b^{-1}y_2x_1^{-1}, \varpi^{s_1}ab^{-1}(y_2-1)^{-1}x_1)=\Kl(b;\varpi^{t_2}b^{-1}y_2x_1^{-1},0).$$
We find that
\begin{equation*}
\J_{3,S_1}^l(a,b;i,j)=\abs{a}^{-2}\int_{|x_1|\leq1,\ |y_2|<|x_1|} \abs{x_1}^{-1} \Kl(b;b^{-1}y_2x_1^{-1},0)
\psi(\varpi^{2i-j}a^{-1}y_2x_1^{-1})\ d(x_1,y_2).
\end{equation*}
Changing $y_2\mapsto bx_1y_2$ and integrating over $x_1$, we obtain
$$\abs{a}^{-1}\int_{|y_2|<|b|^{-1}}\Kl(b;y_2,0)
\psi(\varpi^{2i-j}a^{-1}by_2)\ dy_2.$$
Writing out the Kloosterman integral gives
$$|a|^{-1}\int_{|u|=1,\,|y_2|<|b|^{-1}} (b,u)\psi(y_2(u+\varpi^{2i-j}a^{-1}b))\,dy_2\,du.$$
The $y_2$ integral gives zero unless $2i=j$ or $\val(a)=1$, and if $\val(a)=1$ then the $u$ integral is zero. If $2i=j$, $\val(a)>1$, the $y_2$ integral
is zero except on the region $|u+a^{-1}b|\leq q|b|$, and on this region $(b,u)=(a,b)$ and the value of $\psi$ is 1.
We find that
\[
\J_{3,S_1}^l(a,b;i,j)=\begin{cases} \abs{a}^{-1}(a,b) & \val(a)>1 \text{ and } 2i=j \\ 0 & \val(a)=1 \text{ or } 2i>j. \end{cases} 
\]

In the domain $S_2$, we have 
$$\Kl(b;\varpi^{t_2}b^{-1}y_2x_1^{-1}, \varpi^{s_1}ab^{-1}(y_2-1)^{-1}x_1)=\Kl(b;0,\varpi^{s_1}ab^{-1}(y_2-1)^{-1}x_1).$$
Also, since $|y_2|\leq |bx_1|<1$, we have $(y_2-1,x_1)=1$, and we may remove the factor $(y_2-1)^{-1}$ in the Kloosterman integral above using \eqref{eq klinv}. We find that
\begin{equation*}
\J_{3,S_2}^l(a,b;i,j)=\abs{a}^{-2}\int_{|y_2|\leq |bx_1|,\, |x_1|<|b^{-1}|} \abs{x_1}^{-1} \Kl(b;0,\varpi^{2j-i}ab^{-1}x_1)
\psi(a^{-1}bx_1)\ d(x_1,y_2).
\end{equation*}
Integrating over $y_2$, we obtain
$$\abs{a}^{-1}\int_{|x_1|<|b|^{-1}}\Kl(b;0,\varpi^{2j-i}ab^{-1}x_1)
\psi(a^{-1}bx_1)\ dx_1.$$
Proceeding as in the first case, we find that
\[
\J_{3,S_2}^l(a,b;i,j)=\begin{cases} \abs{a}^{-1}(a,b) & \val(a)>1 \text{ and } 2j=i\\ 0 & \val(a)=1 \text{ or } 2j>i. \end{cases} 
\]

Finally, on $S_1\cap S_2$, we have
$$\Kl(b;\varpi^{t_2}b^{-1}y_2x_1^{-1}, \varpi^{s_1}ab^{-1}(y_2-1)^{-1}x_1)=\Kl(b;0,0).$$
We get
\begin{align*}
\J_{3,S_1\cap S_2}^l(a,b;i,j)&=\abs{a}^{-2}\int_{|b^{-1}y_2|\leq |x_1|\leq1} \abs{x_1}^{-1} \Kl(b;0,0)\ d(x_1,y_2)\\
&=
\abs{a}^{-1}(1-q^{-1})\delta_{3|\val(a)},
\end{align*}
where the second equality follows from \eqref{eq klvol} and a straightforward volume computation.  Combining the evaluations (see \eqref{inclusion-exclusion}), 
the lemma follows.
\end{proof}

\subsubsection{Evaluation of the term $\J_{3,S_k'}^d(a,b;i,j)$: Preparation}
We suppose $\val(a)\geq2$ as otherwise the domain of integration is empty.
Each domain $S_k'$ is a union of annuli of fixed $|x_1|$.  Accordingly, 
to evaluate these terms, we will integrate over the annulus of fixed $|x_1|$ and then carry out the remaining $y_2$ integration as appropriate for the domains under consideration.
For $k\in\Z$, $1\leq k\leq\val(a)-1$ and $y_2\in F$ such that $|y_2-1|=|ab^{-1}|$, let
\begin{multline*}
\j_k^d(y_2)=\int_{b^{-1}\varpi^{k}\O^*} \abs{x_1}^{-1}\Kl(b;\varpi^{t_2}b^{-1}y_2x_1^{-1}, \varpi^{s_1}ab^{-1}(y_2-1)^{-1}x_1)(y_2-1,x_1)\\ 
\psi(\varpi^{t_1}a^{-1}y_2x_1^{-1}+\varpi^{s_2}a^{-1}bx_1)\ dx_1.
\end{multline*}
Then 
\begin{equation}\label{eq JSk'int}
\J_{3,S_k'}^d(a,b;i,j)=\abs{a}^{-2}\int_{q^{-k}<\abs{y_2},\,\abs{y_2-1}=\abs{ab^{-1}}}\j_k^d(y_2)\ dy_2.
\end{equation}

Applying the variable change $x_1\mapsto b^{-1}\varpi^{k}x_1$ and using \eqref{eq klinv} we have
\begin{equation}\label{two Kloostermans}
\j_k^d(y_2)=(y_2-1,b^{-1}\varpi^k)\Kl(b;\varpi^{t_2-k}y_2, \varpi^{s_1+k}ab^{-2}(y_2-1)^{-1})\Kl(a;\varpi^{s_2+k}a^{-1},\varpi^{t_1-k}a^{-1}by_2).
\end{equation}
We now analyze this product of Kloosterman sums and use this to evaluate \eqref{eq JSk'int}.  The case $i=j=0$ was treated in \cite{FO}, so we restrict to the case $i+j>0$.

\subsubsection{Evaluation of the term $\J_{3,S_k'}^l(a,b;i,j)$: Vanishing and division into cases}

We begin with the case $d=l$ and show that \eqref{eq JSk'int} often vanishes. We then divide the remaining evaluations into cases. We
will evaluate these in subsequent sections.

\begin{lemma}\label{lem jlk}
Let $\abs{b}\le \abs{a}<q^{-1}$ and $i+j>0$. Then
\begin{enumerate}
\item $\J^l_{3,S_k'}(a,b;i,j)=0$ for $1\le k\le \val(a)-2$.
\item $\J^l_{3,S'_{\val(a)-1}}(a,b;i,j)=0$ unless either
\begin{enumerate}
\item $\abs{b}<\abs{a}=q^{-2}$ and $i=2j$ or
\item $\abs{a}=\abs{b}$.
\end{enumerate}
\end{enumerate}
Furthermore, if $\abs{a}=\abs{b}$, $\abs{y_2-1}=1$ and $q\abs{a}<\abs{y_2}$ then $\j^l_{\val(a)-1}(y_2)=0$ unless $\abs{y_2}=q^2\abs{a}$.
\end{lemma}

\begin{proof}
First, suppose $|b|<|a|<q^{-1}$ and $|y_2|=|ab^{-1}|$. 
When $d=l$, then $|\varpi^{t_2-k}y_2|>q$, so applying  \eqref{eq klvan}, the first Kloosterman integral in \eqref{two Kloostermans} is zero unless 
$2k=\val(a)+i-2j$.  For $k=\val(a)-1$ this gives $\J^l_{3,S_k'}(a,b;i,j)=0$ unless $\val(a)=2+i-2j$. Since $\val(a)\ge 2$ and $i-2j\le 0$ this implies that $i=2j$ and $\val(a)=2$.
Also if $k\le \val(a)-2$, then $|\varpi^{s_2+k}a^{-1}|>q$, and using \eqref{eq klvan} we see that the second integral is zero unless $\val(a)=2k+j-2i$. 
We conclude that if $k\le \val(a)-2$ then $\J^l_{S_k'}(a,b;i,j)=0$ unless $i=j=0$.   This establishes the lemma if $|b|<|a|$.

Assume now that $\abs{b}=\abs{a}$ and that $y_2$ satisfies $\abs{1-y_2}=1$ and $q^k\abs{y_2}>1$.
Since $\abs{\varpi^{-k}y_2}>1$ and $\abs{\varpi^{2j-i+k}ab^{-2}(y_2-1)^{-1}}=\abs{a}^{-1}q^{i-2j-k}$, for $d=l$, applying  \eqref{eq klvan}, the first Kloosterman integral 
 in \eqref{two Kloostermans} vanishes unless
either $q^k\abs{y_2}=q$, $k\ge  \val(a)-1+i-2j$, or else $q^{\val(a)-k+i-2j}=q^k\abs{y_2}\ge q^2$.  Applying this for $k=\val(a)-1$ gives the last part of the lemma.
Similarly, since $\abs{a^{-1}\varpi^k}>1$ and $\abs{\varpi^{2i-j-k}a^{-1}by_2}=q^{k+j-2i}\abs{y_2}$ it follows that the second Kloosterman integral
in \eqref{two Kloostermans} is zero unless
either $k=\val(a)-1$, $q^k\abs{y_2}\le q^{1+2i-j}$, or else $k\le \val(a)-2$, $q^k\abs{y_2}=q^{\val(a)-k+2i-j}$.  If $k\le \val(a)-2$ these conditions
are only compatible if $i=j=0$.  
The lemma follows.
\end{proof}

We now turn to the evaluations of $\J_{S'_{3,\val(a)-1}}^l(a,b;i,j)$.  We will handle the following 4 cases separately:  
\begin{enumerate}
\item $\abs{b}<\abs{a}=q^{-2}$ and $j\ge 1$, $i=2j$ (Section~\ref{case a<b});
\item $\val(a)=\val(b)\ge 2$ and $2i-j>0$, $2j-i>0$ (Section~\ref{ss jl a=b});
\item $\val(a)=\val(b)\ge 2$ and $j\ge 1$, $i=2j$ (Section~\ref{case a=b,i=2j});
\item $\val(a)=\val(b)\ge 2$ and $i\ge 1$, $j=2i$ (Section~\ref{case a=b, j=2i}).
\end{enumerate}
Throughout the evaluations, we use \eqref{eq di} to systematically rewrite the Kloosterman integrals appearing in $\j_{\val(a)-1}^l(y_2)$ as cubic exponential integrals whenever possible.
(In each case above this step is different.)
When the cubic integrals have coefficients that depend on $y_2$, we expand them and interchange integration.
Doing so (remarkably!) results in integrals that may be evaluated.

\subsubsection{Evaluation in the case $\abs{b}<\abs{a}=q^{-2}$ and $j\ge 1$, $i=2j$}\label{case a<b}

\begin{lemma}\label{lem j1l b<a}
Let $\abs{b}<\abs{a}=q^{-2}$ and $j\ge 1$. Then
\[
(b,a)\J_{3,S_1'}^l(a,b;2j,j)=\abs{ab}^{-1} \Cu^*(-3b^{-1},a^{-1}b^{-1};-1).
\]
\end{lemma}

\begin{proof}
In this case, for $|y_2|=|ab^{-1}|$, 
$$
\j_1^l(y_2)=(y_2,b^{-1}\varpi)\Kl(b;\varpi^{-1}y_2, \varpi ab^{-2}(y_2-1)^{-1})\Kl(a;\varpi a^{-1},0).
$$
It follows from \eqref{eq di} that
\[
\Kl(b;\varpi^{-1}y_2, \varpi ab^{-2}(y_2-1)^{-1})=(a\varpi^{-1}y_2,b)\Cu(-3\varpi b^{-1},\varpi^3 a^{-1}b^{-1}(1-y_2^{-1});0),
\]
while from \eqref{eq kl as gauss}
\[
\Kl(a;\varpi a^{-1},0)=q^{-1}(\varpi,a) \g_a.
\]

Expanding the cubic exponential integral we find $\J_{3,S_1'}^l(a,b;2j,j)$ equals
\[
q^{3}(\varpi,ab^{-1})(a,b)\g_a\int_{\abs{y_2}=\abs{ab^{-1}}} (y_2,\varpi) \int_\O\psi(-3\varpi b^{-1} x+\varpi^3 a^{-1}b^{-1}(1-y_2^{-1})x^3) \ dx\ dy.
\]
We express the inner integral over $x$ as a sum of the integrals over $\p$ and over $\O^*$. After changing the order of integration, the first summand has an inner integral of
\[
\int_{\abs{y_2}=\abs{ab^{-1}}} (y_2,\varpi)\ dy_2=0,
\]
and the second summand equals
\[
q^{3}(\varpi,ab^{-1})(a,b)\g_a\int_{\O^*} \psi(-3\varpi b^{-1} x+\varpi^3 a^{-1}b^{-1}x^3) \int_{\abs{y_2}=\abs{ab^{-1}}} (y_2,\varpi) \psi(-\varpi^3 a^{-1}b^{-1}x^3y_2^{-1})\ dy_2\ dx.
\]
Using the variable change $y_2\mapsto ab^{-1}y_2$ we find that
\[
(\varpi,ab^{-1})\int_{\abs{y_2}=\abs{ab^{-1}}} (y_2,\varpi) \psi(-\varpi^3 a^{-1}b^{-1}x^3y_2^{-1})\ dy_2=\abs{ab^{-1}}\Kl(\varpi^{-1};0,-\varpi^3a^{-2}x^3).
\]
Since $x\in \O^*$ it follows from \eqref{eq kl as gauss} that this last expression equals
$
q^{-1}\abs{ab^{-1}}(\varpi,ab)\g.
$
We conclude that
\[
\J_{3,S_1'}^l(a,b;2j,j) = \abs{ab}^{-1}(a,b) (a,\varpi)\g_a\g\, \Cu^*(-3\varpi b^{-1},\varpi^3 a^{-1}b^{-1};0).
\]
Since $\val(a)=2$ it follows from \eqref{the-gauss-sum} that $(a,\varpi)\g_a=\bar \g$ and therefore 
from \eqref{eq abs gauss} that $(a,\varpi)\g_a\g=q$. It further follows from \eqref{eq cubicvarch} that 
$$\Cu^*(-3\varpi b^{-1},\varpi^3 a^{-1}b^{-1};0)=q\Cu^*(-3b^{-1},a^{-1}b^{-1};-1).$$
The lemma follows.
\end{proof}

\subsubsection{Evaluation in the case $\val(a)=\val(b)\ge 2$ and $2i-j>0$, $2j-i>0$}\label{ss jl a=b}
\begin{lemma}\label{lem jl a=b}
Let $2j-i,\,2i-j>0$ and $\val(a)=\val(b)\ge 2$. Then
\[
|a|^2\J_{3,S'_{\val(a)-1}}^l(a,b;i,j) =\begin{cases} \abs{a}(1-q^{-1})& 3\mid \val(a) \\ 0 & 3\nmid\val(a)>2 \\ q^{-2}(b,a)\overline{\g_b} & \val(a)=2. \end{cases}
\]
\end{lemma}
\begin{proof}
We have
\begin{equation*}
|a|^2J_{3,S'_{\val(a)-1}}^l(a,b;i,j) = \int_{\abs{y_2}>q\abs{a}, |y_2-1|=1} \j_{\val(a)-1}^l(y_2)\,dy_2.
\end{equation*}
By Lemma~\ref{lem jlk},  the integrand is nonzero only when $|y_2|=q^2|a|$.
Note further that if $2j>i$, $\abs{y_2}=q^2\abs{a}$ and $|y_2-1|=1$, then, by \eqref{eq kl as gauss}, for $k=\val(a)-1$ we have
\[
\Kl(b;\varpi^{-k}y_2, \varpi^{2j-i+k}ab^{-2}(y_2-1)^{-1})=\Kl(b;\varpi^{1-\val(a)}y_2,0)=q^{-1}(\varpi^{1-\val(a)}y_2,b)\g_b
\]
and similarly, if $2i>j$, on this region
\[
\Kl(a;\varpi^{k}a^{-1},\varpi^{2i-j-k}a^{-1}by_2)=\Kl(b;a^{-1}\varpi^{\val(a)-1},0)=q^{-1}(a^{-1}\varpi^{\val(a)-1},b)\g_b.
\]
These give
\[
|a|^2\J_{3,S'_{\val(a)-1}}^l(a,b;i,j) = q^{-2}\g_b^2 (b,a)\int_{\abs{y_2}=q^2\abs{a}, |y_2-1|=1} (\varpi,y_2-1) (y_2,b) \, dy_2.
\]
If $\val(a)>2$ then the domain is given by $\abs{y_2}=q^2\abs{a}<1$ so that
\begin{multline*}
|a|^2\J_{3,S'_{\val(a)-1}}^l(a,b;i,j) = q^{-2}\g_b^2 (b,a)\int_{\abs{y_2}=q^2\abs{a}} (y_2,b) \, dy_2\\
=\begin{cases}\g_b^2 (\varpi,b)\abs{a}(1-q^{-1}) & 3\mid \val(b)\\ 0 & 3\nmid \val(b).  \end{cases}
\end{multline*}
Note that if $3|\val(b)$ then $\g_b=-(\varpi,b)$. This completes the case $\val(a)>2$. 
Suppose instead that $\val(a)=2$. We have
\[
|a|^2\J_{S'_{3,\val(a)-1}}^l(a,b;i,j) =q^{-2}\g_b^2 (b,a)\int_{y_2\in\O^*,1-y_2\in \O^*} (b^{-1},y_2(1-y_2))\, dy_2.
\]
The rest of the lemma follows from Lemma \ref{lem jac} and equations \eqref{eq abs gauss}, \eqref{conjugate gauss}, and \eqref{the-gauss-sum}.
\end{proof}

\subsubsection{Evaluation in the case $\val(a)=\val(b)\ge 2$ and $j\ge 1$, $i=2j$}\label{case a=b,i=2j}
\begin{lemma}\label{lem a=b,i=2j}
Let $\val(a)=\val(b)\ge 2$ and $j\ge 1$. We have
\[
(b,a) |a|^2\J_{S'_{3,\val(a)-1}}^l(a,b;2j,j) =\begin{cases}-q^{-1}\abs{a} & 3\mid \val(a) \\ -\abs{a} & 3\nmid \val(a)>2 \\ \Cu^*(-3b^{-1},a^{-1}b^{-1};-1) & \val(a)=2. \end{cases}
\]
\end{lemma}
\begin{proof}
Consider first the case $\val(a)>2$. Again from Lemma~\ref{lem jlk}, \eqref{two Kloostermans}  vanishes unless $\abs{y_2}=q^2\abs{a}$, and also for $k=\val(a)-1$ we have
\[
\Kl(a;\varpi^{k}a^{-1},\varpi^{3j-k}a^{-1}by_2)=\Kl(b;a^{-1}\varpi^{\val(a)-1},0)=q^{-1}(a^{-1}\varpi^{\val(a)-1},b)\g_b.
\]
Thus
\begin{multline*}
|a|^2\J_{3,S'_{\val(a)-1}}^l(a,b;2j,j) =\\
q^{-1}(a^{-1}\varpi^{\val(a)-1},b)\g_b \int_{\abs{y_2}=q^2\abs{a}}\Kl(b;\varpi^{1-\val(a)}y_2, \varpi^{\val(a)-1}ab^{-2}(y_2-1)^{-1})
\ dy_2.
\end{multline*}
Note that in the domain of integration $\varpi^{\val(a)-1}ab^{-2}(y_2-1)^{-1}\in -\varpi^{\val(a)-1}ab^{-2}+\O$. 
Expanding out the Kloosterman integral,
\begin{multline*}
|a|^2\J_{3,S'_{\val(a)-1}}^l(a,b;2j,j)  =q^{-1}(a^{-1}\varpi^{\val(a)-1},b)\g_b \int_{\O^*}(b,u)\psi(-\varpi^{\val(a)-1}ab^{-2}u^{-1})\\ \int_{\abs{y_2}=q^2\abs{a}} \psi(\varpi^{1-\val(a)}y_2u)\ dy_2\ du.
\end{multline*}
Since $\int_{\abs{y_2}=q^2\abs{a}} \psi(\varpi^{1-\val(a)}y_2u^{-1})\ dy_2=-q\abs{a}$ is independent of $u\in \O^*$ and 
\[
\int_{\O^*}(b,u)\psi(-\varpi^{\val(a)-1}ab^{-2}u^{-1})\ du =\Kl(b;0,-\varpi^{\val(a)-1}ab^{-2})=q^{-1}(b,\varpi^{\val(a)-1}a)\bar\g_b
\]
we conclude that
\[
(b,a) |a|^2\J_{3,S'_{\val(a)-1}}^l(a,b;2j,j)=-q^{-1}\abs{a}\g_b\bar\g_b=\begin{cases}-q^{-1}\abs{a} & 3\mid \val(a) \\ -\abs{a} & 3\nmid \val(a). \end{cases}
\]

Suppose instead that $\val(a)=2$.
Then
\begin{multline*}
|a|^2\J_{3,S'_{\val(a)-1}}^l(a,b;2j,j) =\\
q^{-1}(a^{-1}\varpi,b)\g_b\int_{|y_2|=|1-y_2|=1}(\varpi,y_2-1)\Kl(b;\varpi^{-1}y_2,\varpi ab^{-2}(y_2-1)^{-1})\ dy_2.
\end{multline*}
By \eqref{eq di} we have
\[
\Kl(b;\varpi^{-1}y_2,\varpi ab^{-2}(y_2-1)^{-1})=(a,b)(b,\varpi(1-y_2)y_2)\Cu(-3\varpi b^{-1},\varpi^3a^{-1}b^{-1}(1-y_2^{-1});0).
\]
Consequently,
\begin{multline*}
|a|^2\J_{3,S'_{\val(a)-1}}^l(a,b;2j,j) =\\
q^{-1}\g_b \int_{|y_2|=|1-y_2|=1}(b,y_2) \int_\O \psi(-3\varpi b^{-1}x+\varpi^3a^{-1}b^{-1}(1-y_2^{-1})x^3)\ dx\ dy_2.
\end{multline*}
We decompose this into a sum of two terms, by integrating $x$ over $\p$ in the first and over $\O^*$ in the second. The first summand becomes
\[
q^{-2}\g_b \int_{|y_2|=|1-y_2|=1}(b,y_2) \ dy_2=-q^{-3}\g_b
\]
where the last equality is obtained by integrating over $y_2\in \O^*$ (the outcome is zero) and subtracting the integral over $y_2\in 1+\p$.
After changing order of integration, the second summand becomes
\[
q^{-1}\g_b \int_{\O^*}\psi(-3\varpi b^{-1}x+\varpi^3a^{-1}b^{-1} x^3)\int_{|y_2|=|1-y_2|=1}(b,y_2)\psi(-\varpi^3a^{-1}b^{-1} x^3 y_2^{-1})\ dy_2\ dx.
\]
Note that
\begin{multline*}
\int_{|y_2|=|1-y_2|=1}(b,y_2)\psi(-\varpi^3a^{-1}b^{-1} x^3 y_2^{-1})\ dy_2\\=\Kl(b;0,-\varpi^3a^{-1}b^{-1}x^3)-\int_{1+\p}\psi(-\varpi^3a^{-1}b^{-1} x^3 y_2^{-1})\ dy_2\\
=q^{-1}[(a,b)\bar\g_b-\psi(-\varpi^3a^{-1}b^{-1}x^3)].
\end{multline*}
Plugging this in the inner integral, the second summand equals 
\[
q^{-1}(a,b)\Cu^*(-3\varpi b^{-1},\varpi^3a^{-1}b^{-1};0)-q^{-2}\g_b\int_{\O^*}\psi(-3\varpi b^{-1}x)\ dx.
\]
Since $\int_{\O^*}\psi(-3\varpi b^{-1}x)\ dx=-q^{-1}$ all together we conclude that
\[
|a|^2\J_{3,S'_{\val(a)-1}}^l(a,b;2j,j) =q^{-1}(a,b)\Cu^*(-3\varpi b^{-1},\varpi^3a^{-1}b^{-1};0).
\]
Applying \eqref{eq cubicvarch} the remaining part of the lemma follows.
\end{proof}

\subsubsection{Evaluation in the case $\val(a)=\val(b)\ge 2$ and $i\ge 1$, $j=2i$}\label{case a=b, j=2i}

\begin{lemma}\label{lem a=b, j=2i}
Let $\val(a)=\val(b)\ge 2$ and $i\ge 1$. We have
\[
(b,a)|a|^2\J_{3,S'_{\val(a)-1}}^l(a,b;i,2i)  =\begin{cases}-q^{-1}\abs{a} & 3\mid \val(a) \\ -\abs{a} & 3\nmid \val(a)>2 \\ \Cu^*(-3a^{-1},a^{-1}b^{-1};-1) & \val(a)=2. \end{cases}
\]
\end{lemma}
\begin{proof}
By Lemma~\ref{lem jlk} we have
\begin{multline*}
|a|^2\J_{3,S'_{\val(a)-1}}^l(a,b;i,2i) =q^{-1}(\varpi^{1-\val(a)},b)\g_b  \\ \int_{\abs{y_2}=q^2\abs{a},1-y_2\in\O^*}(\varpi,y_2-1)(y_2,b)\Kl(b;\varpi^{\val(a)-1}a^{-1},\varpi^{1-\val(a)}a^{-1}b y_2)\ dy_2.
\end{multline*}
Expanding the Kloosterman integral and changing the order of integration, this becomes
\begin{multline}\label{alt-expression-37}
q^{-1}(\varpi^{1-\val(a)},b)\g_b  \int_{\O^*}(b,u)\psi(a^{-1}\varpi^{\val(a)-1}u)\\ \int_{\abs{y_2}=q^2\abs{a},1-y_2\in\O^*}(\varpi,y_2-1)(y_2,b)\psi(\varpi^{1-\val(a)}a^{-1}b y_2u^{-1})\ dy_2\ du.
\end{multline}

Consider first the case $\val(a)>2$. The domain of integration is then $\abs{y_2}=q^2\abs{a}$ and for such $y_2$ we have $1-y_2\in1+\p$. The inner integral is therefore
\begin{multline*}
\int_{\abs{y_2}=q^2\abs{a}}(y_2,b)\psi(\varpi^{1-\val(a)}a^{-1}b u^{-1}y_2) \ dy_2\\
=q^2\abs{a}(\varpi^{-2}a,b)\Kl(b^{-1};\varpi^{-1-\val(a)}bu^{-1},0)=
q\abs{a}(\varpi^{\val(a)-1}au,b)\bar\g_b
\end{multline*}
and therefore
\begin{multline*}
|a|^2\J_{3,S'_{\val(a)-1}}^l(a,b;i,2i) =
\abs{a}(a,b) \g_b\bar\g_b\int_{\O^*}\psi(a^{-1}\varpi^{\val(a)-1}u)\ du\\
=\begin{cases}-q^{-1}\abs{a}(a,b) & 3\mid \val(a)\\ -\abs{a}(a,b) & 3\nmid \val(a). \end{cases}
\end{multline*}

Suppose instead that $\val(a)=2$. Then $y_2$ is integrated over $\O^*\setminus(1+\p)$, and the variable changes $u\mapsto y_2 u$ and then $y_2\mapsto 1-y_2$ in
\eqref{alt-expression-37} give that
\begin{multline*}
|a|^2\J_{3,S'_{1}}^l(a,b;i,2i) 
 =q^{-1}(\varpi^{-1},b)\g_b\int_{\O^*}(b,u)\psi(\varpi^{-1}a^{-1}b u^{-1}+a^{-1}\varpi u)\\ \int_{\O^*\setminus (1+\p)}(\varpi,y_2) \psi(-a^{-1}\varpi y_2u)\ dy_2\ du.
\end{multline*}
We have
\begin{multline*}
\int_{\O^*\setminus (1+\p)}(\varpi,y_2) \psi(-a^{-1}\varpi y_2u)\ dy_2=\Kl(\varpi;-a^{-1}\varpi u,0)-q^{-1}\psi(-a^{-1}\varpi u)=
\\
q^{-1}[(a^{-1}u,\varpi)\g-\psi(-a^{-1}\varpi u)].
\end{multline*}
Note further that $(\varpi^{-1},b)\g_b=\bar\g$. Applying \eqref{eq di} we conclude that
\begin{align*}
|a|^2\J_{3,S'_{1}}^l(a,b;i,2i) &=q^{-1}(\varpi,a)\Kl(b\varpi^{-1};a^{-1}\varpi,\varpi^{-1}a^{-1}b)-q^{-2}(b,\varpi)\g_b\Kl(b;0,\varpi^{-1}a^{-1}b)\\
&=q^{-1}(\varpi,ab^{-1})\Cu(-3a^{-1}\varpi,\varpi^3a^{-1}b^{-1};0)-q^{-2}(a,b).
\end{align*}
To simplify this, note that since $\val(a)=\val(b)=2$ we have $(\varpi,ab^{-1})=(a,b)$ and $\Cu(-3a^{-1}\varpi,\varpi^3a^{-1}b^{-1};-1)=q^{-1}$.  Thus we find that
$$|a|^2\J_{3,S'_{1}}^l(a,b;i,2i) =q^{-1}(a,b)\Cu^*(-3a^{-1}\varpi,\varpi^3a^{-1}b^{-1};0).$$ Then the lemma follows from \eqref{eq cubicvarch}.
\end{proof}

This completes the evaluation of the term $\J_{3,S_k'}^l(a,b;i,j)$ in all cases.

\subsubsection{Evaluation of $\J_{3,S_k'}^r(a,b;i,j)$}
We require this term only when $|b|<|a|$.  In this situation, we have the following lemma.
\begin{lemma}\label{le jrk}
Let $\abs{b}< \abs{a}<q^{-1}$ and $i+j>0$. Then
\begin{enumerate}
\item\label{lemma-Jr-part1} $\J_{3,S_k'}^r(a,b;i,j)=0$ for $2\le k\le \val(a)-1$ and
\item\label{lemma-Jr-part2} $\J_{3,S_1'}^r(a,b;i,j)=0$ unless $\abs{a}=q^{-2}$ and $2j=i$. 
\item\label{lemma-Jr-part3} When $\abs{a}=q^{-2}$, $\J_{3,S_1'}^r(a,b;2j,j)=q^{2}\abs{b}^{-1}(a,b)\Cu(-3b^{-1}, a^{-1}b^{-1};-1)$.
\end{enumerate}
\end{lemma}

\begin{proof}
Parts~\ref{lemma-Jr-part1} and~\ref{lemma-Jr-part2} follow by substituting \eqref{two Kloostermans} into \eqref{eq JSk'int} and using \eqref{eq klvan}
by an argument that is similar to, but simpler than, that of Lemma~\ref{lem jlk}. We omit the detail.  For part~\ref{lemma-Jr-part3}, if $|a|=q^{-2}$,
$\abs{y_2}=\abs{ab^{-1}}>1$, then
\[
\Kl(a;\varpi^{3j+1}a^{-1} ,\varpi^{-1}a^{-1}b y_2)=\Kl(a;0,\varpi^{-1}a^{-1}b y_2)=q^{-1}(a,\varpi^{-1}b y_2)\bar\g_a=q^{-1}(a,b y_2)\g.
\]
Thus
\begin{equation*}
\J_{3,S_1'}^r(a,b;i,j)=q^3(a,b)\g\int_{\abs{y_2}=\abs{ab^{-1}}}  (y_2,a^{-1}b^{-1}\varpi) \Kl(b;\varpi^{-1}y_2,\varpi ab^{-2}(y_2-1)^{-1}) \ dy_2.
\end{equation*}
Note further that since $ab^{-1}(y_2-1)^{-1}\in \O^*$ we have
$$\Kl(b;\varpi^{-1}y_2,\varpi ab^{-2}(y_2-1)^{-1})=(b,ay_2^{-1}) \Kl(b,ab^{-1}\varpi^{-1}y_2(y_2-1)^{-1},b^{-1}\varpi)$$
and that
\[
ab^{-1}\varpi^{-1}y_2(y_2-1)^{-1}=ab^{-1}\varpi^{-1}(1-y_2^{-1})^{-1}\in ab^{-1}\varpi^{-1}(1+y_2^{-1})+\O.
\]
Thus
\[
\Kl(b;\varpi^{-1}y_2,\varpi ab^{-2}(y_2-1)^{-1})=(b,ay_2^{-1})\Kl(b;ab^{-1}\varpi^{-1}(1+y_2^{-1}),b^{-1}\varpi).
\]
It follows that
\begin{multline*}
\J_{3,S_1'}^r(a,b;i,j)=q^3\g\int_{\O^*}(b,u)\psi(ab^{-1}\varpi^{-1}u+b^{-1}\varpi u^{-1})\\ \int_{\abs{y_2}=\abs{ab^{-1}}}  (y_2,a^{-1} \varpi) \psi(ab^{-1}\varpi^{-1}uy_2^{-1})\ dy_2\ du.
\end{multline*}
After the variable change $y_2\mapsto ab^{-1}uy_2$ this becomes
\[
\J_{3,S_1'}^r(a,b;i,j)=q\abs{b}^{-1}\g(b,a)(ab^{-1},\varpi)\Kl(ab\varpi^{-1};ab^{-1}\varpi^{-1},b^{-1}\varpi)\Kl(a\varpi^{-1};0,\varpi^{-1}).
\]
Note that
\[
\Kl(a\varpi^{-1};0,\varpi^{-1})=\Kl(\varpi;0,\varpi^{-1})=q^{-1}\bar \g
\]
and that applying \ref{eq di} we have
\begin{multline*}
\Kl(ab\varpi^{-1};ab^{-1}\varpi^{-1},b^{-1}\varpi)=(ab\varpi^{-1},\varpi^{-2}a)\Cu(-3b^{-1}\varpi,\varpi^3a^{-1}b^{-1};0)=\\
(b,a)(\varpi,ab^{-1})q\Cu(-3b^{-1},a^{-1}b^{-1};-1).
\end{multline*}
Combining these expressions, part~\ref{lemma-Jr-part3} follows.
\end{proof}

\subsection{Proof of Proposition~\ref{prop jl}} 
We combine the terms from the above computations in  order to establish Proposition~\ref{prop jl}.  We systematically make use of
\eqref{eq normJ}, which relates $J^l(a,b;i,j)$ and $\J^l(a,b;i,j)$. 
We also make use of information
about Kloosterman integrals (Section~ \ref{Kloosterman sums with cubic characters}), cubic exponential
integrals (Section~\ref{Cubic exponential notation}) and the relation between these two classes of
integrals (Section~\ref{Kl-Cu-relations}).   We first treat the case $i=j=0$ and then the case $i+j>0$.

\subsubsection{The case $i=j=0$}\label{FO-paper1}
In \cite{FO}, the authors showed that
for $\abs{b}\le \abs{a}$ we have
\begin{multline*}
(b,a) J^l(a,b;0,0)=(b,a) J^r(a,b;0,0)=\\
\begin{cases}
1 & \abs{a}=\abs{b}=1 \\
2q & \abs{a}=\abs{b}=q^{-1} \\
3\abs{a}^{-1}+\abs{ab}^{-1}\sum_{k=\lfloor\frac{\val(a)+1}2\rfloor}^{\val(a)-1} \sum_{\ell=0}^2 \Cu^*(b^{-1}+\rho^\ell a^{-1},-3^{-3}a^{-1}b^{-1};-k) & \abs{a}=\abs{b} \le q^{-2} \\
\abs{b}^{-1}\Cu(b^{-1},-3^{-3}a^{-1}b^{-1};0) & \abs{b}< \abs{a}=1 \\
\abs{ab}^{-1}\sum_{\ell=0}^2 \Cu^*(b^{-1}+\rho^\ell a^{-1},-3^{-3}a^{-1}b^{-1};-{\tfrac{\val(a)}2}) & \abs{b}<\abs{a}<1 \text{ and }2|\val(a) \\
0 &  \abs{b}<\abs{a}<1 \text{ and }2\nmid \val(a) \\
0 & \abs{a}>1.
\end{cases}
\end{multline*}
Then one may check that these expressions are equal to the ones presented in Proposition~\ref{prop jl}.
Indeed, if $\abs{a}=\abs{b}\le q^{-2}$ then $\Cu(b^{-1}+\rho^\ell a^{-1},-3^{-3}a^{-1}b^{-1};-{\val(a)})=\abs{a}$. Thus in this case (also changing by the unit $-3$)
\begin{multline*}
3\abs{a}^{-1}+\abs{ab}^{-1}\sum_{k=\lfloor\frac{\val(a)+1}2\rfloor}^{\val(a)-1} \sum_{\ell=0}^2 \Cu^*(b^{-1}+\rho^\ell a^{-1},-3^{-3}a^{-1}b^{-1};-k) \\
=\abs{ab}^{-1}\sum_{\ell=0}^2 \Cu(-3(b^{-1}+\rho^\ell a^{-1}),a^{-1}b^{-1};{\lfloor\frac{-\val(a)}2\rfloor}).
\end{multline*}
Also, if $|b|<|a|<1$ then
$\Cu(-3(b^{-1}+\rho^\ell a^{-1}),a^{-1}b^{-1};{\lfloor\frac{-\val(a)}2\rfloor})=0$
if $2\nmid \val(a)$.

Finally, if $|b|=q|a|$, this case is treated via Corollary~\ref{cor jij}.

\subsubsection{The case $i+j>0$, $\abs{b_j}\le \abs{a_i}\le 1$}
We suppose that $i+j>0$ and $\abs{b_j}\le \abs{a_i}\le 1$. 
We begin by keeping track of the Hilbert symbols in making the transition from $J^l$ to $\J^l$.
Using \eqref{eq normJ}, 
$$(b,a)J^l(a,b;i,j)=(b,a)(\varpi,b)^j(\varpi,a)^{i-j}\J^l(a_i,b_j;i,j)=(\varpi,ab)^{i+j}(b_j,a_i)\J^l(a_i,b_j;i,j).$$
This equation will be used repeatedly in the sequel.

We begin with the case $\abs{a_i}=1$. The case $|b_j|=1$ is immediate from Lemma~\ref{lem ja=1} and the above.
Similarly if $\abs{a_i}=1$ and $\abs{b_j}<1$, then the evaluation follows from Lemma~\ref{lem ja=1} 
since if $|a|=1$, $|b|<1$, then
$$\Kl(a^{-1}b;b^{-1},\varpi^{2j-i}ab^{-1})=\begin{cases} (a,b)\Cu(-3b^{-1},a^{-1}b^{-1};0)&\text{if $i=2j$}\\
q^{-1}(b,a)\g_{a^{-1}b}&\text{if $2j-i>0$, $|b|=q^{-1}$}\\
0&\text{otherwise.}
\end{cases}$$
The matching of the expression to the result in Proposition~\ref{prop jl} then follows from \eqref{eq cubicvarch} when $i=2j$ and
from \eqref{the-gauss-sum} and, for $2j-i>0$, the equation $(a,b\varpi^{-1})=1$ when $|a|=1$, $|b|=q^{-1}$.

For the case $\abs{a_i}<1$, recall that
\[
\J^l(a,b;i,j)=\sum_{t=1}^3 \J_t^l(a,b;i,j).
\]
The evaluation of $\J_1^l(a,b;i,j)+\J_2^l(a,b;i,j)$ is given in Proposition~\ref{j1 plus j2}. Recall further that
\[
\J_3^l(a,b;i,j)=\J_{3,S_0}^l(a,b;i,j)+\sum_{k=1}^{\val(a)-1} \J_{3,S_k'}^l(a,b;i,j),
\]
where $\J_{3,S_0}^l(a,b;i,j)=0$ unless $\abs{a}=\abs{b}$ and in this case it is computed in Lemma \ref{lem j0}.
Moreover, Lemma~\ref{lem jlk} shows that 
$$\sum_{k=1}^{\val(a)-1} \J_{3,S_k'}^l(a,b;i,j)=\J_{3,S'_{\val(a)-1}}^l(a,b;i,j)$$ 
and this last quantity is zero
unless either (i)~$\abs{b}<\abs{a}=q^{-2}$ and $i=2j$, or (ii)~$\abs{a}=\abs{b}\le q^{-2}$,
and is computed in these cases in Lemmas~\ref{lem j1l b<a}, \ref{lem jl a=b}, \ref{lem a=b,i=2j}, and \ref{lem a=b, j=2i}.
Combining these terms in all cases, one arrives at the evaluation of the proposition. 

We carry out the details for one case to illustrate these computations. 
Suppose $|b_j|=|a_i|=q^{-2}$, $i+j>0$.
In this case $\J^l_1(a_i,b_j;i,j)+\J^l_2(a_i,b_j;i,j)=0$, and so 
$$J^l(a,b;i,j)=(\varpi,b)^j(\varpi,a)^{i-j}\J_3^l(a_i,b_j;i,j).$$
Also
$$\J_{3,S_0}^l=(a_i,b_j)|a_i|^{-1}[\delta_{2i,j}+\delta_{i,2j}]$$
and
$$\J_{3,S_1}'=\begin{cases}|a_i|^{-2}q^{-2}\overline{\g_{a_i}}&\text{$2i-j>0$ and $2j-i>0$}\\
|a_i|^{-2}(a_i,b_j)\Cu^*(-3b_j^{-1},a_i^{-1}b_j^{-1};-1)&i=2j>0\\
|a_i|^{-2}(a_i,b_j)\Cu^*(-3a_i^{-1},a_i^{-1}b_j^{-1};-1)&j=2i>0\\
\end{cases}$$
If $i=2j$ or $j=2i$ the Hilbert symbols are:
$$(\varpi,b)^j(\varpi,a)^{i-j}(a_i,b_j)=(a,b)(\varpi,ab)^{i+j}=(a,b).$$
By \eqref{C and Cstar}
\begin{align*}
\Cu^*(-3b_j^{-1},a_i^{-1}b_j^{-1};-1)&=\Cu(-3b_j^{-1},a_i^{-1}b_j^{-1};-1)-\Cu(-3b_j^{-1},a_i^{-1}b_j^{-1};-2)\\
&=\Cu(-3b_j^{-1},a_i^{-1}b_j^{-1};-1)-q^{-2}.
\end{align*}
Also, if $i=2j$ then
$$\Cu(-3b_j^{-1},a_i^{-1}b_j^{-1};-1)=q^{-j}\Cu(-3b^{-1},a^{-1}b^{-1};j-1)$$
and we find that in this case
$$(b,a)J^l(a,b;i,j)=
q^{4-j}\Cu(-3b^{-1},a^{-1}b^{-1};j-1)=q^2\abs{b}^{-1}\Cu(-3b^{-1},a^{-1}b^{-1};j-1),$$
as claimed.  The case $j=2i$ is similar.   Finally, if $2j-i>0$, $2i-j>0$ then we find that
\begin{align*}
(b,a)J^l(a,b;i,j)&=(b,a)(\varpi,b)^j(\varpi,a)^{i-j}q^2 (b_j,a_i)\overline{\g_{b_j}}\\
&=(a,b)(\varpi,b)^{j-i-1}(\varpi,a)^{i}q^2\g.
\end{align*}
Also we have
$$1=(b\varpi^{j-2},a\varpi^{i-2})=(b,a)(\varpi,a)^{j-2} (\varpi,b)^{2-i}.$$
Multiplying by this factor gives $(\varpi,ab)^{i+j+1}\g q^2$, as claimed.

\subsubsection{The case $i+j>0$, $\abs{b_j}=q \abs{a_i}$}
Suppose now that $i+j>0$, $\abs{b_j}=q \abs{a_i}$. 
First, by Corollary~\ref{cor jij}
and \eqref{eq normJr}, 
$$J^l(a,b;i,j)=\overline{J^r(b,a;j,i)}=(\varpi,b)^j(\varpi,a)^{i-j}\overline{\J^r(b_j,a_i;j,i)}.$$
If $|b_j|>1$, this is zero by Lemma~\ref{lem ja=1}.  If $|b_j|=1$, the evaluation is handled by the same
lemma, and analyzed as above.  Suppose that $|b_j|<1$.  Then 
\[
\J^r(b_j,a_i;j,i)=\sum_{t=1}^3 \J_t^r(b_j,a_i;j,i).
\]
The sum $\J_1^r(b_j,a_i;j,i)+\J_2^r(b_j,a_i;j,i)$ is evaluated in Proposition~\ref{j1 plus j2}.  Also combining \eqref{sum-3-va neq vc} with
Lemma~\ref{le jrk}, we find that in this case $\J_3^r(b_j,a_i;j,i)=\J_{3,S'_1}^r(b_j,a_i;j,i)$, and this term is zero unless $|b_j|=q^{-2}$, $2i=j$, and then
it is evaluated in Lemma~\ref{le jrk}.  Substituting in these evaluations,  the proposition follows in this case as well.

We give one example.
Suppose $|a_i|=q^{-2}$,  $|b_j|=q^{-1}$, $2i-j=1$.  Then
\begin{align*}
J^l(a,b;i,j)=&(\varpi,b)^j(\varpi,a)^{i-j}\overline{\J^r(b_j,a_i;j,i)}\\
=&(\varpi,b)^j(\varpi,a)^{i-j}(a_i,b_j){\g}|a_i|^{-1}(\varpi,a_ib_j)\Cu(-3b_ja_i^{-1},\varpi^{-1}b_j^2a_i^{-1};0)\\
=&(a,b) \g (\varpi,ab)^{i+j+1} |a_i|^{-1} \Cu(-3\varpi a_i^{-1},\varpi^2 b_j^{-1}a_i^{-1};0)\\
=&(a,b)\g q |a|^{-1}\Cu(-3 a^{-1},b^{-1}a^{-1};i-1).
\end{align*}

\subsection{Evaluation of the big cell orbital integral for the 3-fold cover of $\SL_3$}\label{sec: evaluation of J integral}

We are now able to give the evaluation of the big cell orbital integrals.  The case $m=n=0$ is treated in \cite{FO},
so we suppose that $m+n>0$.  We have the following result.

\begin{proposition}\label{prop j}  Let $m,n\geq0$ with $m+n\ge1$. 
For $a,b\in F^*$ such that $\abs{b_{m+2n}}\le \abs{a_{n+2m}}$
we have $(b,a)J_{m,n}(a,b)=$
\[
\begin{cases} 
1 &  \abs{b_{m+2n}}=\abs{a_{n+2m}}=1 \\ 
\g_{ab}+\bar\g_{ab} &  n\ge 1,\ \abs{b_{m+2n}}=q^{-1}, \abs{a_{n+2m}}=1\\
q & n\ge 1,\ \abs{b_{m+2n}}=q^{-2}, \abs{a_{n+2m}}=1\\
\abs{b}^{-1}\Cu(-3b^{-1},a^{-1}b^{-1};m) & n=0,  \abs{b_{m+2n}}\le q^{-1},\ \abs{a_{n+2m}}=1\\
2q & \abs{b_{m+2n}}= \abs{a_{n+2m}}=q^{-1} \\    
2q^2 & n\ge 1,\ \abs{b_{m+2n}}=q^{-2},\ \abs{a_{n+2m}}=q^{-1} \\  
q^2[\g_{ab}+\bar \g_{ab}] & n\ge 2,\ \abs{b_{m+2n}}= q^{-3},\,\abs{a_{n+2m}}=q^{-1}\\
2q^2\abs{b}^{-1}\Cu(-3b^{-1},a^{-1}b^{-1};m) & n=1,\,\ \abs{b_{m+2n}}\le q^{-3},\,\abs{a_{n+2m}}=q^{-1}\\
3q^2[\g_{ab}+\bar \g_{ab}] & m,n\ge 1,\ \abs{b_{m+2n}}= \abs{a_{n+2m}}=q^{-2} \\
3q^2\abs{a}^{-1}\Cu(-3a^{-1},a^{-1}b^{-1};n-1)  & m=0,\ n\ge1, \ \abs{b_{m+2n}}= \abs{a_{n+2m}}=q^{-2} \\
3q^2\abs{b}^{-1} \Cu(-3b^{-1},a^{-1}b^{-1};m-1) & n=0,\ m\ge1, \ \abs{b_{m+2n}}\le \abs{a_{n+2m}}=q^{-2} \\
2q^3 & n\ge 2,\ \abs{b_{m+2n}}=q^{-3},\ \abs{a_{n+2m}}=q^{-2} \\
q^4 & n\ge 2,\ \abs{b_{m+2n}}=q^{-4},\ \abs{a_{n+2m}}=q^{-2} \\
3q^3 & n=1,\ \abs{b_{m+2n}}=q^{-3},\, \abs{a_{n+2m}}=q^{-2} \\
2q^4& m,n\ge 1,\ \abs{b_{m+2n}}=\abs{a_{n+2m}}=q^{-3} \\
q^5 & m,n\ge 1,\ \abs{b_{m+2n}}=\abs{a_{n+2m}}=q^{-4} \\
q^4[\g_{ab}+\bar\g_{ab}] & n\ge 2,\ m\ge 1,\ \abs{b_{m+2n}}=q^{-4},\ \abs{a_{n+2m}}=q^{-3}\\
q^4\abs{a}^{-1}\Cu(-3a^{-1},a^{-1}b^{-1};n-2)  & n\ge 2,\ m=0,\ \abs{b_{m+2n}}=q^{-4},\ \abs{a_{n+2m}}=q^{-3}\\
q^4\abs{a}^{-1}\Cu(-3a^{-1},a^{-1}b^{-1};n-2)  & n\ge 2,\ m=0,\ \abs{b_{m+2n}}=q^{-4},\ \abs{a_{n+2m}}=q^{-4}\\
2q^4\abs{b}^{-1}\Cu(-3b^{-1},a^{-1}b^{-1};m-1)& n=1,m\ge 1,\ \abs{b_{m+2n}}\le q^{-4},\ \abs{a_{n+2m}}= q^{-3}\\
q^2\abs{ab}^{-1}\sum_{\ell=0}^2 \Cu(-3(a^{-1}+\rho^\ell b^{-1}),a^{-1}b^{-1};\lfloor-\frac{\val(b)}2\rfloor) & n=1,\ m=0,\ \abs{b_{m+2n}}=\abs{a_{n+2m}}\le q^{-3}  \\
q^2\abs{ab}^{-1}\sum_{\ell=0}^2 \Cu(-3(b^{-1}+\rho^\ell a^{-1}),a^{-1}b^{-1};\lfloor-\frac{\val(a)}2\rfloor) & n=1,\ m=0 ,\ \abs{b_{m+2n}}<\abs{a_{n+2m}}\le q^{-3}  \\
q^4\abs{b}^{-1}\Cu(-3b^{-1},a^{-1}b^{-1};m-2)& n=0,\ m\ge 2 ,\ \abs{b_{m+2n}}\le \abs{a_{n+2m}}= q^{-4}  \\
q^2\abs{ab}^{-1}\sum_{\ell=0}^2 \Cu(-3(b^{-1}+\rho^\ell a^{-1}),a^{-1}b^{-1};\lfloor-\frac{\val(a)}2\rfloor) & n=0,\ m=1 ,\ \abs{b_{m+2n}}\le \abs{a_{n+2m}}\le q^{-4}  \\
0& \text{otherwise.}
\end{cases}
\]
\end{proposition}

It is an elementary (albeit lengthy) check that Proposition~\ref{prop j} 
follows directly from Proposition~\ref{prop jl} and Corollary~\ref{cor exp j}.  We do one case as an example and leave the remaining cases
to the reader.

Consider the case $n=0$, $m\geq2$, $|b_{m}|\leq |a_{2m}|=q^{-4}$.   Then by Proposition~\ref{prop jl}, $J_{m,n}(a,b)$ is expressed as a weighted sum of the terms 
$J^l(a,b;i,j)$ where 
$$(i,j)\in\{(2m,m),(2m-2,m-2),(2m-1,m),(2m,m-1),(2m-2,m-1),(2m-1,m-2)\}.$$  
The pairs $(2m-2,m-2)$, $(2m,m-1)$ and $(2m-1,m-2)$ are not in $\Lambda'$
and so do not contribute.  Also, $J^l(a,b;2m,m)$ and $J^l(a,b;2m-1,m)$ are zero by Proposition~\ref{prop jl}  (for the latter, note that $|a_{2m-1}|=q^{-3}$
and when we apply the proposition with $i=2m-1$, $j=m$, there is no contribution).  Finally, to evaluation $J^l(a,b;2m-2,m-1)$ we observe that $|a_{2m-2}|=q^{-2}$,
$|b_{m-1}|<|a_{2m-2}|$, and we apply Proposition~\ref{prop jl} with $i=2m-2$, $j=m-1$ so $i=2j>0$: 
$$(b,a)J^l(a,b;2m-2,m-1)=q^2|b|^{-1}\Cu(-3b^{-1},a^{-1}b^{-1};m-2).$$
Multiplying by the factor $q^2$ that appears in Corollary~\ref{cor exp j}, we arrive at the expression presented in Proposition~\ref{prop j}.

\section{Formulas for the small cell orbital integrals for Hecke functions on the cubic cover of $\SL_3(F)$}\label{sec J comp - small}

In this Section we evaluate the orbital integrals for the remaining relevant cells.   We treat
the three isolated orbits
\[
\gamma[\ell]=\sec(\rho^\ell I_3),\ \ \ \ell=1,2,3
\]
and the two one parameter families $\Xi_i'=\{\gamma_i(a): a\in F^*\}$, $i=1,2$ where 
\[
\gamma_1(a)=\sec\begin{pmatrix} &   a^{-2}  \\ aI_2 &    \end{pmatrix}\ \ \ \text{and}\ \ \ \gamma_2(a)=\sec\begin{pmatrix}  & aI_2  \\ a^{-2}   & \end{pmatrix}.
\]

\subsection{The isolated orbits}
Let $\gamma[\ell]$ be one of the three isolated orbits.  Then it is easy to see that
$$\O'(\gamma[\ell],L(\sec(c_{i,j})^{-1})f_\circ')=\int_N f'_\circ (\sec(c_{i,j})\sec(\rho^\ell I_3)u) \theta(u)\,du
=\delta_{(i,j),(0,0)}.$$
Combining this with Corollary~\ref{cor exp j}, we obtain the following result.

\begin{proposition}\label{isolated-cells-J}
The orbital integral attached to the isolated orbit $\gamma[\ell]$ is given by
\begin{equation*}
\O'(\gamma[\ell],f_{m,n}')=\begin{cases} 1&m=n=0\\
q^2&m+n=1\\
0&otherwise.
\end{cases}
\end{equation*}
\end{proposition}

\subsection{The orbital integral for the one parameter families}

By Corollary~\ref{cor Jfe},
\begin{equation}\label{relate the two one parameter orbital integrals for G'}
\O'(\gamma_2(a);f_{m,n}')=\overline{\O'(\gamma_1(a^{-1}),f_{n,m}')}
\end{equation}
 so it is sufficient to evaluate $\O'(\gamma_1(a),f_{n,m}')$, $a\in F^*$.
 
\subsubsection{Coordinate form of the integral $J^l(\gamma_1(a);i,j)$} 
Let $V_1=\{u\in N\mid u_{2,3}=0\}$.  Let $$g_a=\begin{pmatrix} &   a^{-2}  \\ aI_2 \end{pmatrix}.$$
so that $\gamma_1(a)=\sec(g_a)$.  
Then by definition 
$$J^l(\gamma_1(a);i,j)= \delta_{B'}(\sec(c_{i,j}))\,\O'(\gamma_1(a),L(\sec(c_{i,j})^{-1})f_\circ')$$
with
$$\O'(\gamma_1(a),L(\sec(c_{i,j})^{-1})f_\circ')=\int_{V_1\times N} f'_\circ (\sec(c_{i,j})u_1^{-1}\sec(g_a)u_2) \theta(u_1^{-1}u_2)\,du_1\,du_2.$$
Our first task is to put this in a form that is amenable to calculation.

Changing  $u_1\mapsto \sec(c_{i,j})^{-1}u_1^{-1}\sec(c_{i,j})$ gives
$$J^l(\gamma_1(a);i,j)= q^{2i-j}\, \int_{V_1\times N} f'_\circ (u_1\sec(c_{i,j})\sec(g_{a})u_2) \theta(c_{i,j}^{-1}u_1c_{i,j})\theta(u_2)\ du_1\ du_2.$$
Let 
$$s_2=\begin{pmatrix}1&&\\&&1\\&-1&\end{pmatrix}.$$ 
Then $s_2$ acts by conjugation on $V_1$ and $\sec(s_2)\in K'$.  Upon changing $u_1$ to $u_1^{s_2}:=s_2u_1s_2^{-1}$ and using the left-invariance of $f'_\circ$ under $K'$, we obtain
$$J^l(\gamma_1(a);i,j)= q^{2i-j}\, \int_{V_1\times N} f'_\circ (u_1\sec(s_2)^{-1}\sec(c_{i,j})\sec(g_{a})u_2) \theta(c_{i,j}^{-1}u_1^{s_2}c_{i,j})\theta(u_2)\ du_1\ du_2.$$

Multiplying in the metaplectic group, we have:
$$\sec(s_2)^{-1}\sec(c_{i,j})\sec(g_{a})=(\varpi,a)^{i+j}\sec(g_{\varpi^{j-i}a,\varpi^{j}a^2}).$$
(The matrix $g_{a,b}$ is defined in Section~\ref{ss orbreps}.)
Notice that introducing $s_2$ had taken us to the big cell.

For $a,b\in F^*$, let
\[
A'[a,b]=\{(u_1,u_2)\in V_1\times N: u_1g_{a,b}u_2\in K_1\}.
\]
Then since $f'_\circ$ is anti-genuine, we arrive at the following result.
\begin{lemma}\label{lem small orbit int for J}
For $a\in F^*$ and $(i,j)\in\Lambda'$ we have
\[
J^l(\gamma_1(a);i,j)=(\varpi,a)^{-i-j}q^{2i-j}\int_{A'[a_{j-i},\varpi^j a^2]} \kappa(u_1g_{a_{j-i},\varpi^{j}a^2}u_2)\theta(c_{i,j}^{-1}u_1^{s_2}c_{i,j})\theta(u_2)\ du_1\ du_2.
\]
\end{lemma}

Observe that if
$$u_1=\begin{pmatrix}1&x_1&z_1\\&1&\\&&1\end{pmatrix}$$
then
$$\theta(c_{i,j}^{-1}u_1^{s_2}c_{i,j})=\psi(\varpi^{2j-i}z_1).$$

For convenience and to avoid conflict with prior notation, we now set
$$\a=\varpi^{j-i}a\qquad \b=\varpi^{j}a^2.$$
Then $(\b,\a)=(\varpi,a)^{-i-j}$.  In our case, we have
\[
u_1g_{\a,\b}u_2=\begin{pmatrix} \a z_1 & \a z_1x_2-\a^{-1}\b x_1 & \a z_1z_2 -\a^{-1}\b x_1y_2+\b^{-1}  \\ 0 &  -\a^{-1}\b & -\a^{-1}\b y_2 \\ \a & \a x_2 & \a z_2\end{pmatrix}.
\]
We see immediately that $A'[\a,\b]$ is nonempty only when $|\b|\leq |\a|\leq 1$; otherwise the integral $J^l(\gamma_1(a);i,j)$ is zero.  We note that this implies that
$$|a|\leq \min(q^{j-i},q^{i}).$$
We also note that
$|\a|=|\b|$ implies that $|a|=q^{i}$ and, since $(i,j)\in\Lambda'$, together with the last inequality this implies that $j=2i$.  But this implies that $|\a|=1$.
Similarly, $|\b|=q^{-1}|\a|$ implies that $|a|=q^{i-1}$ and either ($j=2i$ and $|\a|=q^{-1}$), or else ($j=2i-1$ and $|\a|=1$).

When $|\b|\leq |\a|\leq 1$, we now compute the orbital integrals $J^l(\gamma_1(a);i,j)$.

\subsubsection{Evaluation of $J^l$ when $|\a|=1$}  If $|\a|=1$ we have the following evaluation.
\begin{lemma}\label{small cell a=1}
Suppose that $|\a|=1$.  Then
$$J^l(\gamma_1(a);i,j)=\begin{cases}1&|\b|=1\\
q \g&|\b|=q^{-1}\\
0&\text{otherwise.}\end{cases}$$
\end{lemma}

\begin{proof} If $|\a|=|\b|=1$ then we must have $x_1,z_1,x_2,y_2,z_2\in\O$ and the integral is $1$.
If $|\b|<|\a|=1$, then applying \cite[Lemma 8.7(2)]{FO} to evaluate $\kappa(u_1g_{\a,\b}u_2)$ gives
\[
J^l(\gamma_1(a);i,j)=q^{2i-j}(\a,\b)\int (\b \a^{-1},y_2)
\psi(\varpi^{2j-i}z_1+x_2+y_2)\ d(x_1,z_1,x_2,y_2,z_2),
\]
where the integral is over the domain
$$x_2,z_1,z_2, \b x_1,\b y_2,\b^{-1}(1-\a^{-1}\b^2x_1y_2)\in \O.$$
We integrate over $x_1,x_2,z_1,z_2$ as in the proof of Proposition~\ref{lem ja=1} to arrive at
$$q^{2i-j}(\a,\b)\int_{|y_2|=|\b^{-1}|} (\b \a^{-1},y_2)\psi(y_2)\,dy_2.$$
This integral is zero if $|\b|\leq q^{-2}$ and equals the Gauss sum $\g_{\b\a^{-1}}=(\varpi,a)\g$ if $|\b|=q^{-1}$.  Since as noted above the conditions $|\a|=1$, $|\b|=q^{-1}$ imply that
$2i-j=1$, the result follows.
\end{proof}

\subsubsection{Evaluation of $J^l$ when $|\a|<1$}  When $|\a|<1$, we follow the approach for the big cell but the situation is much easier.  We will establish the following result.
\begin{lemma}\label{small cell a<1}
Suppose that $|\a|<1$.  Then
$$J^l(\gamma_1(a);i,j)=\begin{cases}
q^2&|\a|=q^{-1}, j=2i>0\\
\abs{a}^{-2}\sum_{k=0}^2 \psi(-3\rho^ka^{-1})&|\a|\leq q^{-1}, i=j=0\\
0&\text{otherwise.}\end{cases}$$
\end{lemma}

\begin{proof}
Since $|\a|<1$, the integral is zero unless $|\b|<|\a|$ as explained above. 
Thus $A'[\a,\b]$ is nonempty only when $|\a x_2|=1$.
Applying \cite[Lemma 8.7(3)(b)]{FO} to evaluate $\kappa(u_1g_{\a,\b}u_2)$, we have
\[
J^l(\gamma_1(a);i,j)=q^{2i-j}(\a,\b)\int (\b,x_2)(z_2x_2^{-1}-y_2,\b^{-1}\a x_2)\psi(\varpi^{2j-i}z_1+x_2+y_2)\ d(x_1,z_1,x_2,y_2,z_2) 
\]
where the integral is over
\[
\a x_2\in \O^*,\ \a z_1,\,\a z_2,\, \a^{-1}\b y_2 ,\,\a z_1x_2-\a^{-1}\b x_1,\,  \a z_1z_2-\a^{-1}\b x_1y_2+\b^{-1}\in \O,
\]
or equivalently
\[
\a x_2,\,\b x_1\in \O^*, \,\a z_2\in\O,\,z_1\in \a^{-2}\b x_1x_2^{-1}+\O, y_2\in \a\b^{-2}x_1^{-1}+\a^2\b^{-1}x_1^{-1}z_1z_2+\a\O.
\]
As in the proof of Lemma~\ref{lem j2}, on this domain the Hilbert symbol inside the integrand simplifies to $$(\a,\b)(x_2,x_1)(\b^{-1}\a,x_1x_2).$$
After integrating over $y_2$, $z_1$ and $z_2$, we obtain
\[
q^{2i-j}(\b,\a)\int_{\b^{-1}\O^*} \int_{\a^{-1}\O^*}(x_2,x_1)(\b^{-1}\a,x_1x_2)\psi(\varpi^{2j-i}\a^{-2}\b x_1x_2^{-1}+x_2+\a\b^{-2}x_1^{-1})\ dx_2\ dx_1.
\]
Then using the variable change $x_2\mapsto \a^{-1}x_2$, $x_1\mapsto \b^{-1}x_1$ we obtain
\[
q^{2i-j}(\a,\b)|\a\b|^{-1}\int_{\O^*} \int_{\O^*}(\b^{-1},x_1)(\a,x_2)\psi(\varpi^{2j-i}\a^{-1} x_1x_2^{-1}+\a^{-1}x_2+\a\b^{-1}x_1^{-1})\ dx_2\ dx_1.
\]

Suppose $|\a|<q^{-1}$.  Then the $x_2$ integral vanishes unless $i=2j$.  But in this case the $x_1$ integral vanishes unless $|\a|^2=|\b|$. This is the case only when $i=j=0$.  
In that case the integral was evaluated in \cite{FO}, Lemma 11.4.   The value of the orbital integral in this case is
$\abs{a}^{-2}\sum_{k=0}^2 \psi(-3\rho^ka^{-1})$ (note $|a|=|\a|$ as $i=j=0$).  We point out that this is the conjugate of the value obtained in \cite{FO}, Lemma 11.4,
since we work with the anti-genuine Hecke algebra in this paper rather than the genuine Hecke algebra used in \cite{FO}. 

This leaves the case $|\a|=q^{-1}$.  Then the  $x_1$ integral vanishes unless $|\b|=q^{-2}$, which implies that $j=2i$. Suppose that $|\b|=q^{-2}$.  If $i>0$, then $\psi(\varpi^{2j-i}\a^{-1} x_1x_2^{-1})=1$ and the integral
gives a product of two conjugate Gauss sums, and the total contribution is $q^{2}$.  If $i=j=0$ we get (again by \cite{FO}, Lemma 11.4) the same contribution as above.
\end{proof}

\subsubsection{Completion of the evaluation} Rewriting the evaluation of $J^l$ in terms of $|a_i|$ instead of $|\a|$, we obtain the following result.

\begin{proposition}\label{one parameter J^l}
Suppose $a\in F^*$ and $(i,j)\in \Lambda'$.  Then
$$J^l(\gamma_1(a);i,j)=\begin{cases}1&|a_i|=1, j=2i\\
q^2&|a_i|=q^{-1}, i>0,  j=2i\\
q \g&|a_i|=q^{-1}, i>0,  j=2i-1\\
\abs{a}^{-2}\sum_{k=0}^2 \psi(-3\rho^ka^{-1})&|a|<1, i=j=0\\
0&\text{otherwise.}\end{cases}$$
\end{proposition}  

We now substitute this evaluation into Corollary~\ref{cor exp j} to arrive at the evaluation of $\O'(\gamma_1(a);f'_{m,n})$.

\begin{proposition}\label{formula for one parameter J orbital integral}
Let $m,n\geq0$ with $m+n\ge1$. 
Suppose $a\in F^*$.   Then
$$\O'(\gamma_1(a),f'_{m,n})=\begin{cases}1&|a_{2m+n}|=1, m=0\\
3q^2&|a_{2m+n}|=q^{-1}, m=0\\
q^2\abs{a}^{-2}\sum_{k=0}^2 \psi(-3\rho^ka^{-1})&|a_{2m+n}|=q^{-2}, m=0, n=1\\
q^4&|a_{2m+n}|=q^{-2}, m=0, n>1\\
2q^2&|a_{2m+n}|=q^{-2}, m=1\\
2q^4&|a_{2m+n}|= q^{-3}, m=1, n>0\\
q^2\abs{a}^{-2}\sum_{k=0}^2 \psi(-3\rho^ka^{-1})&|a_{2m+n}|\leq q^{-3}, m+n=1\\
0&\text{otherwise.}\end{cases}$$
\end{proposition}

\section{Comparison of the orbital integrals}\label{sec compare}

Finally, we prove Theorem~\ref{thm main local}.  Each matching for all $f\in \H$ follows by establishing the matching on the basis $\{f_{m,n}\}$ for $\H$.
In view of \eqref{Satake relation}, we must compare the relative orbital integrals for $f_{m,n}$ to the corresponding integrals on the cubic cover for $f'_{m,n}$.
The comparison when $m=n=0$ was established in \cite{FO}, Theorems 6.1 and 11.7.  Note that because on the metaplectic side we worked with genuine functions
there and we use anti-genuine functions here, the transfer factors in Theorem~\ref{thm main local} are the conjugates of the ones given in \cite{FO}, and in passing from $a$ to $c_i$ 
in Theorem~\ref{thm main local}, part~\ref{FL-part2} we must change the map of \cite{FO}, Theorem 11.7, parts (2) and (3), by a minus sign. 

We suppose that $m+n>0$.   For the generic families, the two orbital integrals are evaluated
in Proposition~\ref{Final-form-I-OI} and \ref{prop j}.   The desired matching is immediate from these evaluations after using the substitution shown.
For the one-parameter families $\xi_1(a)$ and $\gamma_1(a)$, the orbital integrals $\O(\xi_1(a),\omega(f_{m,n})\phi_\circ)$ and $\O'(\gamma_1(a),f'_{m,n})$ are evaluated in Propositions~\ref{formula for one parameter I orbital integral} 
and \ref{formula for one parameter J orbital integral}.  Once again the matching of Theorem~\ref{thm main local} follows directly from these evaluations.
For the one-parameter families $\xi_2(a)$ and $\gamma_2(a)$, we make use of the comparison
\eqref{degenerate orbit I symmetry} (resp.\ 
\eqref{relate the two one parameter orbital integrals for G'}) that relates the orbital integral for $\xi_2(a)$ (resp.\ $\gamma_2(a)$)  to the orbital integral for 
$\xi_1(a)$ (resp.\ $\gamma_1(a)$), for which the matching has now been established.  
We also use Corollary~\ref{change rho to rho squared in OO} to keep track of the change of $\rho$ to $\rho^2$ that appears in \eqref{degenerate orbit I symmetry}.
These give
\begin{align*}
\O_\rho(\xi_2(a),\omega(f_{m,n})\phi_\circ)&=\overline{\O_{\rho^2}(\xi_1(a),\omega(f_{n,m})\phi_\circ)}\\
&=\overline{\O_{\rho}(\xi_1(-\rho a),\omega(f_{n,m})\phi_\circ)}\\
&=\overline{t_1(c_1')}~\overline{\O(\gamma_1(c_1'),\Sh(f_{n,m}))}\\
&=\overline{t_1(c_1')}~\O'(\gamma_2((c_1')^{-1}),\Sh(f_{m,n})).
\end{align*}
Here $c_1'=3(1-\rho)a^{-1}$ is obtained from the $c_1$ given in Theorem~\ref{thm main local}, part~\ref{FL-part2}, by replacing $a$ by $-\rho a$.
Setting $c_2=(c_1')^{-1}=[3(1-\rho)]^{-1}a$,
the comparison for these orbital integrals asserted in Theorem~\ref{thm main local} follows.
For the three isolated families, the orbital integrals are evaluated in Propositions~\ref{isolated-cells-I} and \ref{isolated-cells-J}.
In each case, the matching follows from these evaluations.

\begin{remark}\label{closing remark} For all $f\in\H$, $f'\in\H'$ it follows from Propositions~\ref{Final-form-I-OI} and \ref{prop j} that
\begin{align*}
\O(\xi(\rho a,\rho^2 b),\omega(f)\phi_\circ)&=\O(\xi(a,b),\omega(f)\phi_\circ)\\
\O'(\gamma(\rho c,\rho^2 d),f')&=\O'(\gamma(c,d),f').
\end{align*}
The map $(a,b)\to (c,d)$ given in Theorem~\ref{thm main local} could also be changed by incorporating this transformation. Similarly, the 
orbital integrals $\O'(\gamma_i(a),f')$, $i=1,2,$ are invariant under $a\mapsto \rho a$;
for $\O'(\gamma_1(a),f')$ this is immediate from Proposition~\ref{formula for one parameter J orbital integral}, and upon using \eqref{relate the two one parameter orbital integrals for G'}, the same holds for 
$\O'(\gamma_2(a),f')$.

The orbital integral for $\xi_1(a)$ and the unit element $f_{0,0}$ of $\H$ may be matched to the corresponding integral for either $\gamma_1(c)$ or $\gamma_2(c)$ (for suitable $c$).
See the paragraph following Theorem 11.7 in \cite{FO}.
However, for general $f\in\H$ we see by comparing
\eqref{degenerate orbit I symmetry} and \eqref{relate the two one parameter orbital integrals for G'} that in Theorem~\ref{thm main local} 
we must match the orbital integral for the orbit $\xi_1(a)$ to the orbital integral
for the orbit $\gamma_1(c)$ in order to be compatible with the isomorphism $\Sh: \H\rightarrow \H'$ .
\end{remark}

\def\cprime{$'$} \def\Dbar{\leavevmode\lower.6ex\hbox to 0pt{\hskip-.23ex
  \accent"16\hss}D} \def\cftil#1{\ifmmode\setbox7\hbox{$\accent"5E#1$}\else
  \setbox7\hbox{\accent"5E#1}\penalty 10000\relax\fi\raise 1\ht7
  \hbox{\lower1.15ex\hbox to 1\wd7{\hss\accent"7E\hss}}\penalty 10000
  \hskip-1\wd7\penalty 10000\box7}
  \def\polhk#1{\setbox0=\hbox{#1}{\ooalign{\hidewidth
  \lower1.5ex\hbox{`}\hidewidth\crcr\unhbox0}}} \def\dbar{\leavevmode\hbox to
  0pt{\hskip.2ex \accent"16\hss}d}
  \def\cfac#1{\ifmmode\setbox7\hbox{$\accent"5E#1$}\else
  \setbox7\hbox{\accent"5E#1}\penalty 10000\relax\fi\raise 1\ht7
  \hbox{\lower1.15ex\hbox to 1\wd7{\hss\accent"13\hss}}\penalty 10000
  \hskip-1\wd7\penalty 10000\box7}
  \def\ocirc#1{\ifmmode\setbox0=\hbox{$#1$}\dimen0=\ht0 \advance\dimen0
  by1pt\rlap{\hbox to\wd0{\hss\raise\dimen0
  \hbox{\hskip.2em$\scriptscriptstyle\circ$}\hss}}#1\else {\accent"17 #1}\fi}
  \def\bud{$''$} \def\cfudot#1{\ifmmode\setbox7\hbox{$\accent"5E#1$}\else
  \setbox7\hbox{\accent"5E#1}\penalty 10000\relax\fi\raise 1\ht7
  \hbox{\raise.1ex\hbox to 1\wd7{\hss.\hss}}\penalty 10000 \hskip-1\wd7\penalty
  10000\box7} \def\lfhook#1{\setbox0=\hbox{#1}{\ooalign{\hidewidth
  \lower1.5ex\hbox{'}\hidewidth\crcr\unhbox0}}}
\providecommand{\bysame}{\leavevmode\hbox to3em{\hrulefill}\thinspace}
\providecommand{\MR}{\relax\ifhmode\unskip\space\fi MR }
\providecommand{\MRhref}[2]{
  \href{http://www.ams.org/mathscinet-getitem?mr=#1}{#2}
}
\providecommand{\href}[2]{#2}

\end{document}